\documentclass{amsart}
\usepackage{amsfonts,amssymb,amsmath,amsthm}
\usepackage{url}
\usepackage{enumerate}
\usepackage{xypic}

\newtheorem{thrm}{Theorem}[section]
\newtheorem{lem}[thrm]{Lemma}
\newtheorem{prop}[thrm]{Proposition}
\newtheorem{cor}[thrm]{Corollary}
\theoremstyle{definition}
\newtheorem{definition}[thrm]{Definition}
\newtheorem{notion}[thrm]{}
\newtheorem{example}[thrm]{Example}
\newtheorem{remark}[thrm]{Remark}
\numberwithin{equation}{section}

\author{Qingnan An, Zhichao Liu and Yuanhang Zhang}

\begin{document}
\title[Classification of $\mathrm{C}^*$-algebras]{On the classification of certain real rank zero $\mathrm{C}^*$-algebras}

\begin{abstract}
In this paper, a classification is given of real rank zero $\mathrm{C}^*$-algebras that can be expressed as inductive limits of a sequence of a subclass of Elliott-Thomsen algebras $\mathcal{C}$.
\end{abstract}
\maketitle

\section{Introduction}

A $\mathrm{C}^*$-algebra $A$ is called $AH$ algebras if it is an inductive limit of the $\mathrm{C}^*$-algebras of form $A_n=P_nM_{n}(C(X_n))P_n$, where $X_n$ are compact metric spaces and $P_n \in M_{n}(C(X_n))$ are projections.
In particular, if $X_n$ can be chosen to be disjoint union of several circles (or intervals, respectively),
then $A$ is called an $AT$ algebra (or $AI$ algebras, respectively). More general, if $A_n \subset M_{n}(C(X_n))$,
then $A$ is called an $ASH$ algebra. It seems $AH$ algebras and $ASH$ algebras are quite special class of
$\mathrm{C}^*$-algebras, but these classes of $\mathrm{C}^*$-algebras play important roles in the Elliott classification program. It is a
conjecture that all simple separable stable finite nuclear $\mathrm{C}^*$-algebras are $ASH$ algebras. Recent work
of Gong-Lin-Niu \cite{GLN:2015} and Elliott-Gong-Lin-Niu \cite{EGLN:2015} have verified the conjecture for all simple separable $\mathrm{C}^*$-algebras of finite decomposition rank with UCT. Combine with the work of Tikusis-White-Winter \cite{TWW:2017},
one knows that the conjecture is also true for all simple separable stably finite $\mathrm{C}^*$-algebras of finite
nuclear dimension. For purely infinite case, Kirchberg and Phillips classify all purely infinite simple separable amenable ${\mathrm C}^*$-algebras  which satisfy the UCT (see \cite{R:2002} for an overview).
In 1989, G. Elliott initiated a program aimed at the classification of all separable, nuclear
$\mathrm{C}^*$-algebras. At 1989, Elliott classified all real rank zero $AT$ algebras and made the following
conjecture.

\underline{Elliott Classification Conjecture 1}: $(\mathrm{K}_*(A), \mathrm{K}_*(A)_+, \Sigma A)$ is the complete invariant for separable nuclear $\mathrm{C}^*$-algebras of real rank zero and stable rank one.

In 1993, Elliott classified all simple algebras and made the following conjecture.

\underline{Elliott Classification Conjecture 2}: $Ell(A)$ is the complete invariant for simple separable nuclear $\mathrm{C}^*$-algebras.

For recent 25 years, the classification of simple nuclear $\mathrm{C}^*$-algebras (the verification of the above
Elliott Classification Conjecture 2) has tremendous advance.

On the other hand, the classification of non-simple $\mathrm{C}^*$-algebras is far from being satisfactory,
even for $\mathrm{C}^*$-algebras of real rank zero and stable rank one.
Also in 1994, Gong disproved the Elliott Classification Conjecture 1 for $\mathrm{C}^*$-algebras of real
rank zero and stable rank one by give an example of real rank zero $AH$ algebras of local dimension
at most two. This implies that for the classification of non simple $\mathrm{C}^*$-algebras of real rank zero,
one needs new invariant. The new invariant is developed in \cite{DG:1997}, \cite{DL:1996}, \cite{DL2:1996}, \cite{Ei:1996}, \cite{EGS:1998}, \cite{G:1998},  \cite{Lin:2010} and \cite{Lin:2017}, and is called
total ordered K-theory $(\underline{\mathrm{K}}(A), \underline{\mathrm{K}}^+(A), \Sigma A)$. Dadarlat-Gong proved the classification of real rank
zero $AH$ algebras of slow dimension growth by using total K-theory in 1997. So it is clear Elliott
classification conjecture 1 (for real rank zero case) should be modified to the following conjecture,
we still call it Elliott Conjecture:

\underline{Elliott Classification Conjecture (for real rank zero $\mathrm{ C}^*$-algebras)}: $(\underline{\mathrm{K}}(A), \underline{\mathrm{K}}^+(A),$ $\Sigma A)$ is the complete invariant for separable nuclear $\mathrm{C}^*$-algebras of real rank zero and stable rank one.

Unlike the case of classification for simple $\mathrm{C}^*$-algebras, there is no significant advance on the
classification $\mathrm{C}^*$-algebras of real rank zero and stable rank one after the Theorem of Dadarlat-Gong \cite{DG:1997},
even for the case of real rank zero $ASH$ algebras. This paper together with \cite{AELL:2019} is an effort to push forward such classification.

In this paper, we shall consider the real rank zero $\mathrm{C}^*$-algebras which can be expressed as inductive limits of subhomogeneous building blocks $\mathcal{D}$ defined in Definition \ref{DI}.
We mention that in \cite{AELL:2019}, the existence theorem is obtained by applying Jiang-Su's criterion that a $\mathrm{KK}$-element between two general dimension drop interval algebras can be lifted, if and only if
the $\mathrm{KK}$-element preserves the order of $\mathrm{K}$-homology (Theorem 3.7 in \cite{JS:1999}). However, for the case of $\mathcal{D}$, the same statement is not true (Example 4.7 in \cite{AE:2017}).
In order to get the existence result (Theorem \ref{composed existence}), we develop a distribution property of the connecting homomorphism, which comes from the real rank zero property of the inductive limit, to prove a decomposition result (Corollary \ref{deccor}) and pair it with
a $\mathrm{KK}$-element preserving the order of $\mathrm{K}$-homology. For the uniqueness theorem, we prove both of the results for the torsion $\mathrm{K}_1$-group case and the free case, then combine them with the existence result, we accomplish the classification.

The following is our main theorem:

 {\bf Theorem 8.3}\,\,Let $A=\underrightarrow{\lim}(A_n,\phi_{n,m})$ and $B=\underrightarrow{\lim}(B_n,\psi_{n,m})$ are real rank zero inductive limit of Elliott-Thomsen algebras in $\mathcal{D}$, then $A\cong B$ if and only if
 $$
 (\underline{\mathrm{K}}(A),\underline{\mathrm{K}}^+(A),[1_A])\cong(\underline{\mathrm{K}}(B),\underline{\mathrm{K}}^+(B),[1_B]).
 $$

It is not quite clear what is the range of invariant for the total K-theory for $\mathrm{C}^*$-algebras of real
rank zero and stable rank one, even for the class of rank zero $\mathrm{C}^*$-algebras classified by Dadarlat-Gong and in this paper. Therefore, it is not clear what is the relation between the class of our class
in the above theorem and the class of Dadarlat-Gong.


\section{Preliminaries}
\begin{definition}[Class $\mathcal{C}$]\label{ET}
Let $F_1$ and $F_2$ be two finite dimensional ${\mathrm C}^*$-algebras and let
$\varphi_0,\,\varphi_1:\,F_1\to F_2$ be two homomorphisms.
Set
\begin{align*}
A&=A(F_1,F_2,\varphi_0,\varphi_1)
\\
&=\{(f,a)\in  C([0,1],F_2) \oplus F_1:\,f(0)=\varphi_0(a)\,  {\rm and}\, f(1)=\varphi_1(a)\}.
\end{align*}
Denote by $\mathcal{C}$ the class of all such $\mathrm{C}^*$-algebras.
\end{definition}
The ${\mathrm C}^*$-algebras constructed in this way have been studied
by Elliott and Thomsen \cite{ET:1994} (see also \cite{Ell:1996} and \cite{Th:1996}), which are sometimes called Elliott-Thomsen algebras or one dimensional non-commutative finite CW complexes. Following \cite{GLN:2015}, let us say that a ${\mathrm C}^*$-algebra $A\in\mathcal{C}$ is $\mathbf{minimal}$,
or a minimal block, if it is
indecomposable, i.e.,
not the direct sum of two or more ${\mathrm C}^*$-algebras in $\mathcal{C}$.
\begin{definition}[Class $\mathcal{D}$]\label{DI}
Let $F_1$ be a finite dimensional ${\mathrm C}^*$-algebra, $\mathbf{F_2=M_n(\mathbb{\mathbf C})}$ {\bf (not a direct sum)} and let $\varphi_0,\,\varphi_1:\,F_1\to F_2$ be two unital homomorphisms.
Set
\begin{align*}
A&=A(F_1,M_n(\mathbb{C}),\varphi_0,\varphi_1) \\
&=\{(f,a)\in  C([0,1],M_n(\mathbb{C})) \oplus F_1:\,f(0)=\varphi_0(a)\,  {\rm and}\, f(1)=\varphi_1(a)\}.
\end{align*}
Denote by $\mathcal{D}$ the class  consists of all the finite direct sums of $A(F_1,M_n(\mathbb{C}),\varphi_0,\varphi_1)$ and finite dimensional algebras.
And we use
  $A\mathcal{D}$ to denote the inductive limits of algebras in $\mathcal{D}$.
\end{definition}
\begin{remark}
Note that $\mathcal{D}$ is a sub class of $\mathcal{C}$,
and in this paper, we consider $\mathcal{D}$ not $\mathcal{C}$.
This is because for $\mathcal{D}$, we have Theorem \ref{d_0 F lift}, which shows that for two minimal blocks in $\mathcal{D}$, a $\mathrm{KK}$-element between them preserving Dadarlat-Loring order also preserves the order of $\mathrm{K}$-homology.
But Example 5.8 in \cite{AE:2017} shows the same statement can be not true for two minimal blocks in $\mathcal{C}$.

Even though $\mathcal{D}$ is a sub class of $\mathcal{C}$, the algebra in $\mathcal{D}$ can have arbitrary finite generated $\mathrm{K}_1$ group.
\end{remark}

\begin{notion}[Section 3 of \cite{GLN:2015}]\rm
For $A=A(F_1,F_2,\varphi_0,\varphi_1)\in\mathcal{C}$ with $K_0(F_1)={\mathbb{Z}}^p$ and   $K_0(F_2)={\mathbb{Z}}^l$, consider the short exact sequence
$$
0\rightarrow SF_2 \xrightarrow{\iota} A \xrightarrow{\pi} F_1\rightarrow0,
$$
where $SF_2=C_0(0,1)\otimes F_2$ is the suspension of $F_2$, $\iota$ is the embedding map, and $\pi(f,a)=a,\,(f,a)\in A$ .
Then one has the six-term exact sequence
$$
0\to \mathrm{K}_0(A)\xrightarrow{\pi_*} \mathrm{K}_0(F_1)\xrightarrow{\partial} \mathrm{K}_0(F_2)\xrightarrow{\iota_*} \mathrm{K}_1(A)\to 0,
$$
where $\partial=\alpha-\beta$, and $\alpha$, $\beta$ are the matrices (with entries $\alpha_{ij},\beta_{ij}\in\mathbb{N},i=1,\cdots,l,j=1,\cdots,p$) correspond the maps $\mathrm{K}_0(\varphi_0),\,\mathrm{K}_0(\varphi_1):K_0(F_1)\rightarrow K_0(F_2)$, respectively.
Hence,
$$
\mathrm{K}_0(A)=\ker(\alpha -\beta)\subset \mathbb{Z}^p,\quad
\mathrm{K}_1(A)=\mathbb{Z}^l /{\rm Im}(\alpha - \beta),
$$
and $$\mathrm{K}_0^+ (A)=\ker(\alpha - \beta)\cap \mathrm{K}_0^+ (F_1).$$

At the case $F_2=M_n(\mathbb{C})$, the matrices $\alpha,\beta$ are reduced to matrices with only one row, $\alpha=(\alpha_1,\cdots,\alpha_p)$, $\beta=(\beta_1,\cdots,\beta_p)$.
\end{notion}
\begin{notion}\label{notion of A B}
Throughout this paper, when talking about $\mathrm{KK}(A,B)$ with $A,\,B\in\mathcal{D}$,
we shall assume the notational convention
that
$$A=A(F_1,M_n(\mathbb{C}),\varphi_0,\varphi_1),\,B=B(F_1',M_{n'}(\mathbb{C}),\varphi_0 ',\varphi_1 ')$$
with $$F_1= \bigoplus_{i=1}^p M_{k_i}(\mathbb{C}),\quad
F_1 '= \bigoplus_{i'=1}^{p'} M_{{k'}_{i'}}(\mathbb{C}).$$
And we use $\alpha,\,\beta,\,\alpha'$ and $\beta'$ to denote the matrices induced by
$\varphi_{0*},\,\varphi_{1*},\,\varphi_{0*}'$ and $\varphi_{1*}'$, respectively.
\end{notion}
\begin{notion}
  We use the notation $\#(\cdot)$ to denote the cardinal number of the set, if the argument is a finite set. Very often,
   the sets under consideration will be sets with multiplicity,
  and then we  shall also count multiplicity when we use the notation $\#$. We use $\bullet$ or $\bullet\bullet$ to denote any possible positive integer. We shall use $\{a^{\thicksim k}\}$ to denote $\{\underbrace{a,\cdots,a}_{k\ times}\}$.
\end{notion}
\begin{notion}
  Let us use $\theta_1,\theta_2,\cdots,\theta_p$ denote the spectrum of $F_1$ and denote the spectrum of $C([0,1],F_2)$ by $(t,i)$,
 where $0\leq t\leq1$ and $i\in \{1,2,\cdots,l\}$ indicates that it is in $i^{th}$ block of $F_2$. So
 $$
 Sp(C([0,1],F_2))=\coprod_{i=1}^{l}\{(t,i),\,0\leq t\leq1\}.
 $$
 Using identification of $f(0)=\varphi_0(a)$ and $f(1)=\varphi_1(a)$ for $(f,a)\in A,\,(0,i)\in Sp(C[0,1])$ is identified with
 $$
 (\theta_1^{\thicksim\alpha_{i1}},\theta_2^{\thicksim\alpha_{i2}},\cdots,\theta_p^{\thicksim\alpha_{ip}})\subset Sp(F_1)
 $$
 and $(1,i)\in Sp(C([0,1],F_2))$ is identified with
 $$
 (\theta_1^{\thicksim\beta_{i1}},\theta_2^{\thicksim\beta_{i2}},\cdots,\theta_p^{\thicksim\beta_{ip}})\subset Sp(F_1)
 $$
 as in $Sp(A)=Sp(F_1)\cup \coprod_{i=1}^{l}(0,1)_i$.
\end{notion}

\begin{notion}\label{SPCForm}\rm
Let $A\in \mathcal{C}$, $\phi:A\rightarrow M_n(\mathbb{C})$ be a homomorphism, then there exists a unitary $u$ such that
 $$
 \phi(f,a)=u^*\cdot{\rm diag}\big(\underbrace{a(\theta_1),
 \cdots,a(\theta_1)}_{t_1},\cdots,\underbrace{a(\theta_p),\cdots,a(\theta_p)}_{t_p},f(y_1),\cdots,f(y_{\bullet})\big)\cdot u,
 $$
 where $y_1,y_2,\cdots,y_{\bullet}\in \coprod_{i=1}^{l}[0,1]_i$. For $y=(0,i)$ (also denoted by $0_i$), one can replace $f(y)$ by
 $$\big(\underbrace{a(\theta_1),
 \cdots,a(\theta_1)}_{\alpha_{i1}},\cdots,\underbrace{a(\theta_p),\cdots,a(\theta_p)}_{\alpha_{ip}}\big)$$
 in the above expression, and do the same with $y=(1,i)$. After this procedure, we can assume each $y_k$ is strictly in the open interval $(0,1)_i$ for some $i$.
 We write the spectrum of $\phi$ by
 \begin{align}
 Sp\phi=\{\theta_1^{\thicksim t_1},\theta_2^{\thicksim t_2},\cdots,\theta_p^{\thicksim t_p},y_1,y_2,\cdots,y_{\bullet}\},
\label{spectform}
\end{align}
 where $y_k\in \coprod_{i=1}^{l}(0,1)_i$.

   Denote $\omega_i=\#(Sp\phi\cap(0,1)_i)$ the number of $y_k$'s which are in the $i^{th}$ open interval $(0,1)_i$ counting multiplicity. We will say that $\phi$ is of {\bf type}
  $(t_1,t_2,\cdots,t_p,w_1,\cdots,w_l)$.
\end{notion}
\begin{notion}
 For any self-adjoint element $f\in M_n(\mathbb{C})$, denote $Eig(f)$ the set of all the eigenvalues of $f$. If $Eig(f_1)$ and $Eig(f_2)$ can be paired to within $\varepsilon$ one by one, we denote it by $dist(Eig(f_1),Eig(f_2))<\varepsilon$.
 \end{notion}

\begin{notion} \label{homodeff}
  Let $A=A(F_1,M_n(\mathbb{C}),\varphi_0,\varphi_1)\in\mathcal{D}$, with $F_1=\bigoplus_{j=1}^p M_{k_j}(\mathbb{C}),$
  write $a\in F_1$ as $a=(a(\theta_1),a(\theta_2),\cdots,a(\theta_p))$, where $a(\theta_j)\in M_{k_j}(\mathbb{C})$. For $t\in [0,1]$, define $\pi_t:A\rightarrow F_2$ by $\pi_t((f,a))=f(t)$ for all $(f,a)\in A$. There is a canonical map $\pi_e:A\rightarrow F_1$ defined by $\pi_e((f,a))=a$ for all $(f,a)\in A$. For any $j=1,2,\cdots,p$, define $\pi_e^j:A\rightarrow M_{k_j}(\mathbb{C})$ by $\pi_e^j((f,a))=a(\theta_j)$ for any $(f,a)\in A$.

  Denote $\{e_{ss'}\}(1\leq s,s'\leq n)$ the set of matrix units for $M_n({\mathbb{C}})$ and $\{f_{ss'}^j\}(1\leq j\leq p,\,1\leq s,s'\leq k_j)$ the set of matrix units for $\bigoplus_{j=1}^p  M_{k_j}({\mathbb{C}})$.
  \end{notion}

\begin{notion}
 Let $A=A(F_1,M_n(\mathbb{C}),\varphi_0,\varphi_1)\in\mathcal{D}$, $\varphi_{0*},\varphi_{1*}$ be represented by matrices $\alpha=(\alpha_1,\cdots,\alpha_p)$ and $\beta=(\beta_1,\cdots,\beta_p)$. Then for each $\eta=\frac{1}{m}$ where $m\in\mathbb{N}_+$. Let $0=x_0<x_1<\cdots<x_{m}=1$ be a partition of $[0,1]$ into $m$ subintervals with equal length $\frac{1}{m}$. We will define a finite subset $H(\eta)\subset A_+$, consisting of two kinds of
 elements described below.

 (a). For each subset $X_j=\{\theta_j\}\subset Sp(F_1)=\{\theta_1,\theta_2,\cdots,\theta_p\}$ and integers $r,s$
 with $r+2\leq s$, denote $W_j\triangleq\{\cup_{\alpha_{j}\neq0}[0,r\eta]\}\cup\{\cup_{\beta_{j}\neq 0}[s\eta,1]\}$.
 Then we call $W_j$  the closed neighborhood of $X_j$, we define element $(f,a)\in A_+$ corresponding to $X_j$  and $W_j$  as follows:

 For each $t\in[0,1]$, define element $(f,a)\in A_+$ by
 $$
 a=(\underbrace{0,\cdots,0}_{j-1},f_{11}^j,\underbrace{0,\cdots,0}_{p-j})\in F_1
 $$
 and
 $$
 f(t)=
 \begin{cases}
  \varphi_0(a)\dfrac{\eta-dist(t,[0,r\eta])}{\eta}, & \mbox{if } 0\leq t\leq (a+1)\eta \\
  0, & \mbox{if } (r+1)\eta\leq t\leq (s-1)\eta \\
  \varphi_1(a)\dfrac{\eta-dist(t,[s\eta,1])}{\eta}, & \mbox{if } (s-1)\eta\leq t\leq 1
\end{cases}
$$
  All such elements $(f,a)\in A_+$ are included in the set $H(\eta)$ and are called {\bf test functions of type 1}.

(b). For each closed subset $X \subset [\eta,1-\eta]$, which is a finite union of
closed intervals $[x_r,x_{r+1}]$ and finite many points $\{x_r\}$.
Define $(f,a)$ corresponding to $X$ by $a=0$ and for $t\in(0,1)$ define
$$
f(t)=
\begin{cases}
  e_{11}\big(1-\dfrac{dist(t,X)}{\eta}\big )e_{11}, & \mbox{if } dist(t,X)<\eta \\
  0, & \mbox{if } dist(t,X)\geq \eta.
\end{cases}
$$
There are finite such elements in $A_+$. All such elements are called {\bf test functions of type 2}.
\end{notion}
\begin{notion}
 In general, a unital homomorphism $\phi: C[0,1]\rightarrow B$ with finite dimensional image is always of the form :
 $$
 \phi(f)=\sum f(x_k)p_k,\,\,\forall f\in C[0,1]
 $$
 where $\{x_k\}$ is a finite subset of [0,1] and $\{p_k\}$ is a set of mutually orthogonal projections with $\sum p_k=1_B$.
 A homomorphism $\phi: A=M_n(C([0,1]))\rightarrow B$ with finite dimensional image is of the form:
 $$
 \phi(f)=\sum f(x_k)\otimes p_k,\,\,\forall f\in A
 $$
 for a certain identification of $\phi(1_A)B\phi(1_A)\cong M_n(\mathbb{C})\otimes (\phi(e_{11})B\phi(e_{11}))$,
  where $\{p_k\}$ is a set of mutually orthogonal projections in $\phi(e_{11})B\phi(e_{11})$ (see \cite[Remark 1.6.4]{G2:2002}).
\end{notion}

\section{Weak variation }
  Inspired by \cite{EG:1996} and \cite{DG:1997}, we define the weak variation as follows:
\begin{definition}
 Let $A\in \mathcal{C}$, the weak variation of a finite set $F\subset A$ is defined by
 $$
 \omega(F)=\sup\limits_{\phi,\psi}\inf\limits_{u\in U(r)}\max\limits_{f\in F}\|u\phi(f)u^*-\psi(f)\|,
 $$
 where $\phi,\psi$ run through the homomorphisms from $A$ to $M_r(\mathbb{C})$ such that $KK(\phi)=KK(\psi)$.
 \end{definition}
\begin{definition}
  Suppose that $A$ is a $\mathrm{C}^*$-algebra, $B\subset A$ is a subalgebra, $F\subset A$ is a finite subset and let $\varepsilon>0$. If
 for each $f\in F$, there exists an element $g\in B$ such that $\|f-g\|<\varepsilon$, then we shall say that $F$ is approximately contained in $B$ to within $\varepsilon$, and denote this by $F\subset_{\varepsilon} B$.
\end{definition}
 The following lemma is clear by the standard techniques of spectral theory \cite{BBEK:1992}.
\begin{lem} \label{fini}
  Let $A=\underrightarrow{lim}(A_n,\phi_{n,m})$ be an inductive limit of $\mathrm{C}^*$-algebras $A_n$ with morphisms
  $\phi_{n,m}:A_n\rightarrow A_m$, then $RR(A)=0$ if and only if for any finite self-adjoint subset $F\subset A_n$ and $\varepsilon >0$, there exists an integer $m\geq n$ such
  that $$\phi_{n,m}(F)\subset_{\varepsilon}\{ f\in (A_m)_{sa}\,|\,f\,has\,finite\,spectrum\}.$$
\end{lem}
The following lemma is essentially contained in an old version of \cite{GLN:2015}.
\begin{lem} \label{kkeq}
  Let $A=A(F_1,F_2,\varphi_0,\varphi_1)\in \mathcal{C}$ and $\phi,\psi:A\rightarrow M_r(\mathbb{C})$ be two unital homomorphisms. Suppose that $\phi$ is of the
   type $(t_1,t_2,\cdots,t_p,w_1,w_2,\cdots,w_l)$, $\psi$ is of the type $(s_1,s_2,\cdots,s_p,w_1,w_2,\cdots,w_l)$, then the following conditions are equivalent.

   (a) $KK(\phi)=KK(\psi);$

   (b) there exists a
   $
   \left(
\begin{array}{c}
  c_1 \\
  c_2 \\
   \vdots\\
  c_l
\end{array}
  \right)
\in \mathbb{Z}^l
   $
   such that
   $$
  \left(
  \begin{array}{c}
    t_1 \\
    t_2 \\
    \vdots \\
    t_p
  \end{array}
  \right)-
  \left(
  \begin{array}{c}
    s_1 \\
    s_2 \\
    \vdots \\
    s_p
  \end{array}
  \right)
  =(\alpha^t-\beta^t)\cdot
  \left(
\begin{array}{c}
  c_1 \\
  c_2 \\
   \vdots\\
  c_l
\end{array}
  \right).
  $$
\end{lem}

\begin{proof}
  Obviously, any homomorphisms from $A$ to $M_{\bullet}(\mathbb{C})$ of the same type are homotopy, we may assume that $Sp\phi\cap(0,1)_i=Sp\psi\cap(0,1)_i$ counting multiplicity for all $i=1,2,\cdots,l$.

  Note that $A\subset C([0,1],F_2)\oplus F_1$, then $\phi,\psi$ naturally induce
  homomorphisms
  $$\widetilde{\phi},\widetilde{\psi}:C([0,1],F_2)\oplus F_1\rightarrow M_r(\mathbb{C})$$
   in an obviously way, with $\phi=\widetilde{\phi}\circ \iota$
  and $\psi=\widetilde{\psi}\circ \iota$, where $\iota:A\rightarrow C([0,1],F_2)\oplus F_1$ is the inclusion.
  Note that $Sp\widetilde{\phi}\cap (0,1)_i=Sp\widetilde{\psi}\cap (0,1)_i$ for all $i=1,2,\cdots,l$, then $\widetilde{\phi}|_{C([0,1],F_2)}$ is unitary equivalent to $\widetilde{\psi}|_{C([0,1],F_2)}$.
  So $\widetilde{\phi}_*,\widetilde{\psi}_*:K_0(F_2)\oplus K_0(F_1)\rightarrow K_0(\mathbb{C})$ satisfy
   $\widetilde{\phi}_*|_{K_0(F_2)}=\widetilde{\psi}_*|_{K_0(F_2)}$. That is $\widetilde{\phi}_*-\widetilde{\psi}_*$ defines a map from
   $K_0(F_1)$ to $\mathbb{C}$. Let $d_j=t_j-s_j$, then
   $$
   \widetilde{\phi}_*-\widetilde{\psi}_*:K_0(F_1)=\mathbb{Z}^p\rightarrow K_0(\mathbb{C})=\mathbb{Z}
   $$
   is given by
   $$
   (\widetilde{\phi}_*-\widetilde{\psi}_*)
   \left(
\begin{array}{c}
  x_1 \\
  x_2 \\
   \vdots\\
  x_p
\end{array}
  \right)
  =d_1x_1+d_2x_2+\cdots+d_px_p.
   $$
   Denote $\widetilde{\phi}_*-\widetilde{\psi}_*$ by $D:\mathbb{Z}^p\rightarrow \mathbb{Z}$. At the same time $D$ can be considered as an
   element (still denote by $D$) $D\in KK(F_1,\mathbb{C})={\rm Hom}(K_0(F_1),K_0(\mathbb{C}))$. Since
   $A\xrightarrow{\pi} F_1$ factors through as
   $$
   A\xrightarrow{\iota}C([0,1],F_2)\oplus F_1\rightarrow F_1,
   $$
   and
   $$
   KK(\phi)- KK(\psi)=KK(\widetilde{\phi}\circ \iota)-KK(\widetilde{\psi}\circ \iota),
   $$
   we have
   $$
   KK(\phi)- KK(\psi)=[\pi]\times D.
   $$
   where $[\pi]\in KK(A,F_1)$ is induced by  $A\xrightarrow{\pi} F_1$.

   Let $I=C_0((0,1),F_2)\subset A$ be the ideal. The short exact sequence
   $$
   0\rightarrow I\xrightarrow{\iota} A\xrightarrow{\pi}F_1\rightarrow 0
   $$
   induces two exact sequences
   $$
   0=K_0(I)\rightarrow K_0(A)\xrightarrow{\pi_*} K_0(F_1)\xrightarrow{\partial}K_1(I)\rightarrow K_1(A)\rightarrow K_1(F_1)=0
   $$
   and
   $$
   0\leftarrow KK(A,\mathbb{C})\xleftarrow{[\pi]\times} KK(F_1,\mathbb{C})\xleftarrow{\delta}KK^1(I,\mathbb{C})\leftarrow KK^1(A,\mathbb{C})\leftarrow 0.
   $$
   In particular $K_1(I)=K_0(F_2)$ and $KK^1(I,\mathbb{C})= KK(F_2,\mathbb{C})$, and the map $\delta$ is the dual map of $\partial$ as
   $$
   KK(F_2,\mathbb{C})={\rm Hom}(K_0(F_2),\mathbb{Z})
   \quad
   {\rm and}
   \quad
   KK(F_1,\mathbb{C})={\rm Hom}(K_0(F_1),\mathbb{Z}).
   $$
   Since $\partial:K_0(F_1)=\mathbb{Z}^p\rightarrow \mathbb{Z}^l=K_0(F_2)$ is given by the matrix $\alpha-\beta$.
   So $\delta:KK(F_2,\mathbb{C})\rightarrow KK(F_1,\mathbb{C})$ is given by $\alpha^t-\beta^t$ ($\alpha^t$ denotes the
   transpose of $\alpha$).

   Then
   $$
   KK(\phi)- KK(\psi)=[\pi]\times D=0
   $$
    if and only if $D\in Im(\delta)\subset KK(F_1,\mathbb{C})=\mathbb{Z}^p$, that is, there is
   a
   $
   \left(
\begin{array}{c}
  c_1 \\
  c_2 \\
   \vdots\\
  c_l
\end{array}
  \right)
\in \mathbb{Z}^l
   $
   such that
   $$
  \left(
  \begin{array}{c}
    d_1 \\
    d_2 \\
    \vdots \\
    d_p
  \end{array}
  \right)
  =(\alpha^t-\beta^t)\cdot
  \left(
\begin{array}{c}
  c_1 \\
  c_2 \\
   \vdots\\
  c_l
\end{array}
  \right).
  $$
\end{proof}
\begin{notion}
  Let $A=A(F_1,F_2,\varphi_0,\varphi_1)\in\mathcal{D}$, where $F_1=\bigoplus_{j=1}^pM_{k_j}(\mathbb{C}),\,\,F_2=M_n(\mathbb{C})$ and $\varphi_{0*},\varphi_{1*}$ be represented by $\alpha=(\alpha_1,\alpha_2,\cdots,\alpha_p)$ and $\beta=(\beta_1,\beta_2,\cdots,\beta_p)$.

  Define $N_A$ as follows:
  $$
  N_A={\rm min}\{\,l\,|\,l\geq \max_{\alpha_j\neq \beta_j}\{|\frac{\alpha_j+\beta_j}{\alpha_j-\beta_j}|\},j\in \{1,2,\cdots,p\},l\in \mathbb{N}\}+1.$$
\end{notion}
\begin{lem} \label{kktoaue}
  Let $A\in \mathcal{D}$, $F\subset A$ be a finite set and $1>\varepsilon>0$. Suppose that $\eta=\frac{1}{m}$ satisfies that for any $x,y\in [0,1]$ with $dist(x,y)<4N_A\eta$, $\|f(x)-f(y)\|<\frac{\varepsilon}{4}$ for any $f\in F$. If homomorphisms $\phi,\psi:A\rightarrow M_r(\mathbb{C})$ satisfy the following conditions:

  $(a)$ $dist(Eig(\phi(h)),Eig(\psi(h)))<1$, $\forall\, h\in H(\eta);$

  $(b)$ $KK(\phi)=KK(\psi)\in KK(A,\mathbb{C});$

  $(c)$ $Sp\phi\cap(0,1)=Sp\psi\cap(0,1)$ counting multiplicity,

  then there exists a unitary $u\in M_r(\mathbb{C})$ such that
$$
  \|\phi(f)-u^*\psi(f)u\|<\varepsilon,\,\,\forall f\in F.
$$\end{lem}
\begin{proof}
  Write $Sp\phi,Sp\psi$ as
  $$
  Sp\phi=\{\theta_1^{\sim t_1},\cdots,\theta_p^{\sim t_p},y_1,y_2,\cdots,y_{\bullet}\},\quad
  Sp\psi=\{\theta_1^{\sim s_1},\cdots,\theta_p^{\sim s_p},y_1,y_2,\cdots,y_{\bullet}\},
  $$
  where $\{y_1,y_2,\cdots,y_{\bullet}\}=Sp\phi\cap(0,1)=Sp\psi\cap(0,1)$ as condition (c) holds.

  From condition (b) and Lemma \ref{kkeq}, there exists an integer $c$ such that
  $$
  \left(
  \begin{array}{c}
    t_1 \\
    t_2 \\
    \vdots \\
    t_p
  \end{array}
  \right)-
  \left(
  \begin{array}{c}
    s_1 \\
    s_2 \\
    \vdots \\
    s_p
  \end{array}
  \right)
  =
  \left(
  \begin{array}{c}
    d_1 \\
    d_2 \\
    \vdots \\
    d_p
  \end{array}
  \right)
  =
  \left(
  \begin{array}{c}
    \alpha_1-\beta_1 \\
    \alpha_2-\beta_2 \\
    \vdots \\
    \alpha_p-\beta_p
  \end{array}
  \right)\cdot c
  $$

 If $\alpha=\beta$ or $c=0$,  then $(\alpha^t-\beta^t)\cdot c=0$, we  will have $t_j=s_j$ for $j=1,2,\cdots,p$, obviously,
 $\phi$ and $\psi$ are unitary equivalent.

 We may assume that $\alpha\neq\beta$ and $c\neq0 $, then
 there exists an integer $1\leq j_0\leq p$ such that $\alpha_{j_0}>\beta_{j_0}$, then $d_{j_0}\neq 0$.
 Set $X_{j_0}=\{\theta_{j_0}\}$, define
 $$
 W_{j_0}^r=
 \begin{cases}
   [0,r\eta], & \mbox{if } \beta_{j_0}=0 \\
   [0,r\eta]\cup[(r+2)\eta,1], & \mbox{if } \beta_{j_0}\neq 0,
 \end{cases}
 $$
 where $r\in\{0,1,\cdots,m-2\}$. Let $h_{j_0}^r$ be the test function corresponding to $X_{j_0}$ and $W_{j_0}^r$,
 then we have
 $$
 \#\{x=1\,|\,x\in Eig(\phi(h_{j_0}^r))\}=t_{j_0}+\alpha_{j_0}\cdot
 \#(Sp\phi\cap(0,r\eta])+\beta_{j_0}\cdot\#(Sp\phi\cap[r\eta+2\eta,1))
 $$
 and
 $$
 \#\{x=1\,|\,x\in Eig(\psi(h_{j_0}^r))\}=s_{j_0}+\alpha_{j_0}\cdot
 \#(Sp\phi\cap(0,r\eta])+\beta_{j_0}\cdot\#(Sp\phi\cap[r\eta+2\eta,1)).
 $$

 Since $dist(Eig(\phi(h_{j_0}^r)),Eig(\psi(h_{j_0}^r)))<1$, then we have
 $$
 \#\{x>0\,|\,x\in Eig(\psi(h_{j_0}^r))\}\geq
 \#\{x=1\,|\,x\in Eig(\phi(h_{j_0}^r))\}
 $$
 and
 $$
 \#\{x>0\,|\,x\in Eig(\phi(h_{j_0}^r))\}\geq
 \#\{x=1\,|\,x\in Eig(\psi(h_{j_0}^r))\}.
 $$
 Hence,
  $$
 \#\{x\neq0,1\,|\,x\in Eig(\psi(h_{j_0}^r))\}\geq
 \#\{x=1\,|\,x\in Eig(\phi(h_{j_0}^r))\}-\#\{x=1\,|\,x\in Eig(\psi(h_{j_0}^r))\}
 $$
 and
 $$
 \#\{x\neq0,1\,|\,x\in Eig(\phi(h_{j_0}^r))\}\geq
 \#\{x=1\,|\,x\in Eig(\psi(h_{j_0}^r))\}-\#\{x=1\,|\,x\in Eig(\phi(h_{j_0}^r))\}.
 $$
 Note that
 $$
 \#\{x\neq0,1\,|\,x\in Eig(\phi(h_{j_0}^r))\}=\#\{x\neq0,1\,|\,x\in Eig(\psi(h_{j_0}^r))\}.
 $$
 Therefore,
 $$
 \#\{x\neq0,1\,|\,x\in Eig(\phi(h_{j_0}^r))\}\geq|t_{j_0}-s_{j_0}|=|d_{j_0}|,
 $$
 this means that
 $$
 \alpha_{j_0}\cdot\#\big(Sp\phi\cap(r\eta,r\eta+\eta)\big)+
 \beta_{j_0}\cdot\#\big(Sp\phi\cap(r\eta+\eta,r\eta+2\eta)\big)\geq |d_{j_0}|.
 $$
 Then for each $r$, we have
 $$
 \#\big(Sp\phi\cap(r\eta,r\eta+2\eta)\big)\geq \frac{|d_{j_0}|}{\alpha_{j_0}+\beta_{j_0}}=|c|
 \cdot \frac{\alpha_{j_0}-\beta_{j_0}}{\alpha_{j_0}+\beta_{j_0}}\geq |c|\cdot \frac{1}{N_A},
 $$
 then there are at least $|c|$ points in  $[r\eta,(r+2N_A)\eta]$ for any $r=0,1,\cdots,m-2N_A$.

 Then we can define $\phi',\psi'$ as follows: If $c>0$, to define $\phi'$, we replace the $c$ largest of $y_k\in (0,1)$ by $1\thicksim \{\theta_1^{\thicksim\beta_1},\cdots,
 \theta_p^{\thicksim\beta_p}\}$, to define $\psi'$, we replace the $c$ smallest of $y_k\in (0,1)$ by $0\thicksim \{\theta_1^{\thicksim\alpha_1},\cdots,\theta_p^{\thicksim\alpha_p}\}$.
  Since
  $$
  \left(
  \begin{array}{c}
    t_1 \\
    t_2 \\
    \vdots \\
    t_p
  \end{array}
  \right)+\beta^t\cdot c=
  \left(
  \begin{array}{c}
    s_1 \\
    s_2 \\
    \vdots \\
    s_p
  \end{array}
  \right)
  +\alpha^t\cdot c,
  $$
 then we have $Sp\phi'\cap Sp(F_1)=Sp\psi'\cap Sp(F_1)$.

 If $c<0$, to define $\phi'$, we replace the $-c$ smallest of $y_k\in (0,1)$ by $0\thicksim \{\theta_1^{\thicksim\alpha_1},\cdots,
 \theta_p^{\thicksim\alpha_p}\}$, to define $\psi'$, we replace the $-c$ largest of $y_k\in (0,1)$ by $1\thicksim \{\theta_1^{\thicksim\beta_1},\cdots,\theta_p^{\thicksim\beta_p}\}$, in this case
  $$
  \left(
  \begin{array}{c}
    t_1 \\
    t_2 \\
    \vdots \\
    t_p
  \end{array}
  \right)+\alpha^t\cdot (-c)=
  \left(
  \begin{array}{c}
    s_1 \\
    s_2 \\
    \vdots \\
    s_p
  \end{array}
  \right)
  +\beta^t\cdot (-c),
  $$
  we still have $Sp\phi'\cap Sp(F_1)=Sp\psi'\cap Sp(F_1)$.

 Note that, the $|c|$ points change within 2$N_A$, the we have
 $$
 \|\phi(f)-\phi'(f)\|<\frac{\varepsilon}{4},\quad
 \|\psi(f)-\psi'(f)\|<\frac{\varepsilon}{4}, \quad\forall\, f\in F.
 $$

 If the set $Sp\phi\cap(0,1)=Sp\psi\cap(0,1)$ are written in increasing order as $y_1\leq y_2\leq\cdots\leq y_{\bullet}$, then
  $|y_k-y_{k+|c|}|<4N_A\eta$. As sets $Sp\phi'\cap(0,1)$ and $Sp\psi'\cap(0,1)$ can be paired to within $4N_A\eta$, therefore, there exists a unitary $u\in M_r(\mathbb{C})$ such that
 $$
 \|\phi'(f)-u^*\psi'(f)u\|<\frac{\varepsilon}{4},\,\,\,\forall f\in F.
 $$
 Hence for this unitary,
 $$
 \|\phi(f)-u^*\psi(f)u\|<\frac{3}{4}\varepsilon<\varepsilon,\,\,\,\forall f\in F.
 $$
\end{proof}
In \cite{Liu:2019}, the second author uses Gong's pairing lemma (see \cite{EGLN2:2017}) to prove the following result.
\begin{lem}\label{marr}
  Let $A\in \mathcal{D}$ and $\eta=\frac{1}{m}$, $m\in\mathbb{N}_+$. If $\phi,\psi:A\rightarrow M_n(\mathbb{C})$ are homomorphisms such that $dist(Eig(\phi(h)),Eig(\psi(h)))<1$ for all $h\in H(\eta_1)$, then there exist $X\subset Sp\phi\cap(0,1)$, $Y\subset Sp\psi\cap(0,1)$ with $X\supset Sp\phi\cap[\eta,1-\eta]$ , $Y\supset Sp\psi\cap[\eta,1-\eta]$ such that $X$ and $Y$ can be paired to within $2\eta$ one by one.
\end{lem}
\begin{lem} \label{kktoaue pro}
  Let $A\in \mathcal{D}$, $F\subset A$ be finite set and $1>\varepsilon>0$. Suppose that $\eta=\frac{1}{m}$ satisfies that for any $x,y\in [0,1]$ with $dist(x,y)<4N_A\eta$, $\|f(x)-f(y)\|<\frac{\varepsilon}{8}$ for any $f\in F$. Let $\eta_1=\frac{1}{m_1}<\frac{\eta}{2}$ satisfy that $\|h(x)-h(y)\|<\frac{1}{4}$ for any $x,y\in [0,1]$ with $dist(x,y)\leq2\eta_1$ and for all $h\in H(\eta)$. If homomorphisms $\phi,\psi:A\rightarrow M_r(\mathbb{C})$ satisfy the following conditions:

  $(a)$ $dist(Eig(\phi(h)),Eig(\psi(h)))<\frac{1}{2}$, $\forall\,h\in H(\eta);$

  $(b)$ $dist(Eig(\phi(h)),Eig(\psi(h)))<1$, $\forall\,h\in H(\eta_1);$

  $(c)$ $KK(\phi)=KK(\psi)\in KK(A,\mathbb{C}),$

  then there exists a unitary $u\in M_r(\mathbb{C})$ such that
  $$
  \|\phi(f)-u^*\psi(f)u\|<\varepsilon,\,\,\forall f\in F.
  $$
\end{lem}
\begin{proof}
  From Lemma \ref{marr} , there exist $X\subset Sp\phi\cap(0,1)$, $Y\subset Sp\psi\cap(0,1)$
   with $X\supset Sp\phi\cap[\eta_1,1-\eta_1]$, $Y\supset Sp\psi\cap[\eta_1,1-\eta_1]$ such that $X$ and $Y$ can be paired to within $2\eta_1$ one by one, denote the one to one correspondence by $\pi:X\rightarrow Y$.

  There exist unitaries $u_1,u_2\in M_n(\mathbb{C})$ such that
  $$
  \phi(f,a)=u_1^*\cdot{\rm diag}\big(a(\theta_1)^{\sim t_1},\cdots,a(\theta_p)^{\sim t_p},f(x_1),f(x_2),\cdots,f(x_{\bullet}),0,\cdots,0\big)\cdot u_1
  $$
  $$
  \psi(f,a)=u_2^*\cdot{\rm diag}\big(a(\theta_1)^{\sim s_1},\cdots,a(\theta_p)^{\sim s_p},f(y_1),f(y_2),\cdots,f(y_{\bullet\bullet}),0,\cdots,0\big)\cdot u_2,
  $$
  then $Sp\phi\cap(0,1)=\{x_1,x_2,\cdots,x_{\bullet}\}$ and $Sp\psi\cap(0,1)=\{y_1,y_2,\cdots,y_{\bullet\bullet}\}$.

   We will perturb $\phi,\psi$ to obtain $\phi',\psi'$. To define $\phi'$, change all the elements $x_k\in Sp\phi\cap(0,\eta_1)\backslash X$ to $0\thicksim \{\theta_1^{\thicksim\alpha_1},\cdots,
 \theta_p^{\thicksim\alpha_p}\}$ and $x_k\in Sp\phi\cap(1-\eta_1,1)\backslash X$ to $1\thicksim \{\theta_1^{\thicksim\beta_1},\cdots,\theta_p^{\thicksim\beta_p}\}$.
 And finally, change all the elements $x_k\in X$ to $\pi(x_k)\in Y$. To define $\psi'$, keep all the points in $Y$, change $x_k\in Sp\psi\cap(0,\eta_1)\backslash X$ to $0$ and $x_k\in Sp\psi\cap(1-\eta_1,1)\backslash X$ to $1$. In this way, each eigenvalue is perturbed by
 at most $2\eta_1$.

 After this modification, we have
 $$
 Sp\phi'\cap(0,1)=Sp\psi'\cap(0,1)=Y,\,\,({\rm counting\,\,multiplicity}),
 $$
 $$
 \|\phi(f)-\phi'(f)\|<\frac{\varepsilon}{8},\quad \|\psi(f)-\psi'(f)\|<\frac{\varepsilon}{8},\quad\forall\,f\in F,
 $$
 $$
 \|\phi(h)-\phi'(h)\|<\frac{1}{4},\quad \|\psi(h)-\psi'(h)\|<\frac{1}{4},\quad\forall\,h\in H(\eta).
 $$

 Then $dist(Eig(\phi(h)),Eig(\phi'(h)))<\frac{1}{4}$ and $dist(Eig(\psi(h)),Eig(\psi'(h)))<\frac{1}{4}$ for any $h\in H(\eta)$, and we have $dist(Eig(\phi'(h)),Eig(\psi'(h)))<1$ for any $h\in H(\eta)$.

 Since $\phi$ is homotopic to $\phi'$ and $\psi$ is homotopic to $\psi'$, then
 $$
 KK(\phi')=KK(\phi)=KK(\psi)=KK(\psi').
 $$
 Apply Lemma \ref{kktoaue} for $\phi',\psi'$ and $\frac{\varepsilon}{2}$ (in place of $\varepsilon$), there exists a unitary $u\in M_r(\mathbb{C})$ such that
 $$
 \|\phi'(f)-u^*\psi'(f)u\|<\frac{\varepsilon}{2},\,\,\forall f\in F.
 $$
 Consequently,
 $$
 \|\phi(f)-u^*\psi(f)u\|<\frac{\varepsilon}{2}+\frac{\varepsilon}{8}
 +\frac{\varepsilon}{8}<\varepsilon,\,\,\forall f\in F.
 $$
\end{proof}
\begin{cor}\label{wv1}
  Let $A\in \mathcal{D}$ be minimal, $F\subset A$ be a finite set and $1>\varepsilon>0$. Choose $\eta,\eta_1$ as in Lemma \ref{kktoaue pro}. Suppose that $B\in \mathcal{D}$, $\tau:A\rightarrow B$ is a homomorphism  with  $$\tau(H(\eta)\cup H(\eta_1))\subset_{1/4}\{f\in B_{sa} |f\,has\,finite\,spectrum\,\},$$ then we have $\omega(\tau(F))<\varepsilon$.
\end{cor}
\begin{proof}
  We need only to prove that for any homomorphism $\phi,\psi: B\rightarrow M_r(\mathbb{C})$ with $KK(\phi)=KK(\psi)$, there exists a unitary $u$ such that
  $$
  \|u(\phi\tau(f))u^*-\psi\tau(f)\|<\varepsilon,\,\,\forall f\in F.
  $$

  For each $h\in H(\eta)\cup H(\eta_1)$, there are mutually orthogonal projections $p_1(h)$, $p_2(h)$, $\cdots$, $p_{m(h)}(h)\in B$
  and real numbers $\lambda_1(h),\lambda_2(h),\cdots,\lambda_{m(h)}(h)$ such that
  $$
  \|\tau(h)-\sum_{k=1}^{m(h)}\lambda_k(h)p_k(h)\|<\frac{1}{4}.
  $$
  Then we have
  $$
  \|\phi(\tau(h))-\sum_{k=1}^{m(h)}\lambda_k(h)\phi(p_k(h))\|<\frac{1}{4},
  \quad
  \|\psi(\tau(h))-\sum_{k=1}^{m(h)}\lambda_k(h)\psi(p_k(h))\|<\frac{1}{4}.
  $$
  Since $KK(\phi)=KK(\psi)$, we have
  $$[\phi(p_k(h))]= [\psi(p_k(h))]\,\, \,\,in\,\,K_0(M_r(\mathbb{C})),\quad  k=1,2,\cdots,m(h),$$
  and there exists a unitary $v\in M_r(\mathbb{C})$ such that
 $$
  \|v(\phi\tau(h))v^*-\psi\tau(h)\|<\frac{1}{2}.
  $$
  Then we have $dist(Eig(\phi\tau(h)),Eig(\phi\tau(h)))<\frac{1}{2}$ for any $h\in H(\eta)\cup H(\eta_1)$.

  Since $KK(\phi\tau)=KK(\psi\tau)$, by Lemma \ref{kktoaue pro}, there exists a unitary $u\in M_r(\mathbb{C})$ such that
  $$\|u(\phi\tau(f))u^*-\psi\tau(f)\|<\varepsilon,\,\,\forall f\in F.$$
  Now we have  $\omega(\tau(F))<\varepsilon$.
\end{proof}
  Combine Lemma \ref{fini} and Corollary \ref{wv1}, we will have
\begin{thrm}\label{wv2}
  Let $A=\underrightarrow{lim}(A_n,\phi_{n,m})$ be a real rank zero inductive limit of $\mathrm{C}^*$-algebras in $\mathcal{D}$.
  Let $F\subset A_n$ be a finite subset and $\varepsilon>0$, there exists an integer $m\geq n$ such that $\omega(\phi_{n,r}(F))<\varepsilon$ for
  all $r\geq m$.
\end{thrm}

\section{Decomposition theorem}

\begin{lem}\label{ccut}
   Let $E_1,\cdots,E_s\subset [0,1]$ be a collection of finitely many finite sets, $L$ be a positive integer, let $\eta=\frac{1}{(L+1)^s}$, then there exist integers $0\leq c<d \leq (L+1)^s$ with $d-c=1$ such that
   $$
   (L+1)\cdot\#\big(E_i\cap(c\eta,d\eta)\big)\leq \#(E_i),\quad\forall \,i\in \{1,\cdots,s\}.
   $$
\end{lem}
 \begin{proof}
 This proof is just a slight generalization of the first part of
 \cite[Lemma 2.21]{EG:1996}.

 Firstly, we divide $[0,1]$ into $L+1$ intervals of equal length $\frac{1}{L+1}$ by points
 $$
 0=a_0^1<a_1^1<\cdots<a_{L+1}^1=1,
 $$
 then there is an interval $[a_{r_1}^1,a_{r_1+1}^1]\subset[0,1]$ such that
 $$
 (L+1)\cdot\#\big(E_1\cap(a_{r_1}^1,a_{r_1+1}^1)\big)\leq \#(E_1).
 $$
 Next we divide $[a_{r_1}^1,a_{r_1+1}^1]$ into $L+1$ intervals of equal length $\frac{1}{(L+1)^2}$ by points
 $$
 a_{r_1}^1=a_0^2<a_1^2<\cdots<a_{L+1}^2=a_{r_1+1}^1,
 $$
 there is an interval $[a_{r_2}^2,a_{r_2+1}^2]\subset[a_{r_1}^1,a_{r_1+1}^1]$ such that
 $$
 (L+1)\cdot\#\big(E_2\cap(a_{r_2}^2,a_{r_2+1}^2)\big)\leq
  \#\big(E_2\cap(a_{r_1}^1,a_{r_1+1}^1)\big)\leq \#(E_2),
 $$
 we also have
 $$
 (L+1)\cdot\#\big(E_1\cap(a_{r_2}^2,a_{r_2+1}^2)\big)\leq
  (L+1)\cdot\#\big(E_1\cap(a_{r_1}^1,a_{r_1+1}^1)\big)\leq \#(E_1).
 $$

 Repeat this operation, then we find an interval $[a_{r_s}^s,a_{r_s+1}^s]$ such that
 $$
 (L+1)\cdot\#\big(E_i\cap(a_{r_s}^s,a_{r_s+1}^s)\big)\leq \#(E_i),\quad \forall\,i\in\{1,2,\cdots,s\}
 $$
 Let $c=a_{r_s}^s\cdot(L+1)^s,d=a_{r_s+1}^s\cdot(L+1)^s$, this completes the proof.
\end{proof}
\begin{definition}\label{defdis}
  Let $A,B\in \mathcal{D}$ be minimal and $\eta^{-1},K,L\in \mathbb{N}_{+}$, a homomorphism $\phi:A \rightarrow B$ has $(\eta,K,L)$-$distribution$, if for any $r\in\{1,2,\cdots,K\}$, there are integers $a_r,b_r$ with
 $$
 \frac{r-1}{K}\leq a_r\eta<b_r\eta\leq \frac{r}{K}\quad{\rm and}\quad
  b_r-a_r=1,
 $$
  such that the following properties:

  (i) For any $x\in [0,1]\subset Sp(B)$, we have
  $$
  \#\big(Sp\,(\pi_x\circ\phi)\cap (a_r\eta,b_r\eta)\big)\leq\frac{1}{L+1}\cdot
  \#\big(Sp\,(\pi_x\circ\phi)\cap (\frac{r-1}{K},\frac{r}{K})\big);
  $$

  (ii) For any $i'=1,2,\cdots,p'$, we have
  $$
  \#\big(Sp\,(\pi_e^{i'}\circ\phi)\cap (a_r\eta,b_r\eta)\big)\leq\frac{1}{L+1}\cdot
  \#\big(Sp\,(\pi_e^{i'}\circ\phi)\cap (\frac{r-1}{K},\frac{r}{K})\big),
  $$
  where $\pi_x,\pi_e^{i'}$ are defined in \ref{homodeff}.

   For $A,B\in\mathcal{D}$, we say a homomorphism $\phi:A \rightarrow B$ has $(\eta,K,L)$-$decomposition$, if there exist finite many homomorphisms $\phi_1,\phi_2,\cdots,\phi_s$ between minimal blocks in $\mathcal{D}$ such that $\phi=\phi_1\oplus\phi_2\oplus\cdots\oplus\phi_s$ and each of them has $(\eta,K,L)$-$distribution$.
\end{definition}
\begin{lem}\label{disexist}
  Let $A,B\in \mathcal{D}$ be minimal and $K,L\in \mathbb{N}_{+}$, set $\eta=\frac{1}{8K(L+1)}$, if a homomorphism $\phi: A\rightarrow B$ satisfys that
  $$\phi(H(\eta))\subset_{1/6}\{ f\in B_{sa}\, |f\,has\,finite\,spectrum\},$$
  then there exists $\delta>0$ such that $\phi$ has $(\delta,K,L)$-distribution.
\end{lem}
\begin{proof}
  For any $h\in H(\eta)$, there are mutually orthogonal projections $p_1(h)$, $p_2(h)$, $\cdots$, $p_{m(h)}(h)\in B$ and real numbers $\lambda_1(h),\lambda_2(h),\cdots,\lambda_{m(h)}(h)$ such that
   $$
   \|\phi(h)-\sum_{k=1}^{m(h)}\lambda_k(h)p_k(h)\|<\frac{1}{6}.
   $$

  Then for any $x_1,x_2\in [0,1]\in Sp(B)$, $dist(Eig(\phi_{x_1}(h)),Eig(\phi_{x_2}(h)))<\frac{1}{3}$ for all $h\in H(\eta)$, by Lemma \ref{marr}, then there exist two sets $X(x_1)\subset Sp\phi_{x_1}\cap(0,1)$, $X(x_2)\subset Sp\phi_{x_2}\cap(0,1)$ with $X(x_1)\supset Sp\phi_{x_1}\cap[\eta,1-\eta]$ , $X(x_2)\supset Sp\phi_{x_2}\cap[\eta,1-\eta]$ such that $X(x_1)$ and $X(x_2)$ can be paired to within $2\eta$ one by one.

  Fix $x_0\in [0,1]\subset Sp(B)$.
  Set $\gamma_0=0, \gamma_1=\frac{1}{K},\cdots,\gamma_K=1$, apply Lemma \ref{ccut} for
  $[\gamma_{r-1}+2\eta,\gamma_r-2\eta]$ (in place of $[0,1]$), $Sp\phi_{x_0}\cap(\gamma_{r-1}+2\eta,\gamma_r-2\eta)$ (in place of $E$) and $L$, then there exist integers $c_r,d_r$ $(1\leq r\leq K)$ such that
  $$
 \frac{r-1}{K}+2\eta\leq \frac{c_r}{K(L+1)}<\frac{d_r}{K(L+1)}\leq \frac{r}{K}-2\eta\quad{\rm and}\quad
  d_r-c_r=1.
  $$
  Then we have
  $$
  \#\big(Sp\phi_{x_0}\cap (8c_r\eta,8d_r\eta)\big)\leq\frac{1}{L+1}\cdot
  \#\big(Sp\phi_{x_0}\cap (\frac{r-1}{K}+2\eta,\frac{r}{K}-2\eta)\big).
  $$
  Since $X(x)$ and $X(x_0)$ can be paired to within $2\eta$ one by one for any $x\in [0,1]$, then we have
  $$
  \#\big(Sp\phi_{x_0}\cap (\frac{r-1}{K}+2\eta,\frac{r}{K}-2\eta)\big)\leq
  \#\big(Sp\phi_{x}\cap (\frac{r-1}{K},\frac{r}{K})\big)
  $$
  and
   $$
  \#\big(Sp\phi_{x}\cap (8c_r\eta+2\eta,8d_r\eta-2\eta)\big)\leq
  \#\big(Sp\phi_{x_0}\cap (8c_r\eta,8d_r\eta)\big)
  $$
  for any $x\in [0,1]\subset Sp(B)$.

  Now we have
  $$
  \#\big(Sp\phi_x\cap (8c_r\eta+2\eta,8d_r\eta-2\eta)\big)\leq\frac{1}{L+1}\cdot
  \#\big(Sp\phi_x\cap (\frac{r-1}{K},\frac{r}{K})\big)
  $$
  for any $x\in [0,1]\subset Sp(B)$.

  Set $\delta=\frac{1}{(L+1)^{p'}}\cdot \eta$, then for each $r$, apply Lemma \ref{ccut} for interval $[8c_r\eta+2\eta,8d_r\eta-2\eta]$ and $(\pi_e^{i'}\circ\phi)\cap(\frac{r-1}{K},\frac{r}{K})$ (in place of $E_1,E_2,\cdots$) ,   there exist integers
  $a_r,b_r$ with  $b_r-a_r=1$ and
  $$
  \frac{c_r}{K(L+1)}+2\eta\leq a_r\delta<b_r\delta\leq \frac{d_r}{K(L+1)}-2\eta
  $$
  such that
  $$
  \#\big(Sp\,(\pi_e^{i'}\circ\phi)\cap (a_r\delta,b_r\delta)\big)\leq\frac{1}{L+1}\cdot
  \#\big(Sp\,(\pi_e^{i'}\circ\phi)\cap (\frac{r-1}{K},\frac{r}{K})\big)
  $$
  for any $i'=1,2,\cdots,p'$.

  Note that for any $x\in [0,1]\subset Sp(B)$,
 \begin{eqnarray*}
    \#\big(Sp\phi_x\cap (a_r\delta,b_r\delta)\big) &\leq&
     \#\big(Sp\phi_x\cap (8c_r\eta+2\eta,8d_r\eta-2\eta)\big) \\
     &\leq& \frac{1}{L+1}\cdot\#\big(Sp\phi_x\cap (\frac{r-1}{K},\frac{r}{K})\big).
  \end{eqnarray*}
 Then we have $\phi$ has $(\delta,K,L)-distribution$.
\end{proof}
\begin{lem}\label{discirc}
  Let $A,B,C\in \mathcal{D}$ be minimal, $K,L\in \mathbb{N}_{+}$ and $\delta>0$. Let $\phi: A\rightarrow B$ and $\psi:B\rightarrow C$ be homomorphisms. Suppose that $\phi$ has $(\delta,K,L)$-$distribution$, then $\psi\circ\phi$ has $(\delta,K,L)$-$distribution$.
\end{lem}
\begin{proof}
  Since $\phi$ is $(\delta,K,L)-distribution$, there are integers $a_r,b_r$ $(1\leq r\leq K)$  with
  $$
  \frac{r-1}{K}\leq a_r\delta<b_r\delta\leq \frac{r}{K}\quad{\rm and}\quad
  b_r-a_r=1,
  $$
  such that
  $$
  \#\big(Sp\,(\pi_x\circ\phi)\cap (a_r\delta,b_r\delta)\big)\leq\frac{1}{L+1}\cdot
  \#\big(Sp\,(\pi_x\circ\phi)\cap (\frac{r-1}{K},\frac{r}{K})\big).
  $$
  for any $x\in Sp(B)$.

  For any $y\in Sp(C)$, it is clear that
  \begin{eqnarray*}
    \#\big(Sp\,(\pi_y\circ\psi\circ\phi)\cap (a_r\delta,b_r\delta)\big) &= & \sum_{x\in Sp\psi_y}\#\big(Sp\,(\pi_x\circ\phi)\cap (a_r\delta,b_r\delta)\big) \\
     &\leq & \frac{1}{L+1}\cdot\sum_{x\in Sp\psi_y}\#\big(Sp\,(\pi_x\circ\phi)\cap (\frac{r-1}{K},\frac{r}{K})\big) \\
     &=& \frac{1}{L+1}\cdot
  \#\big(Sp\,(\pi_y\circ\psi\circ\phi)\cap (\frac{r-1}{K},\frac{r}{K})\big)
  \end{eqnarray*}
  Then $\psi\circ\phi$ has $(\delta,K,L)$-$distribution$.
\end{proof}
The following two lemmas are a slight generalization of Lemma 6.3 and Lemma 6.4 in \cite{DG:1997} with the same proof.
\begin{lem}\label{simcase0}
  Let $A=A(F_1,M_n(\mathbb{C}),\varphi_0,\varphi_1)\in \mathcal{D}$ be minimal. Suppose that $B\in \mathcal{D}$ be minimal, $\phi,\psi: A\rightarrow B$ are two homomorphisms which can be factored through $F_1$ and $\phi',\psi':F_1\rightarrow B$ are homomorphisms with $\phi=\phi'\circ\pi_e$ and $\psi=\psi'\circ\pi_e$. If $KK(\phi')=KK(\psi')$, there exist a unitary $u\in B$ such that $u\phi u^*=\psi$.
\end{lem}
\begin{lem}\label{simcase}
  Let $A=A(F_1,M_n(\mathbb{C}),\varphi_0,\varphi_1)\in \mathcal{D}$ be minimal. Suppose that $B\in \mathcal{D}$ be minimal, $\phi,\psi: A\rightarrow B$ are two homomorphisms which can be factored through $F_1$ and $\phi',\psi':F_1\rightarrow B$ are homomorphisms with $\phi=\phi'\circ\pi_e$ and $\psi=\psi'\circ\pi_e$. If $[\phi'(e)]\geq [\psi'(e)]$ holds for any projection $e\in F_1$, there exist a unitary $u\in B$ and a homomorphism $\tau:A\rightarrow B$ which can be factored through $F_1$ such that $u\phi u^*=\psi+\tau$.
\end{lem}
Now we have the following decomposition theorem.
\begin{thrm}\label{decthm}
  Let $A\in \mathcal{D}$ be minimal, $G\subset A$ be a finite set, $\varepsilon>0$, and $L\in \mathbb{N}_{+}$. Suppose that $K$ is an integer such that $dist(x_1,x_2)\leq \frac{4}{K}$ implies $\|g(x_1)-g(x_2)\|<\varepsilon$ for all $g\in G$.

  Suppose that $B\in \mathcal{D}$ is a minimal block, $\delta>0$ and $\phi: A\rightarrow B$ is a homomorphism with the following properties:

  $(a)\,\,\phi$ has $(\delta,K,L)$-$distribution;$

  $(b)\,\,\phi(H(\delta/8))\subset_{1/6}\{ f\in B_{sa}\, |f\,has\,finite\,spectrum\},$

  then there exist a projection $q\in B$ and two homomorphisms $\nu,\rho: A\rightarrow (1-q)B(1-q)$ with finite dimensional images such that the following holds:

$(1)\,\,L\cdot [q]\leq [\nu(1)]$ in $K_0(B);$

$(2)\,\,Sp\nu\subset Sp(A)\cap(0,1),\,Sp\rho \subset Sp(A)\cap Sp(F_1);$

$(3)\,\,\|q\phi(g)-\phi(g)q\|<4\varepsilon,\,\forall\,\, g\in G;$

$(4)\,\,\|\phi(g)-q\phi(g)q\oplus\nu(g)\oplus\rho(g)\|<4\varepsilon,\,\forall\,\,g\in G.$
\end{thrm}
\begin{proof}
  Set $\eta=\delta/8$. Since $\phi$ is $(8\eta,K,L)$-$distribution$, then there exist integers $a_r,b_r$, $(1\leq r\leq K)$ with
$$
\frac{r-1}{K}\leq a_r\eta<b_r\eta\leq \frac{r}{K}\quad{\rm and}\quad
  b_r-a_r=8.
$$
such that
$$
\#\big(Sp\,\phi_x\cap (a_r\eta,b_r\eta)\big)\leq\frac{1}{L+1}\cdot
\#\big(Sp\,\phi_x\cap (\frac{r-1}{K},\frac{r}{K})\big)
$$
for any $x\in Sp(B)$.

For any $x\in Sp(B)$, write
$$
Sp\phi_x\cap Sp(F_1)=\{\theta_1^{\sim\,t_1(x)},\theta_2^{\sim\,t_2(x)},
\cdots,\theta_p^{\sim\,t_p(x)}\}.
$$

Set
$$
V_0=[0,a_2\eta+2\eta],V_K=[b_{K-1}\eta-2\eta,1].
$$
$$
V_1=[b_1\eta,a_2\eta],V_2=[b_2\eta,a_3\eta],\cdots, V_{K-1}=[b_{K-1}\eta,a_K\eta].
$$
$$
W_r=\{x\in [0,1],\,\,\,d(x,V_r)\leq 2\eta\},\quad\forall\,r=1,2,\cdots,K.
$$
Note that $V_0\supset V_1, V_K\supset V_{K-1}$ and if $r_1,r_2\in \{0,2,\cdots,K-2,K\},r_1\neq r_2$, we have $W_{r_1}\cap W_{r_2}=\varnothing$.

Then we have
$$
 L\cdot\sum_{r=2}^{K-2}
\begin{pmatrix}
 \#\big(Sp\phi_{\theta_1'}\cap (a_r\eta,b_r\eta)\big) \\
 \#\big(Sp\phi_{\theta_2'}\cap (a_r\eta,b_r\eta)\big)\\
 \vdots \\
 \#\big(Sp\phi_{\theta_p'}\cap (a_r\eta,b_r\eta)\big)
\end{pmatrix}
\leq
\sum_{r=1}^{K-1}
\begin{pmatrix}
\#\big(Sp\phi_{\theta_1'}\cap V_r\big) \\
\#\big(Sp\phi_{\theta_2'}\cap V_r\big)\\
\vdots \\
\#\big(Sp\phi_{\theta_p'}\cap V_r\big)
\end{pmatrix}.
$$

For each $h\in H(\eta)$, there are mutually orthogonal projections $p_1(h)$, $p_2(h)$, $\cdots$, $p_{m(h)}(h)\in B$
 and real numbers $\lambda_1(h),\lambda_2(h),\cdots,\lambda_{m(h)}(h)$ such that
 $$
 \|\phi(h)-\sum_{k=1}^{m(h)}\lambda_k(h)p_k(h)\|<\frac{1}{6}.
 $$
 Denote $\Lambda(h)$ by the set of all the eigenvalues of $\sum_{k=1}^{m(h)}\lambda_k(h)p_k(h)$, set
 $$
 \Lambda_1(h)=\{\lambda | \lambda\in \Lambda(h),\,\,\lambda\in(1-\frac{1}{6},\,\,1]\},
 \quad
 [\Lambda_1(h)]=\sum_{\lambda_k(h)\in \Lambda_1(h)}\,[p_k(h)]\in\, K_0(B).
 $$

 Let $h_r$ be the test function corresponding to $V_r$, then from step 2 in \cite[Theorem 3.1]{Liu:2019}, we can construct a collection of mutually orthogonal projections $P_1,P_2,\cdots,P_{K-1}$ almost commutes with $\phi(g)$ for all $g\in G$ and
 $$
 n\cdot
 \begin{pmatrix}
     \#\big(Sp\phi_{\theta_1'}\cap W_r\big) \\
    \#\big(Sp\phi_{\theta_2'}\cap W_r\big)\\
     \vdots \\
    \#\big(Sp\phi_{\theta_p'}\cap W_r\big)
   \end{pmatrix}\geq
 [P_r]=n\cdot[\Lambda_1(h_r)]\geq
 n\cdot
 \begin{pmatrix}
     \#\big(Sp\phi_{\theta_1'}\cap V_r\big) \\
    \#\big(Sp\phi_{\theta_2'}\cap V_r\big)\\
     \vdots \\
    \#\big(Sp\phi_{\theta_p'}\cap V_r\big)
   \end{pmatrix}
 $$
 for all $r= 1,2,\cdots,K-1$.

 Now we deal with $V_0,V_K$, set
 $$
 \widetilde{V_j}=\bigcup_{\alpha_j\neq0}V_0\cup\bigcup_{\beta_j\neq0}V_K,
 \quad j=1,2,\cdots,p.
 $$
 Then we turn these two intervals to $p$ subsets.

 Let $f_j$ be the test function corresponding to $\widetilde{V_j}$ and $X_j=\{\theta_j\}$. From the step 3 in \cite[Theorem 3.1]{Liu:2019}, we can construct a collection of mutually orthogonal projections $Q_1,Q_2,\cdots,Q_p$. Each of them almost commutes with $\phi(g)$ for all $g\in G$ and
 $$
 [Q_j]=k_j\cdot[\Lambda_1(f_j)]\geq k_j\cdot
 \begin{pmatrix}
     t_j(\theta_1') \\
    t_j(\theta_2') \\
     \vdots \\
    t_j(\theta_p')
   \end{pmatrix}+ k_j\alpha_j\cdot
 \begin{pmatrix}
     \#\big(Sp\phi_{\theta_1'}\cap V_0\big) \\
    \#\big(Sp\phi_{\theta_2'}\cap V_0\big)\\
     \vdots \\
    \#\big(Sp\phi_{\theta_p'}\cap V_0\big)
   \end{pmatrix}+k_j\beta_j\cdot
 \begin{pmatrix}
     \#\big(Sp\phi_{\theta_1'}\cap V_K\big) \\
    \#\big(Sp\phi_{\theta_2'}\cap V_K\big)\\
     \vdots \\
    \#\big(Sp\phi_{\theta_p'}\cap V_K\big)
   \end{pmatrix}
 $$
for all $j= 1,2,\cdots,p$.

 Change all the spectra in $Sp\phi\cap W_r$ to $\frac{r}{K}\in W_r$ for each $r=0,2\cdots,K-2,K$. We obtain a pointwise homomorphism $\psi$.

 Denote
 $$
 P=\sum_{r=2}^{K-2}P_r,\quad Q=\sum_{j=1}^{p}Q_j,\quad  R=P_1+P_{K-1},\quad
 q=1-P-Q.
 $$
 Then we have
$$
 [Q]\geq \sum_{j=1}^{p}k_j
 \begin{pmatrix}
     t_j(\theta_1') \\
    t_j(\theta_2') \\
     \vdots \\
    t_j(\theta_p')
   \end{pmatrix}+ n\cdot
 \begin{pmatrix}
     \#\big(Sp\phi_{\theta_1'}\cap V_0\big) \\
    \#\big(Sp\phi_{\theta_2'}\cap V_0\big)\\
     \vdots \\
    \#\big(Sp\phi_{\theta_p'}\cap V_0\big)
   \end{pmatrix}+n\cdot
 \begin{pmatrix}
     \#\big(Sp\phi_{\theta_1'}\cap V_K\big) \\
    \#\big(Sp\phi_{\theta_2'}\cap V_K\big)\\
     \vdots \\
    \#\big(Sp\phi_{\theta_p'}\cap V_K\big)
   \end{pmatrix}
   \geq [R]
 $$
Hence,
 $$
[q]\leq
 n\cdot \sum_{r=2}^{K-2}
 \begin{pmatrix}
     \#\big(Sp\phi_{\theta_1'}\cap (a_r\eta,b_r\eta)\big) \\
    \#\big(Sp\phi_{\theta_2'}\cap (a_r\eta,b_r\eta)\big)\\
     \vdots \\
    \#\big(Sp\phi_{\theta_p'}\cap (a_r\eta,b_r\eta)\big)
 \end{pmatrix}.
 $$

 Define homomorphisms $\psi_1,\psi_2,\psi_3:A\,\rightarrow (1-q)B(1-q)$ as follows:
 $$
 \psi_1=P\psi P,\quad \psi_2=Q\psi Q,\quad \psi_3=R\psi R.
 $$
 Then $\psi_1,\psi_2,\psi_3$ can be factored through a finite dimensional $C^*$-algebra and we have
 $$
 \|\phi(g)q-q\phi(g)\|<4\varepsilon,
 \quad \forall\,g\in G
 $$
 and
 $$
 \|\phi(g)-q\phi(g)q\oplus\psi_1(g)\oplus\psi_2(g)\|<4\varepsilon,
 \quad \forall\,g\in G.
 $$

 Note that $\psi_2$ and $\psi_3$ factor through $F_1$, there exist homomorphisms $\psi_2',\psi_3':F_1\rightarrow B$ such that $\psi_2=\psi_2'\circ\pi_e$ and $\psi_3=\psi_3'\circ\pi_e$. Since $[Q]\geq [R]$, we have $[\psi_2'(e)]\geq [\psi_3'(e)]$ holds for any nonzero projection $e\in F_1$, by Lemma \ref{simcase}, there exists a homomorphism $\psi_4:A\rightarrow B$ which can be factored through $F_1$ and a unitary $v\in B$ such that
 $$
 u\psi_2u^*=\psi_3\oplus\psi_4
 \quad{\rm and}\quad
 u^*Ru\leq Q.
 $$

 Denote
 $$
 \nu=\psi_1\oplus u^*\psi_3u,\quad \rho=u^*\psi_4u.
 $$
 Then we have
 $$
 \|\phi(g)-q\phi(g)q\oplus\nu(g)\oplus\rho(g)\|<4\varepsilon,
 \quad\forall\,g\in G.
 $$
 Since $u^*\psi_3u$ factors through $f(0)\oplus f(1)$, we can perturb $0,1$ small enough to make sure that condition (4) still holds, then we have $Sp\nu\subset Sp(A)\cap(0,1)$ and $Sp\rho\subset Sp(A)\cap Sp(F_1).$

 We need only to verify (1), note that
 $$
 [\nu(1)]=[P]+[R]\geq n\cdot\sum_{r=1}^{K-1}
 \begin{pmatrix}
     \#\big(Sp\phi_{\theta_1'}\cap V_r\big) \\
    \#\big(Sp\phi_{\theta_2'}\cap V_r\big)\\
     \vdots \\
    \#\big(Sp\phi_{\theta_p'}\cap V_r\big)
   \end{pmatrix}.
 $$
 Now we have
 $$L\cdot[q]\leq [P]+[R] \leq [\nu(1)].$$
\end{proof}
\begin{lem}\label{smltrick}
  Let $A=\underrightarrow{lim}(A_n,\phi_{n,m})$ be a real rank zero inductive limit of algebras in $\mathcal{D}$. Let $m,K,L\in \mathbb{N}_{+}$. There exist $\delta>0$ and an integer $m\geq n$ such that the following holds:

  $(1)\,\, \phi_{n,m}$ has $(\delta,K,L)$-$decomposition;$

  $(2)\,\, \phi_{n,m}(H(\delta/8))\subset_{1/6}\{ f\in (A_m)_{sa}\, |f\,has\,finite\,spectrum\}.$
\end{lem}
\begin{proof}
  Let $\eta=\frac{1}{8K(L+1)}$, we have a finite set $H(\eta)\subset A_n$,
  by Lemma \ref{fini}, there exists $r>n$ such that
  $$
  \phi_{n,r}(H(\eta))\subset_{1/6}\{ f\in (A_r)_{sa}\, |f\,has\,finite\,spectrum\}.
  $$
  Apply Lemma \ref{disexist} for $\eta$ and each partial map $\phi_{n,r}^{i,j}$, there exists $\delta_{i,j}>0$ such that $\phi_{n,r}^{i,j}$ has $(\delta_{i,j},K,L)$-$distribution$. Set $\delta=(\prod_{i,j}\delta_{i,j}^{-1})^{-1}$, then $\phi_{n,r}^{i,j}$ has $(\delta,K,L)$-$distribution$. Now we consider the finite set $H(\delta/8)\subset A_n$, use Lemma \ref{fini} again, there exists an integer $m\geq r>n$ such that
  $$
  \phi_{n,m}(H(\delta/8))\subset_{1/6}\{ f\in (A_m)_{sa}\, |f\,has\,finite\,spectrum\}
  $$
  From Lemma \ref{discirc}, $\phi_{n,r}^{i,j}$ has $(\delta,K,L)$-$distribution$ implies that $\phi_{r,m}^{j,s}\circ\phi_{n,r}^{i,j}$ has $(\delta,K,L)$-$distribution$, then $\phi_{n,m}=\bigoplus_{i,j,s}(\phi_{r,m}^{j,s}\circ\phi_{n,r}^{i,j})$ has $(\delta,K,L)$-$decomposition$.
\end{proof}
\begin{thrm}[Theorem 6.2.2 in \cite{ELP1:1998}]\label{nccw semi}
Every one-dimensional NCCW complex is semiprojective.
\end{thrm}
  From Theorem \ref{nccw semi}, we can easily get the following lemma.
\begin{lem}\label{ccp to homo}
Let $A\in \mathcal{C}$, for any finite subset $F\subset A$, $\varepsilon>0$, there exist a finite subset
$G\subset A$ and $\delta>0$ such that if
$\phi:\,A\to B$ is a c.c.p. map with
$$
\parallel\phi(g_1g_2)-\phi(g_1)\phi(g_2)\parallel<\delta,\quad \forall g_1,\,g_2\in G,
$$
then there is a homomorphism $\psi:\,A\to B$ satisfying
$$\parallel\psi(f)-\phi(f)\parallel<\varepsilon,\quad\forall f\in F.$$
\end{lem}
\begin{lem}\label{stronglem}
  Let $A=A(F_1,M_n(\mathbb{C}),\varphi_0,\varphi_1)\in \mathcal{D}$. Suppose that $B\in \mathcal{D}$ and $\nu: A\rightarrow B$ is a homomorphism with finite dimensional range and $Sp\nu\subset Sp(A)\cap(0,1)$. Then
  $$
  n\cdot[\nu(e)]\geq [\nu(1)]\,\,\, in\,\,\, K_0(B)
  $$
  holds for any nonzero projection $e\in A$.
\end{lem}
\begin{proof}
  It follows from
  $$
  n\cdot[\nu(e)]\geq n\cdot
  \begin{pmatrix}
     \#(Sp(\pi_e^1\circ\nu)\cap (0,1)) \\
     \#(Sp(\pi_e^2\circ\nu)\cap (0,1)) \\
     \vdots \\
     \#(Sp(\pi_e^{p'}\circ\nu)\cap (0,1))
  \end{pmatrix}
  = [\nu(1)].
$$
\end{proof}
\begin{cor} \label{deccor}
  Let $A=\underrightarrow{lim}(A_n,\phi_{n,m})$ be a real rank zero inductive limit of Elliott-Thomsen algebras in $\mathcal{D}$.
  For any $\varepsilon >0$, finite set $F\subset A_n$ and a positive integer $L$, there exist an integer $m\geq n$, a projection $q\in A_m$, a homomorphism $\lambda:A_n\rightarrow qA_mq$ and homomorphisms  $\nu,\rho:A_n\rightarrow (1-q)A_m(1-q)$ with finite dimensional ranges such that

  $(1)\,\,L\cdot [\lambda(1)]\leq [\nu(e)]$ in $K_0(A_m)$ for any nonzero projection $e\in A_n;$

  $(2)\,\, Sp\nu\subset Sp(A_n)\cap (0,1),\, Sp\rho\subset Sp(A_n)\cap Sp(F_1);$

  $(3)\,\,\|\phi_{n,m}(f)-\lambda(f)\oplus\nu(f)\oplus \rho(f)\|<5\varepsilon,\,\forall\,\,f\in F.$
\end{cor}
\begin{proof}
  It follows from Theorem \ref{decthm}, Lemma \ref{smltrick},  Lemma \ref{ccp to homo} and Lemma \ref{stronglem}.
\end{proof}

\section{The Invariant and KK-Lifting}
Before we give our Existence Theorem, we should introduce the invariant we concern.
\begin{notion}[\cite{DL2:1996}, \cite{Ell2:1993}]\rm
Consider the algebra
$$
I_p=\{f\in M_p(C_0(0,1]):\,f(1)=\lambda\cdot1_p,\,1_p {\rm \,is\, the\, identity\, of}\, M_p\},
$$
and the algebra $\widetilde{I}_p$ obtained by adjoining a unit to $I_p$.
\end{notion}
\begin{notion}[\cite{DG:1997}]\rm
For a  ${\mathrm C}^*$-algebra $A$, the total $\mathrm{K}$-theory of $A$ is defined by
$$
\underline{\mathrm{K}}(A)=\bigoplus_{p=0}^\infty \mathrm{K}_* (A;\mathbb{Z}_p),
$$
with $\mathrm{K}_*(A;\mathbb{Z}_p)=\mathrm{K}_*(A)$ for $p=0$, $\mathrm{K}_*(A;\mathbb{Z}_p)=0$ for $p=1$,
and $\mathrm{K}_*(A;\mathbb{Z}_p)=\mathrm{KK}(I_{p},A\otimes C(S^1))$ for $p\geq2$.
\end{notion}
\begin{notion}[\cite{DG:1997}]\rm
We will consider the group
$$\mathrm{K}_*(A;\mathbb{Z}\oplus\mathbb{Z}_p)=\mathrm{K}_*(A)\oplus \mathrm{K}_*(A;\mathbb{Z}_p).$$
By Section 4 of \cite{DG:1997},
$$\mathrm{K}_*(A;\mathbb{Z}\oplus\mathbb{Z}_p)\cong \mathrm{KK}(\widetilde{I_{p}},A\otimes C(S^1)).$$
\end{notion}

\begin{notion}[\cite{DG:1997}, \cite{DL2:1996}]\label{Def DL} (Dadarlat-Loring order) The order structure we work with is $\mathrm{K}_*(A;\mathbb{Z}\oplus\mathbb{Z}_p)^{+}$, which can be identified as the image of the abelian semigroup
$[\widetilde{I_{p}},A\otimes C(S^1)\otimes \mathcal{K}]$ in $\mathrm{KK}(\widetilde{I_{p}},A\otimes C(S^1))(\cong \mathrm{K}_*(A;\mathbb{Z}\oplus\mathbb{Z}_p))$.
\end{notion}

\begin{notion}[\cite{DG:1997}, \cite{DL2:1996}]
For a C*-algebra $A$, the invariant we concern with is the tuple $$(\underline{\mathrm{K}}(A),\underline{\mathrm{K}}^+(A),\Sigma(A)),$$
where $\underline{\mathrm{K}}^+(A)$ is the cone generated by all $\mathrm{K}_*(A;\mathbb{Z}\oplus\mathbb{Z}_p)^{+}$,
$p\geq0$ and
$\Sigma(A)$ is the scale of $A$.

We say
$$(\underline{\mathrm{K}}(A),\underline{\mathrm{K}}^+(A),\Sigma(A))
\cong
(\underline{\mathrm{K}}(B),\underline{\mathrm{K}}^+(B),\Sigma(B)),$$
if there is an ordered scaled isomorphism $\rho:\,{\underline{\mathrm{K}}}(A)\to{\underline{\mathrm{K}}}(B)$, which preserves the action of the Bockstein operations.
\end{notion}

We also introduce some preliminaries which will be useful, when we concern the orders on $\mathrm{KK}$-group.
\begin{notion}\label{A^0}\rm
Let $A(F_1,M_n(\mathbb{C}),\varphi_0,\varphi_1)$ be a minimal building block in $\mathcal{D}$ which is not finite dimensional, then $\varphi_0,\,\varphi_1$ induce two maps
$\mathrm{K}_0(F_1)=\mathbb{Z}^p\to \mathbb{Z}=\mathrm{K}_0(M_n(\mathbb{C}))$, which are represented by two $1\times p$ matrices
$\alpha$ and $\beta$. Rearrange $F_1= \bigoplus_{i=1}^p M_{k_i}(\mathbb{C})$ such that
$$
\alpha-\beta=(a_1,a_2,\cdots ,a_r,-b_1,-b_2,\cdots ,-b_l,0,0,\cdots ,0),
$$
where $a_1,a_2,\cdots,a_r,b_1,b_2,\cdots ,b_l>0$.
Note that
$$
\sum_{i=1}^r a_i\cdot k_i=\sum_{i=r+1}^{r+l} b_{i-r}\cdot k_i.
$$
\end{notion}
\begin{definition}\label{liftable}
Let $A,\,B\in\mathcal{C}$.
Define $\mathrm{KK}^+(A,B)$ as the image of the abelian semigroup (of homotopy classes of homomorphisms) $[A,B\otimes\mathcal{K}]$ in $\mathrm{KK}(A,B)$.
We shall say that $\alpha\in \mathrm{KK}(A,B)$ is positive (or lifitable), if $\alpha\in \mathrm{KK}^+(A,B)$.
\end{definition}
\begin{definition}
Let $A,\,B$ be $\mathrm{C}^*$-algebras. Denote ${\bf Hom}_\Lambda(\underline{\mathrm{K}}(A),\underline{\mathrm{K}}(B))$ all the morphisms from $\underline{\mathrm{K}}(A)$ to $\underline{\mathrm{K}}(B)$, which preserves the action of the Bockstein operations.

For $\gamma\in{\bf Hom}_\Lambda(\underline{\mathrm{K}}(A),\underline{\mathrm{K}}(B))$, we say that $$\gamma\in {\bf Hom}^+_\Lambda(\underline{\mathrm{K}}(A),\underline{\mathrm{K}}(B))$$ if $\gamma$ preserves Dadarlat-Loring order.
\end{definition}
We list the following results we need from \ref{K-HOMor} to Theorem \ref{ONEORDER}:

\begin{notion}\label{K-HOMor}
\rm
Let $\phi:\,A\to M_r(\mathbb{C})$ be as described in \ref{SPCForm}, with $Sp(\phi)$ as in (\ref{spectform}).
Even though in general the point
$y_i\in[0,1]_j$ (in $Sp(\phi)$ as in ((\ref{spectform}) of \ref{SPCForm}) may not be the endpoint
$0_j$ or $1_j$, the homomorphism defined by evaluating at this point is homotopic to
the homomorphism defined by evaluating at $0_j$ or $1_j$.
Consequently we can find a new homomorphism $\widetilde{\phi}$ with
$$
\mathrm{KK}(\phi)=\mathrm{KK}(\widetilde{\phi}),\quad Sp(\widetilde{\phi})\subset Sp(F_1).
$$

Now, let us extend this procedure to a homomorphism between two Elliott-Thomsen algebras, as a prelude to describing concretely the $\mathrm{KK}$-group
of these two ${\mathrm C}^*$-algebras.
\end{notion}
\begin{notion}\rm\label{bianjietongtai}
Let $A(F_1,\,F_2,\,\varphi_0,\,\varphi_1),\,B(F_1',\,F_2',\,\varphi_0',\,\varphi_1')$ be in $\mathcal{C}$, let $\phi:\,A\to B$ be a homomorphism, and consider the maps $\pi_0',\,\pi_1':\,B\to F_2'$, where $\pi_t'(f,a)=f(t)=\varphi_t'(a)$,
$t$=0 or 1. Then we can always choose a new homomorphism $\psi:\,A\to B$ such that
$$
\psi\sim_h \phi,\quad
\mathrm{KK}(\psi)=\mathrm{KK}(\phi) \quad{\rm and}
\quad
Sp(\pi_0'\circ\psi),Sp(\pi_1'\circ\psi)\subset Sp(F_1).
$$
The above condition on $Sp(\pi_0'\circ\psi),Sp(\pi_1'\circ\psi)$ is equivalent to $\psi(SF_2)\subset SF_2'$.
Hence, we have a commutative diagram as follows:
$$
\xymatrixcolsep{2pc}
\xymatrix{
{\,\,0\,\,} \ar[r]^-{}
& {\,\,\mathrm{K}_0(A)\,\,} \ar[d]_-{\psi_{0*}} \ar[r]^-{\pi_*}
& {\,\,\mathrm{K}_0(F_1)\,\,} \ar[d]_-{\psi_{0**}} \ar[r]^-{\alpha-\beta}
& {\,\,\mathrm{K}_1(SF_2)\,\,} \ar[d]_-{\psi_{1**}} \ar[r]^-{\iota_*}
& {\,\,\mathrm{K}_1(A)\,\,} \ar[d]_-{\psi_{1*}} \ar[r]^-{}
& {\,\,0\,\,}\\
{\,\,0\,\,} \ar[r]^-{}
& {\,\,\mathrm{K}_0(B)\,\,} \ar[r]_-{\pi_*'}
& {\,\,\mathrm{K}_0(F_1') \,\,} \ar[r]_-{\alpha'-\beta'}
& {\,\,\mathrm{K}_1(SF_2') \,\,} \ar[r]_-{\iota_*'}
& {\,\,\mathrm{K}_1(B)\,\,} \ar[r]^-{}
& {\,\,0\,\,}.}
$$
\end{notion}
\begin{notion}[\cite{AE:2017}]\label{cmset}
Let $A,\,B\in \mathcal{C}$.
Denote by $C(A,B)$ the set of all the commutative diagrams
  $$
\xymatrixcolsep{2pc}
\xymatrix{
{\,\,0\,\,} \ar[r]^-{}
& {\,\,\mathrm{K}_0(A)\,\,} \ar[d]_-{\lambda_{0*}} \ar[r]^-{\pi_*}
& {\,\,\mathrm{K}_0(F_1)\,\,} \ar[d]_-{\lambda_{0}} \ar[r]^-{\alpha-\beta}
& {\,\,\mathrm{K}_1(SF_2)\,\,} \ar[d]_-{\lambda_{1}} \ar[r]^-{\iota_*}
& {\,\,\mathrm{K}_1(A)\,\,} \ar[d]_-{\lambda_{1*}} \ar[r]^-{}
& {\,\,0\,\,}\\
{\,\,0\,\,} \ar[r]^-{}
& {\,\,\mathrm{K}_0(B)\,\,} \ar[r]_-{\pi_*'}
& {\,\,\mathrm{K}_0(F_1') \,\,} \ar[r]_-{\alpha'-\beta'}
& {\,\,\mathrm{K}_1(SF_2') \,\,} \ar[r]_-{\iota_*'}
& {\,\,\mathrm{K}_1(B)\,\,} \ar[r]^-{}
& {\,\,0\,\,},}
$$
and
by $M(A,B)$ the subset of $C(A,B)$ of all the commutative diagrams
$$
\xymatrixcolsep{2pc}
\xymatrix{
{\,\,0\,\,} \ar[r]^-{}
& {\,\,\mathrm{K}_0(A)\,\,} \ar[d]_-{0} \ar[r]^-{\pi_*}
& {\,\,\mathrm{K}_0(F_1)\,\,} \ar[d]_-{\mu_{0}} \ar[r]^-{\alpha-\beta}
& {\,\,\mathrm{K}_1(SF_2)\,\,} \ar[d]_-{\mu_{1}} \ar[r]^-{\iota_*}
& {\,\,\mathrm{K}_1(A)\,\,} \ar[d]_-{0} \ar[r]^-{}
& {\,\,0\,\,}\\
{\,\,0\,\,} \ar[r]^-{}
& {\,\,\mathrm{K}_0(B)\,\,} \ar[r]_-{\pi_*'}
& {\,\,\mathrm{K}_0(F_1') \,\,} \ar[r]_-{\alpha'-\beta'}
& {\,\,\mathrm{K}_1(SF_2') \,\,} \ar[r]_-{\iota_*'}
& {\,\,\mathrm{K}_1(B)\,\,} \ar[r]^-{}
& {\,\,0\,\,}}
$$
such that there exists $\mu\in {\bf Hom}(\mathrm{K}_1(SF_2), \mathrm{K}_0(F_1 '))$ satisfying $\mu_0=\mu \circ(\alpha-\beta)$, $\mu_1=(\alpha'-\beta')\circ \mu$. Since such a diagram is completely determined by $\mu$, we may denote it by $\lambda_\mu$.
\end{notion}
\begin{notion}\rm
For two commutative diagrams $\lambda_I,\,\lambda_{II}\in C(A,B)$,
$$
\xymatrixcolsep{2pc}
\xymatrix{
{\,\,0\,\,} \ar[r]^-{}
& {\,\,\mathrm{K}_0(A)\,\,} \ar[d]_-{\lambda_{I0*}} \ar[r]^-{\pi_*}
& {\,\,\mathrm{K}_0(F_1)\,\,} \ar[d]_-{\lambda_{I0}} \ar[r]^-{\alpha-\beta}
& {\,\,\mathrm{K}_1(SF_2)\,\,} \ar[d]_-{\lambda_{I1}} \ar[r]^-{\iota_*}
& {\,\,\mathrm{K}_1(A)\,\,} \ar[d]_-{\lambda_{I1*}} \ar[r]^-{}
& {\,\,0\,\,}\\
{\,\,0\,\,} \ar[r]^-{}
& {\,\,\mathrm{K}_0(B)\,\,} \ar[r]_-{\pi_*'}
& {\,\,\mathrm{K}_0(F_1') \,\,} \ar[r]_-{\alpha'-\beta'}
& {\,\,\mathrm{K}_1(SF_2') \,\,} \ar[r]_-{\iota_*'}
& {\,\,\mathrm{K}_1(B)\,\,} \ar[r]^-{}
& {\,\,0\,\,}}
$$
and
$$
\xymatrixcolsep{2pc}
\xymatrix{
{\,\,0\,\,} \ar[r]^-{}
& {\,\,\mathrm{K}_0(A)\,\,} \ar[d]_-{\lambda_{II0*}} \ar[r]^-{\pi_*}
& {\,\,\mathrm{K}_0(F_1)\,\,} \ar[d]_-{\lambda_{II0}} \ar[r]^-{\alpha-\beta}
& {\,\,\mathrm{K}_1(SF_2)\,\,} \ar[d]_-{\lambda_{II1}} \ar[r]^-{\iota_*}
& {\,\,\mathrm{K}_1(A)\,\,} \ar[d]_-{\lambda_{II1*}} \ar[r]^-{}
& {\,\,0\,\,}\\
{\,\,0\,\,} \ar[r]^-{}
& {\,\,\mathrm{K}_0(B)\,\,} \ar[r]_-{\pi_*'}
& {\,\,\mathrm{K}_0(F_1') \,\,} \ar[r]_-{\alpha'-\beta'}
& {\,\,\mathrm{K}_1(SF_2') \,\,} \ar[r]_-{\iota_*'}
& {\,\,\mathrm{K}_1(B)\,\,} \ar[r]^-{}
& {\,\,0\,\,},}
$$
define the sum of $\lambda_I$ and $\lambda_{II}$ as
$$
\xymatrixcolsep{2pc}
\xymatrix{
{\,\,0\,\,} \ar[r]^-{}
&{\,\,\mathrm{K}_0(A)\,\,} \ar[d]_-{\lambda_{I0*}+\lambda_{II0*}} \ar[r]^-{\pi_*}
& {\,\,\mathrm{K}_0(F_1)\,\,} \ar[d]_-{\lambda_{I0}+\lambda_{II0}} \ar[r]^-{\alpha-\beta}
& {\,\,\mathrm{K}_1(SF_2)\,\,} \ar[d]_-{\lambda_{I1}+\lambda_{II1}} \ar[r]^-{\iota_*}
& {\,\,\mathrm{K}_1(A)\,\,} \ar[d]_-{\lambda_{I1*}+\lambda_{II1*}} \ar[r]^-{}
& {\,\,0\,\,}\\
{\,\,0\,\,} \ar[r]^-{}
&{\,\,\mathrm{K}_0(B)\,\,} \ar[r]_-{\pi_*'}
& {\,\,\mathrm{K}_0(F_1') \,\,} \ar[r]_-{\alpha'-\beta'}
& {\,\,\mathrm{K}_1(SF_2') \,\,} \ar[r]_-{\iota_*'}
& {\,\,\mathrm{K}_1(B)\,\,} \ar[r]^-{}
& {\,\,0\,\,}.}
$$
Note that $\lambda_I+\lambda_{II}\in C(A,B)$.
The diagram
$$
\xymatrixcolsep{2pc}
\xymatrix{
{\,\,0\,\,} \ar[r]^-{}
&{\,\,\mathrm{K}_0(A)\,\,} \ar[d]_-{0} \ar[r]^-{\pi_*}
& {\,\,\mathrm{K}_0(F_1)\,\,} \ar[d]_-{0} \ar[r]^-{\alpha-\beta}
& {\,\,\mathrm{K}_1(SF_2)\,\,} \ar[d]_-{0} \ar[r]^-{\iota_*}
& {\,\,\mathrm{K}_1(A)\,\,} \ar[d]_-{0} \ar[r]^-{}
& {\,\,0\,\,}\\
{\,\,0\,\,} \ar[r]^-{}
&{\,\,\mathrm{K}_0(B)\,\,} \ar[r]_-{\pi_*'}
& {\,\,\mathrm{K}_0(F_1') \,\,} \ar[r]_-{\alpha'-\beta'}
& {\,\,\mathrm{K}_1(SF_2') \,\,} \ar[r]_-{\iota_*'}
& {\,\,\mathrm{K}_1(B)\,\,} \ar[r]^-{}
& {\,\,0\,\,},}
$$
to be denoted by 0, is the (unique) zero element of $C(A,B)$.
(Clearly, $\lambda+0=\lambda$ for $\lambda\in C(A,B)$.)

Given a commutative diagram $\lambda\in C(A,B)$,
$$
\xymatrixcolsep{2pc}
\xymatrix{
{\,\,0\,\,} \ar[r]^-{}
& {\,\,\mathrm{K}_0(A)\,\,} \ar[d]_-{\lambda_{0*}} \ar[r]^-{\pi_*}
& {\,\,\mathrm{K}_0(F_1)\,\,} \ar[d]_-{\lambda_{0}} \ar[r]^-{\alpha-\beta}
& {\,\,\mathrm{K}_1(SF_2)\,\,} \ar[d]_-{\lambda_{1}} \ar[r]^-{\iota_*}
& {\,\,\mathrm{K}_1(A)\,\,} \ar[d]_-{\lambda_{1*}} \ar[r]^-{}
& {\,\,0\,\,}\\
{\,\,0\,\,} \ar[r]^-{}
& {\,\,\mathrm{K}_0(B)\,\,} \ar[r]_-{\pi_*'}
& {\,\,\mathrm{K}_0(F_1') \,\,} \ar[r]_-{\alpha'-\beta'}
& {\,\,\mathrm{K}_1(SF_2') \,\,} \ar[r]_-{\iota_*'}
& {\,\,\mathrm{K}_1(B)\,\,} \ar[r]^-{}
& {\,\,0\,\,},}
$$
the inverse of $\lambda$, to be denoted by $-\lambda$, is
$$
\xymatrixcolsep{2pc}
\xymatrix{
{\,\,0\,\,} \ar[r]^-{}
& {\,\,\mathrm{K}_0(A)\,\,} \ar[d]_-{-\lambda_{0*}} \ar[r]^-{\pi_*}
& {\,\,\mathrm{K}_0(F_1)\,\,} \ar[d]_-{-\lambda_{0}} \ar[r]^-{\alpha-\beta}
& {\,\,\mathrm{K}_1(SF_2)\,\,} \ar[d]_-{-\lambda_{1}} \ar[r]^-{\iota_*}
& {\,\,\mathrm{K}_1(A)\,\,} \ar[d]_-{-\lambda_{1*}} \ar[r]^-{}
& {\,\,0\,\,}\\
{\,\,0\,\,} \ar[r]^-{}
& {\,\,\mathrm{K}_0(B)\,\,} \ar[r]_-{\pi_*'}
& {\,\,\mathrm{K}_0(F_1') \,\,} \ar[r]_-{\alpha'-\beta'}
& {\,\,\mathrm{K}_1(SF_2') \,\,} \ar[r]_-{\iota_*'}
& {\,\,\mathrm{K}_1(B)\,\,} \ar[r]^-{}
& {\,\,0\,\,}.}
$$
Note that $-\lambda\in C(A,B)$, and $\lambda+(-\lambda)=0$.

Then $C(A,B)$ is an Abelian group, and $M(A,B)$ is a subgroup of $C(A,B)$.
\end{notion}
\begin{notion}\label{defone}\rm
As pointed out in \cite{GLN:2015}, for
a minimal block $A=A(F_1,F_2,\varphi_0,\varphi_1)$, we have $\ker\varphi_0\cap\ker\varphi_1=\{0\}$.
Let $A\in \mathcal{C}$ be a minimal block.
Let us use  $\mathcal{C}_\mathcal{O}$ to denote the class of all unital ${\mathrm C}^*$-algebras
$A=A(F_1,\,F_2,\,\varphi_0,\varphi_1)$, where $F_2=M_r(\mathbb{C})$, for some integer $r$, and $\ker \varphi_0\oplus\ker \varphi_1=F_1$ (there is no block of $F_1$ mapping into both 0 and 1). This subclass was studied by Li in {\rm \cite{Li:2012}}. Note that $\widetilde{I}_p\in \mathcal{C}_\mathcal{O}\subsetneqq\mathcal{D}$.
\end{notion}
\begin{definition}\label{a dayu b}
For two matrices $\zeta$ and $\eta$, we say $\zeta\geq\eta$ if $\zeta-\eta$ has no negative entry.
\end{definition}
\begin{definition}\label{CP}
Let $A,\,B\in\mathcal{C}$ be minimal. Let $\lambda\in C(A,B)$:
$$
\xymatrixcolsep{2pc}
\xymatrix{
{\,\,0\,\,} \ar[r]^-{}
& {\,\,\mathrm{K}_0(A)\,\,} \ar[d]_-{\lambda_{0*}} \ar[r]^-{\pi_*}
& {\,\,\mathrm{K}_0(F_1)\,\,} \ar[d]_-{\lambda_{0}} \ar[r]^-{\alpha-\beta}
& {\,\,\mathrm{K}_1(SF_2)\,\,} \ar[d]_-{\lambda_{1}} \ar[r]^-{\iota_*}
& {\,\,\mathrm{K}_1(A)\,\,} \ar[d]_-{\lambda_{1*}} \ar[r]^-{}
& {\,\,0\,\,}\\
{\,\,0\,\,} \ar[r]^-{}
& {\,\,\mathrm{K}_0(B)\,\,} \ar[r]_-{\pi_*'}
& {\,\,\mathrm{K}_0(F_1') \,\,} \ar[r]_-{\alpha'-\beta'}
& {\,\,\mathrm{K}_1(SF_2') \,\,} \ar[r]_-{\iota_*'}
& {\,\,\mathrm{K}_1(B)\,\,} \ar[r]^-{}
& {\,\,0\,\,},}
$$
Let us say that $\lambda$ is positive or $\lambda\in C^+(A,B)$
if $\lambda$ is the zero element or ($\lambda_0\geq 0_{p'\times p}$ and $\lambda_0\neq0$).
And we say $\lambda$ is positive modulo $M(A,B)$, or that $\lambda+M(A,B)$ is positive,
if there exists $\lambda_\mu\in M(A,B)$,
$$
\xymatrixcolsep{2pc}
\xymatrix{
{\,\,0\,\,} \ar[r]^-{}
& {\,\,\mathrm{K}_0(A)\,\,} \ar[d]_-{0} \ar[r]^-{\pi_*}
& {\,\,\mathrm{K}_0(F_1)\,\,} \ar[d]_-{\mu_{0}} \ar[r]^-{\alpha-\beta}
& {\,\,\mathrm{K}_1(SF_2)\,\,} \ar[d]_-{\mu_{1}} \ar[r]^-{\iota_*}
& {\,\,\mathrm{K}_1(A)\,\,} \ar[d]_-{0} \ar[r]^-{}
& {\,\,0\,\,}\\
{\,\,0\,\,} \ar[r]^-{}
& {\,\,\mathrm{K}_0(B)\,\,} \ar[r]_-{\pi_*'}
& {\,\,\mathrm{K}_0(F_1') \,\,} \ar[r]_-{\alpha'-\beta'}
& {\,\,\mathrm{K}_1(SF_2') \,\,} \ar[r]_-{\iota_*'}
& {\,\,\mathrm{K}_1(B)\,\,} \ar[r]^-{}
& {\,\,0\,\,}}
$$
such that $\lambda+\lambda_\mu$ is a positive.
\end{definition}
\begin{remark}\rm
If $A=\oplus A^i,\,B=\oplus B^j$ with each $A^i$ and $B^j$ minimal
Elliott-Thomsen algebras, then we shall say that $\lambda+M(A,B)$ is
positive,
where $\lambda\in C(A,B)$ is determined by $\lambda_{ij}\in C(A^i,B^j)$, if $\lambda_{ij}+M(A^i,B^j)$ is positive for each $i,\,j$.
Let us
write
$$
(C(A,B)/M(A,B))^+=\{
\lambda+M(A,B):\,\lambda+M(A,B)\,\rm{is\,positive}\}
$$
for the positive cone of $C(A,B)/M(A,B)$.
\end{remark}
\begin{thrm}[\cite{AE:2017}]\label{CM}
Let $A$$,\,B \in \mathcal{C}$. Then we have a natural isomorphism of groups
$$
\mathrm{KK}(A,B)\cong C(A,B)/M(A,B).
$$
\end{thrm}
In this paper, we denote $\chi:\mathrm{KK}(A,B)\rightarrow C(A,B)/M(A,B)$ the natural isomorphism,
we will use $\mathrm{KK}(\lambda)$ to denote the $\mathrm{KK}$-element $\chi^{-1}(\lambda+M(A,B))\in\mathrm{KK}(A,B)$.
\begin{notion}
For $\lambda\in C(A,B)$
$$
\xymatrix{
{\,\,0\,\,} \ar[r]^-{}
& {\,\,\mathrm{K}_0(A)\,\,} \ar[d]_-{\lambda_{0*}} \ar[r]^-{\pi_*}
& {\,\,\mathrm{K}_0(F_1)\,\,} \ar[d]_-{\lambda_{0}} \ar[r]^-{\alpha-\beta}
& {\,\,\mathrm{K}_1(SF_2)\,\,} \ar[d]_-{\lambda_{1}} \ar[r]^-{\iota_*}
& {\,\,\mathrm{K}_1(A)\,\,} \ar[d]_-{\lambda_{1*}} \ar[r]^-{}
& {\,\,0\,\,}\\
{\,\,0\,\,} \ar[r]^-{}
& {\,\,\mathrm{K}_0(B)\,\,} \ar[r]_-{\pi_*'}
& {\,\,\mathrm{K}_0(F_1') \,\,} \ar[r]_-{\alpha'-\beta'}
& {\,\,\mathrm{K}_1(SF_2') \,\,} \ar[r]_-{\iota_*'}
& {\,\,\mathrm{K}_1(B)\,\,} \ar[r]^-{}
& {\,\,0\,\,}}
$$
and $\eta\in C(B,C)$
$$
\xymatrix{
{\,\,0\,\,} \ar[r]^-{}
& {\,\,\mathrm{K}_0(B)\,\,} \ar[d]_-{\eta_{0*}} \ar[r]^-{\pi_*'}
& {\,\,\mathrm{K}_0(F_1')\,\,} \ar[d]_-{\eta_{0}} \ar[r]^-{\alpha'-\beta'}
& {\,\,\mathrm{K}_1(SF_2')\,\,} \ar[d]_-{\eta_{1}} \ar[r]^-{\iota_*'}
& {\,\,\mathrm{K}_1(B)\,\,} \ar[d]_-{\eta_{1*}} \ar[r]^-{}
& {\,\,0\,\,}\\
{\,\,0\,\,} \ar[r]^-{}
& {\,\,\mathrm{K}_0(C)\,\,} \ar[r]_-{\pi_*''}
& {\,\,\mathrm{K}_0(F_1'') \,\,} \ar[r]_-{\alpha''-\beta''}
& {\,\,\mathrm{K}_1(SF_2'') \,\,} \ar[r]_-{\iota_*''}
& {\,\,\mathrm{K}_1(C)\,\,} \ar[r]^-{}
& {\,\,0\,\,},}
$$
define the product of $\lambda$ and $\eta$ as $\lambda\times\eta\in C(A,C)$:
$$
\xymatrix{
{\,\,0\,\,} \ar[r]^-{}
& {\,\,\mathrm{K}_0(A)\,\,} \ar[d]_-{\eta_{0*}\circ\lambda_{0*}} \ar[r]^-{\pi_*}
& {\,\,\mathrm{K}_0(F_1)\,\,} \ar[d]_-{\eta_{0}\circ\lambda_{0}} \ar[r]^-{\alpha-\beta}
& {\,\,\mathrm{K}_1(SF_2)\,\,} \ar[d]_-{\eta_{1}\circ\lambda_{1}} \ar[r]^-{\iota_*}
& {\,\,\mathrm{K}_1(A)\,\,} \ar[d]_-{\eta_{1*}\circ\lambda_{1*}} \ar[r]^-{}
& {\,\,0\,\,}\\
{\,\,0\,\,} \ar[r]^-{}
& {\,\,\mathrm{K}_0(C)\,\,} \ar[r]_-{\pi_*''}
& {\,\,\mathrm{K}_0(F_1'') \,\,} \ar[r]_-{\alpha''-\beta''}
& {\,\,\mathrm{K}_1(SF_2'') \,\,} \ar[r]_-{\iota_*''}
& {\,\,\mathrm{K}_1(C)\,\,} \ar[r]^-{}
& {\,\,0\,\,}.}
$$
In fact, the natural product we defined above, exactly induces the Kasparov product on the $\mathrm{KK}$-groups which is isomorphic to the quotient groups (Theorem \ref{CM}).
\end{notion}
\begin{thrm}[\cite{AE:2017}]\label{c_olift}
  Assume that $A\in \mathcal{C}_\mathcal{O}$ and $B\in \mathcal{C}$. Then we have (as order groups, with the natural maps)
$$\mathrm{KK}(A,B)\cong C(A,B)/M(A,B)\cong{\bf Hom}_\Lambda(\underline{\mathrm{K}}(A),\underline{\mathrm{K}}(B)).$$
\end{thrm}
\begin{lem}[\cite{AE:2017}]\label{SUFF yuan}
Let $A,\,B \in \mathcal{C}$ be minimal. Let the diagram $\lambda\in C(A,B)$,
$$
\xymatrixcolsep{2pc}
\xymatrix{
{\,\,0\,\,} \ar[r]^-{}
& {\,\,\mathrm{K}_0(A)\,\,} \ar[d]_-{\lambda_{0*}} \ar[r]^-{\pi_*}
& {\,\,\mathrm{K}_0(F_1)\,\,} \ar[d]_-{\lambda_{0}} \ar[r]^-{\alpha-\beta}
& {\,\,\mathrm{K}_1(SF_2)\,\,} \ar[d]_-{\lambda_{1}} \ar[r]^-{\iota_*}
& {\,\,\mathrm{K}_1(A)\,\,} \ar[d]_-{\lambda_{1*}} \ar[r]^-{}
& {\,\,0\,\,}\\
{\,\,0\,\,} \ar[r]^-{}
& {\,\,\mathrm{K}_0(B)\,\,} \ar[r]_-{\pi_*'}
& {\,\,\mathrm{K}_0(F_1') \,\,} \ar[r]_-{\alpha'-\beta'}
& {\,\,\mathrm{K}_1(SF_2') \,\,} \ar[r]_-{\iota_*'}
& {\,\,\mathrm{K}_1(B)\,\,} \ar[r]^-{}
& {\,\,0\,\,},}
$$
be given, such that $\lambda$ is positive, then, $\lambda_0$ is positive.
If for any $i\in\{1,2,\cdots,p\},\,j'\in\{1,2,\cdots,l'\}$,
$$
(\alpha'\circ\lambda_0)_{j'i}\geq
\sum_{\alpha_{ji}\cdot\lambda_{j'j}^1\geq0}
\alpha_{ji}\cdot\lambda_{j'j}^1
-\sum_{\beta_{ji}\cdot\lambda_{j'j}^1\leq0}
\beta_{ji}\cdot\lambda_{j'j}^1
$$
and
$$
(\beta'\circ\lambda_0)_{j'i}\geq
-\sum_{\alpha_{ji}\cdot\lambda_{j'j}^1\leq0}
\alpha_{ji}\cdot\lambda_{j'j}^1
+\sum_{\beta_{ji}\cdot\lambda_{j'j}^1\geq0}
\beta_{ji}\cdot\lambda_{j'j}^1,
$$
then there is a homomorphism from $A$ to $M_r(B)$ for some integer $r$ inducing
the diagram $\lambda$.
\end{lem}
\begin{thrm}[\cite{AE:2017}]\label{CRS1}
Consider the case $A=C(S^1)$ (not in $\mathcal{C}_\mathcal{O}$), $B \in \mathcal{D}$ is minimal.
Then
$\gamma \in \mathrm{KK}(C(S^1),\,B)$ can be lifted to a homomorphism if and only if $\chi(\alpha)\in C(A,B)/M(A,B)$ is
positive.
\end{thrm}
\begin{thrm}[Corollary 4.3 in \cite{AE:2017}]\label{ONEORDER}
Assume that $A,\,B \in \mathcal{D}$ are $minimal$. If $\gamma$ is a $\mathrm{KK}$-element in $\mathrm{KK}(A,B)$ satisfying
$$
\gamma(\underline{\mathrm{K}}^+(A))\subset \underline{\mathrm{K}}^+(B),
$$
then $\chi(\gamma)$ is positive, where $\chi$ is the natural map from $\mathrm{KK}(A,B)$ to $C(A,B)/M(A,B)$.
\end{thrm}
Now we will introduce a key lemma for the existence result.
As the case of $A=M_m(\mathbb{C})$ is trivial, we will only concern the case of $A(F_1,M_n(\mathbb{C}),\varphi_0,\varphi_1)$. Before we list the lemma, we give one more notation.
\begin{notion}\label{naturally rep}\rm
Let $A(F_1,M_n(\mathbb{C}),\varphi_0,\varphi_1)$ be a minimal block in $\mathcal{D}$. Then $\alpha-\beta$ is a $1\times p$ matrix.
For any $1\times p$ matrix $(t_1,t_2,\cdots,t_p)\geq0$ (the matrix has no negative entry), denote representation $R(t_1,t_2,\cdots,t_p)$, where
$$
R(t_1,t_2,\cdots,t_p)(f,a)={\rm diag}\{\underbrace{a(\theta_1),\,\cdots \,,a(\theta_1)}_{t_1},\cdots\,,\underbrace{a(\theta_p),\,\cdots \,,a(\theta_p)}_{t_p}\},\,\,\,\forall (f,a)\in A.
$$
Note that both $\alpha$ and $\beta$ has no negative entry and then correspond
to classes of representations of $A$, whose spectral points are contained in $Sp(F_1)$.
\end{notion}
Note that the condition required in Lemma \ref{SUFF yuan}, is sometimes quite strong, and we need a weaker one to help us to construct homomorphisms in this paper.
\begin{lem}\label{D^0lift lemma}
Let $A(F_1,M_n(\mathbb{C}),\varphi_0,\varphi_1),\,B(F_1',M_{n'}(\mathbb{C}),\varphi_0',\varphi_1')$ be minimal blocks in $\mathcal{D}$. $\lambda\in C^+(A,B)$ (Definition \ref{CP}) is the following commutative diagram
$$
\xymatrixcolsep{2pc}
\xymatrix{
{\,\,0\,\,} \ar[r]^-{}
& {\,\,\mathrm{K}_0(A)\,\,} \ar[d]_-{\lambda_{0*}} \ar[r]^-{\pi_*}
& {\,\,\mathrm{K}_0(F_1)\,\,} \ar[d]_-{\lambda_{0}} \ar[r]^-{\alpha-\beta}
& {\,\,\mathbb{Z}\,\,} \ar[d]_-{\lambda_{1}} \ar[r]^-{\iota_*}
& {\,\,\mathrm{K}_1(A)\,\,} \ar[d]_-{\lambda_{1*}} \ar[r]^-{}
& {\,\,0\,\,}\\
{\,\,0\,\,} \ar[r]^-{}
& {\,\,\mathrm{K}_0(B)\,\,} \ar[r]_-{\pi_*'}
& {\,\,\mathrm{K}_0(F_1') \,\,} \ar[r]_-{\alpha'-\beta'}
& {\,\,\mathbb{Z} \,\,} \ar[r]_-{\iota_*'}
& {\,\,\mathrm{K}_1(B)\,\,} \ar[r]^-{}
& {\,\,0\,\,}.}
$$
If one of the following holds:

(1)~$\lambda_1>0$ and
$\alpha'\lambda_0-\alpha-l(\alpha-\beta)\geq0$, for all $l=0,1,2,\cdots,\lambda_1-1$,

(2)~$\lambda_1=0$,

(3)~$\lambda_1<0$ and $\beta'\lambda_0-\beta-l(\beta-\alpha)\geq0$, for all $l=0,1,2,\cdots,-\lambda_1-1$,
\\
then there is a homomorphism from $A$
to $M_\bullet(B)$ realizing $\lambda$.
\end{lem}
\begin{proof}
At first we give the proof for (1).
As $\alpha'\lambda_0-\alpha-l(\alpha-\beta)\geq0$, for all $l=0,1,2,\cdots,\lambda_1-1$, then each $\alpha'\lambda_0-\alpha-l(\alpha-\beta)$ naturally corresponds a representation by \ref{naturally rep}.
Let us define $\lambda_1$ maps $\Lambda_l,\,l=1,2,\cdots,\lambda_1$ from $A$ to
$M_k(C[(l-1)/\lambda_1,l/\lambda_1])$ for some integer $k$.
$$
A\ni f\stackrel{\Lambda_l}{\longmapsto}\zeta_l\in M_k(C[(l-1)/\lambda_1,l/\lambda_1]),$$
where
$$
\zeta_l(t)={\rm diag}\{f(\lambda_1 t-(l-1)),R(\alpha'\lambda_0-(l-1)(\alpha-\beta)-\alpha)(f,a)\},\,\,\forall\, t\in [(l-1)/\lambda_1,l/\lambda_1].$$

Note that $\Lambda_l(\cdot)((l-1)/\lambda_1)$ and $\Lambda_l(\cdot)(l/\lambda_1)$
are two representations of $A$ and induce two elements in $\mathrm{K}^0(F_1)\cong\mathrm{KK}(F_0,\mathbb{C})$, that is
$$
[\Lambda_l(\cdot)((l-1)/\lambda_1)]=\alpha'\lambda_0-(l-1)(\alpha-\beta)\quad
{\rm and}
\quad
[\Lambda_l(\cdot)(l/\lambda_1)]=\alpha'\lambda_0-l(\alpha-\beta).
$$
Now we have
$$
[\Lambda_1(\cdot)(0)]=\alpha'\lambda_0\quad
{\rm and}
\quad
[\Lambda_{\lambda_1}(\cdot)(1)]=\alpha'\lambda_0-\lambda_1(\alpha-\beta)=\beta'\lambda_0~~(\lambda~~{\rm is~~commutative})
$$
and
$$
[\Lambda_{l-1}(\cdot)((l-1)/\lambda_1)]=\alpha'\lambda_0-(l-1)(\alpha-\beta)=[\Lambda_l(\cdot)((l-1)/\lambda_1)],\quad
l=2,3,\cdots,\lambda_1
$$
(as class in $\mathrm{K}^0(F_1)$).
Then for $l=2,3,\cdots,\lambda_1$,
there is a unitary $u_l\in M_k(\mathbb(C))$, such that
$$
{\rm Ad}\,u_l\circ \Lambda_l(\cdot)((l-1)/\lambda_1)=\Lambda_{l-1}(\cdot)((l-1)/\lambda_1).
$$
Define a homomorphism $\Lambda:\,A\to M_k(C[0,1])$,
$$
A\ni(f,a)\stackrel{\Lambda}{\longmapsto}\zeta\in M_k(C[0,1])
$$
where
$$
\zeta(t)={\rm Ad}\,(u_l\cdot u_{l-1} \cdot \cdots \cdot u_2)\circ \zeta_l(t),\quad t\in[(l-1)/\lambda_1,l/\lambda_1].
$$
Then by Lemma 3.5 in \cite{AE:2017} (see also \cite{GLN:2015}),
there is a unitary $u\in M_k(C[0,1])$, such that ${\rm Ad}\,u\circ \Lambda$ is a homomorphism from
$A$ to $M_\bullet(B)$.

(2) is trivial.

For (3), let us define $-\lambda_1$ maps $\Lambda_l,\,l=1,2,\cdots,-\lambda_1$ from $A$ to
$M_k(C[(l-1)/-\lambda_1,l/-\lambda_1])$ for some integer $k$.
$$
A\ni f\stackrel{\Lambda_l}{\longmapsto}\zeta_l\in M_k(C[(l-1)/-\lambda_1,l/-\lambda_1]),$$
where
$$
\zeta_l(t)={\rm diag}\{f(l+\lambda_1 t),R(\beta'\lambda_0-\beta-(l-1)(\beta-\alpha))(f,a)\}.$$
And the rest is as same as we give in (1).
\end{proof}

\begin{remark}\label{counter example}
We should mention that even for $A,\,B\in \mathcal{D}$, we may still have
$$\mathrm{KK}^+(A,B)\subsetneqq{\bf Hom}_\Lambda^+(\underline{\mathrm{K}}(A),\underline{\mathrm{K}}(B)),$$
and there is an example shown in \cite{AE:2017}.
We list a new one here, which we will also discuss later.

Let $F_1=\mathbb{C}\oplus\mathbb{C}\oplus\mathbb{C}\oplus\mathbb{C}\oplus\mathbb{C}$, $F_2=M_3(\mathbb{C})$,
$$\varphi_0(a\oplus b\oplus c\oplus d\oplus e)={\rm{diag}}\{a,b,e\},
\quad
\varphi_1(a\oplus b\oplus c\oplus d\oplus e)={\rm{diag}}\{c,d,e\},
$$
$F_1'=\mathbb{C}\oplus\mathbb{C}\oplus\mathbb{C}\oplus\mathbb{C}$, $F_2'=M_2(\mathbb{C})$,
$$\varphi_0'(a\oplus b\oplus c\oplus d)={\rm{diag}}\{a,b\},
\quad
\varphi_1'(a\oplus b\oplus c\oplus d)={\rm{diag}}\{c,d\}.
$$
Set
$A=A(F_1,F_2,\varphi_0,\varphi_1)\in \mathcal{D}$,
$B=B(F_1',F_2',\varphi_0',\varphi_1')\in \mathcal{D}$. Then we have
$$\alpha=(1,1,0,0,1)
\quad
{\rm and}
\quad
\beta=(0,0,1,1,1),$$
$$\alpha'=(1,1,0,0)
\quad
{\rm and}
\quad
\beta'=(0,0,1,1).$$

\noindent Consider the following commutative diagram $\lambda$ in $C(A,B)$:
$$
\xymatrixcolsep{2pc}
\xymatrix{
{\,\,0\,\,} \ar[r]^-{}
& {\,\,\mathrm{K}_0(A)\,\,} \ar[d]_-{} \ar[r]^-{\pi_*}
& {\,\,\mathrm{K}_0(F_1)\,\,} \ar[d]_-{\lambda_{0}} \ar[r]^-{(1,1,-1,-1,0)}
& {\,\,\mathrm{K}_1(SF_2)\,\,} \ar[d]_-{1} \ar[r]^-{\iota_*}
& {\,\,\mathrm{K}_1(A)\,\,} \ar[d]_-{} \ar[r]^-{}
& {\,\,0\,\,}\\
{\,\,0\,\,} \ar[r]^-{}
& {\,\,\mathrm{K}_0(B)\,\,} \ar[r]_-{\pi_*'}
& {\,\,\mathrm{K}_0(F_1') \,\,} \ar[r]_-{(1,1,-1,-1)}
& {\,\,\mathrm{K}_1(SF_2') \,\,} \ar[r]_-{\iota_*'}
& {\,\,\mathrm{K}_1(B)\,\,} \ar[r]^-{}
& {\,\,0\,\,},}
$$
where
$$
\lambda_0=
\left(
  \begin{array}{ccccccc}
    1&  &  & & 0      \\
     & 1 &  & & 0    \\
     &  & 1 & & 0   \\
     &  &  & 1 & 0    \\

  \end{array}
  \right).
$$
Denote the related $\mathrm{KK}$-element by $\gamma$. With Theorem \ref{c_olift} and Theorem \ref{CRS1}, one can check that $\gamma$ satifies
$$
\gamma(\underline{\mathrm{K}}^+(A))\subset \underline{\mathrm{K}}^+(B).
$$
(Here, we mention that when one is checking $\gamma(\mathrm{K}_*^+(A))\subset \mathrm{K}_*^+(B)$, the composed diagram
$\varepsilon\times\lambda$ is not 0 but belongs to $M(C(S^1),B)$, where $\varepsilon\in C^+(C(S^1),A)$ is as follows:
$$
\xymatrixcolsep{2pc}
\xymatrix{
{\,\,0\,\,} \ar[r]^-{}
& {\,\,\mathrm{K}_0(C(S^1))\,\,} \ar[d]_-{} \ar[r]^-{}
& {\,\,\mathbb{Z}\,\,} \ar[d]_-{e} \ar[r]^-{1-1=0}
& {\,\,\mathbb{Z}\,\,} \ar[d]_-{1} \ar[r]^-{}
& {\,\,\mathrm{K}_1(C(S^1))\,\,} \ar[d]_-{} \ar[r]^-{}
& {\,\,0\,\,}\\
{\,\,0\,\,} \ar[r]^-{}
& {\,\,\mathrm{K}_0(A)\,\,} \ar[r]_-{\pi_*}
& {\,\,\mathrm{K}_0(F_1) \,\,} \ar[r]_-{(1,1,-1,-1,0)}
& {\,\,\mathbb{Z} \,\,} \ar[r]_-{\iota_*}
& {\,\,\mathrm{K}_1(A)\,\,} \ar[r]^-{}
& {\,\,0\,\,},}
$$
where $e=(0,0,0,0,1)^{\rm T}$.)
We find that $\lambda$ is the unique positive element in $\chi(\gamma)$, but $\lambda$ can not be lifted (By Corollary 3.7 in \cite{AE:2017}). Then $\gamma$ can not be lifted (see \ref{bianjietongtai}).

Note that
$$
\gamma_*(e)=0,
$$
where $e=(0,0,0,0,1)\in K_0^+(A).$
If one tries to use the condition
$$
\gamma(\underline{\mathrm{K}}^+(A)\setminus\{0\})\subset \underline{\mathrm{K}}^+(B)\setminus\{0\},
$$
instead of
$$
\gamma(\underline{\mathrm{K}}^+(A))\subset \underline{\mathrm{K}}^+(B),
$$
one will get $\gamma$ liftable for the $A,\,B$ we constructed above and this will be helpful in Elliott classification program for simple amenable $C^*$-algebras(see \cite{LN:2008}). But we don't need it for the case of inductive limits of real rank zero in this paper.
\end{remark}

\section{Existence Theorem}
From the example listed in Remark \ref{counter example}, we showed that for the blocks we concerned, a $\mathrm{KK}$-element preserving Dadarlat-Loring order may not be liftable. But this doesn't mean we do not have a existence theorem for the inductive limit algebra of real rank zero. In this section, we will show that the condition of real rank zero will help us to get it.

Interestingly, for the blocks, we have the following existence theorem for the ``0'' case:
\begin{lem}\label{0 existence theorem}
Let $A,\,B$ be minimal blocks in $\mathcal{D}$, if $\gamma\in \mathrm{KK}(A,B)$ satisfies that
$$
\gamma(\underline{\mathrm{K}}^+(A))\subset \underline{\mathrm{K}}^+(B)\quad
{\rm and}\quad
\gamma_*[1_A]=0,
$$
then $\gamma=0$.
\end{lem}
\begin{proof}
Here, we only consider the case of both $A$ and $B$ are not finite dimensional, because those cases are easier.
As $\gamma(\underline{\mathrm{K}}^+(A))\subset \underline{\mathrm{K}}^+(B),$ by Theorem \ref{ONEORDER}, there is a commutative diagram $\lambda\in C^+(A,B)$
$$
\xymatrixcolsep{2pc}
\xymatrix{
{\,\,0\,\,} \ar[r]^-{}
& {\,\,\mathrm{K}_0(A)\,\,} \ar[d]_-{\lambda_{0*}} \ar[r]^-{\pi_*}
& {\,\,\mathrm{K}_0(F_1)\,\,} \ar[d]_-{\lambda_{0}} \ar[r]^-{\alpha-\beta}
& {\,\,\mathbb{Z}\,\,} \ar[d]_-{\lambda_{1}} \ar[r]^-{\iota_*}
& {\,\,\mathrm{K}_1(A)\,\,} \ar[d]_-{\lambda_{1*}} \ar[r]^-{}
& {\,\,0\,\,}\\
{\,\,0\,\,} \ar[r]^-{}
& {\,\,\mathrm{K}_0(B)\,\,} \ar[r]_-{\pi_*'}
& {\,\,\mathrm{K}_0(F_1') \,\,} \ar[r]_-{\alpha'-\beta'}
& {\,\,\mathbb{Z} \,\,} \ar[r]_-{\iota_*'}
& {\,\,\mathrm{K}_1(B)\,\,} \ar[r]^-{}
& {\,\,0\,\,},}
$$
such that $\mathrm{KK}(\lambda)=\gamma$.

We assume the notational convention that
$$A=A(F_1,M_n(\mathbb{C}),\varphi_0,\varphi_1)\quad
{\rm with}\quad F_1= \bigoplus_{i=1}^p M_{k_i}(\mathbb{C}),$$
then $[1_A]=(k_1,k_2,\cdots,k_p)$.
Since $\gamma_*[1_A]=0$ which means
$\lambda_0(k_1,k_2,\cdots,k_p)^{\rm T}=(0,0,\cdots,0)^{\rm T}$.
Note that $\lambda_0\geq 0_{p'\times p}$, so we have $\lambda_0=0_{p'\times p}$.
Then
$$
(0,0,\cdots,0)=\lambda_0(\alpha'-\beta')=\lambda_1(\alpha-\beta).
$$

Case.\,1. If $\lambda_1=0$, we get what we want.

Case.\,2. If $\lambda_1\neq0$, we have $\alpha-\beta=0$.
Construct an element $\xi\in C^+(C(S^1),A)$
$$
\xymatrixcolsep{2pc}
\xymatrix{
{\,\,0\,\,} \ar[r]^-{}
& {\,\,\mathrm{K}_0(C(S^1))\,\,} \ar[d]_-{} \ar[r]^-{}
& {\,\,\mathbb{Z}\,\,} \ar[d]_-{e} \ar[r]^-{1-1=0}
& {\,\,\mathbb{Z}\,\,} \ar[d]_-{1} \ar[r]^-{}
& {\,\,\mathrm{K}_1(C(S^1))\,\,} \ar[d]_-{} \ar[r]^-{}
& {\,\,0\,\,}\\
{\,\,0\,\,} \ar[r]^-{}
& {\,\,\mathrm{K}_0(A)\,\,} \ar[r]_-{\pi_*}
& {\,\,\mathrm{K}_0(F_1) \,\,} \ar[r]_-{\alpha-\beta}
& {\,\,\mathbb{Z} \,\,} \ar[r]_-{\iota_*}
& {\,\,\mathrm{K}_1(A)\,\,} \ar[r]^-{}
& {\,\,0\,\,},}
$$
where $e=(1,0,\cdots,0)^{\rm T}$. By Theorem \ref{CRS1}, $\xi$ induces an element in $\mathrm{K}_*^+(A)$.
Then $\gamma\times\mathrm{KK}(\xi)$ which is exactly the $\mathrm{KK}$-element corresponds to $\xi\times\lambda\in C^+(C(S^1,B)$
$$
\xymatrixcolsep{2pc}
\xymatrix{
{\,\,0\,\,} \ar[r]^-{}
& {\,\,\mathrm{K}_0(C(S^1))\,\,} \ar[d]_-{} \ar[r]^-{}
& {\,\,\mathbb{Z}\,\,} \ar[d]_-{0=\lambda_0 e} \ar[r]^-{1-1=0}
& {\,\,\mathbb{Z}\,\,} \ar[d]_-{\lambda_1} \ar[r]^-{}
& {\,\,\mathrm{K}_1(C(S^1))\,\,} \ar[d]_-{} \ar[r]^-{}
& {\,\,0\,\,}\\
{\,\,0\,\,} \ar[r]^-{}
& {\,\,\mathrm{K}_0(B)\,\,} \ar[r]_-{\pi_*'}
& {\,\,\mathrm{K}_0(F_1') \,\,} \ar[r]_-{\alpha'-\beta'}
& {\,\,\mathbb{Z} \,\,} \ar[r]_-{\iota_*'}
& {\,\,\mathrm{K}_1(B)\,\,} \ar[r]^-{}
& {\,\,0\,\,}}
$$
which can be lifted. By Theorem \ref{CRS1}, there exists a map $\mu\in {\bf Hom}(\mathbb{Z}, \mathrm{K}_0(F_1 '))$,
such that $\lambda_1=(\alpha'-\beta')\mu$. Then $\lambda\in M(A,B)$, i.e., $\gamma=0$.
\end{proof}
\begin{remark}
We list an example shown in \cite{DL2:1996} to show that the condition $\gamma(\underline{\mathrm{K}}^+(A))\subset \underline{\mathrm{K}}^+(B)$ we gave in Theorem \ref{0 existence theorem} is necessary.
Consider the following two commutative diagram $\lambda_1,\,\lambda_2\in C^+(\widetilde{I}_p,\mathbb{C})$ ($p\geq2$):
$$
\xymatrixcolsep{2pc}
\xymatrix{
{\,\,0\,\,} \ar[r]^-{}
& {\,\,\mathrm{K}_0(\widetilde{I}_p)\,\,} \ar[d]_-{} \ar[r]^-{}
& {\,\,\mathbb{Z}\,\,} \ar[d]_-{(1,0)} \ar[r]^-{(p,-p)}
& {\,\,\mathbb{Z}\,\,} \ar[d]_-{} \ar[r]^-{}
& {\,\,\mathrm{K}_1(\widetilde{I}_p)\,\,} \ar[d]_-{} \ar[r]^-{}
& {\,\,0\,\,}\\
{\,\,0\,\,} \ar[r]^-{}
& {\,\,\mathrm{K}_0(\mathbb{C})\,\,} \ar[r]_-{}
& {\,\,\mathbb{Z} \,\,} \ar[r]_-{}
& {\,\,0 \,\,} \ar[r]_-{}
& {\,\,0 \,\,} \ar[r]^-{}
& {\,\,0\,\,}}
$$
$$
\xymatrixcolsep{2pc}
\xymatrix{
{\,\,0\,\,} \ar[r]^-{}
& {\,\,\mathrm{K}_0(\widetilde{I}_p)\,\,} \ar[d]_-{} \ar[r]^-{}
& {\,\,\mathbb{Z}\,\,} \ar[d]_-{(0,1)} \ar[r]^-{(p,-p)}
& {\,\,\mathbb{Z}\,\,} \ar[d]_-{} \ar[r]^-{}
& {\,\,\mathrm{K}_1(\widetilde{I}_p)\,\,} \ar[d]_-{} \ar[r]^-{}
& {\,\,0\,\,}\\
{\,\,0\,\,} \ar[r]^-{}
& {\,\,\mathrm{K}_0(\mathbb{C})\,\,} \ar[r]_-{}
& {\,\,\mathbb{Z} \,\,} \ar[r]_-{}
& {\,\,0 \,\,} \ar[r]_-{}
& {\,\,0 \,\,} \ar[r]^-{}
& {\,\,0\,\,},}
$$
which are induced by two different point valued representations of $\widetilde{I}_p$.
Note that $\mathrm{KK}(\lambda_1)\neq \mathrm{KK}(\lambda_2)$, but $(\mathrm{KK}(\lambda_1)-\mathrm{KK}(\lambda_2))_*[1_{\widetilde{I}_p}]=0.$
\end{remark}
The following theorem is a corollary of the universal coefficient theorem.
\begin{thrm}[\cite{DG:1997}]\label{UCT}
Let $A,\,B$ be ${\mathrm C}^*$-algebras. Suppose that $A\in \mathcal{N}$ (where $\mathcal{N}$ is
the ``bootstrap'' category defined in \cite{RS:1987}), $\mathrm{K}_*(A)$ is finitely generated, and $B$ is $\sigma$-unital.
Then the natural map
$$\Gamma:\,\mathrm{KK}(A,B)\to {\bf Hom}_\Lambda(\underline{\mathrm{K}}(A),\underline{\mathrm{K}}(B))$$ is a group isomorphism.
\end{thrm}
\begin{prop}{\rm (\cite[Proposition 4.13]{DG:1997})}\label{dg finite}
Let $A$ be $\mathrm{C}^*$-algebra in the class $\mathcal{N}$ of \cite{RS:1987}. If the group $\mathrm{K}_*(A)$ is finitely
generated, then $\underline{\mathrm{K}}(A)$ is finitely generated as $\Lambda$-module.
That means there are finitely many elements $x_1,\cdots,x_r\in\underline{\mathrm{K}}(A)$ such that for any
$x\in\underline{\mathrm{K}}(A)$ there exist $\lambda_i\in \Lambda$ and $k_i\in \mathbb{Z}$ such that $x=\sum_{i=1}^r k_i \lambda_i(x_i)$.
\end{prop}
Proposition \ref{dg finite} shows that for $A\in\mathcal{D}$, $\underline{\mathrm{K}}(A)$ is finitely generated. Further more, we show that
${\bf Hom}^+_\Lambda(\underline{\mathrm{K}}(A),\underline{\mathrm{K}}(B))$ is also determined by a finite set
of $\underline{\mathrm{K}}^+(A)$.
\begin{thrm}\label{d_0 F lift}
For any $\mathrm{C}^*$-algebra $A\in\mathcal{D}$ there exists a finite subset $F$ of $\underline{\mathrm{K}}^+(A)$ such that if $B\in\mathcal{D}$ and $\gamma \in \mathrm{KK}(A,B)\cong {\bf Hom}_\Lambda(\underline{\mathrm{K}}(A),\underline{\mathrm{K}}(B))$ satisfies that
$$\gamma_*(F)\subset \underline{\mathrm{K}}^+(B),$$
then $\gamma$ preserves Dadarlat-Loring order (\ref{Def DL}).
\end{thrm}
\begin{proof}
As $A\in \mathcal{D}$, we assume that $A$ is minimal but not finite dimensional. (The case of $A=M_m(\mathbb{C})$ only needs to concern $\mathrm{K}_0^+(A)$, which is much easier.) Recall that in \ref{A^0}, we have
$$
\alpha-\beta=(a_1,a_2,\cdots,a_r,-b_1,-b_2,\cdots,-b_l,0,0,\cdots,0).
$$
Now we construct the finite subset $F\subset\underline{\mathrm{K}}^+(A)$ for $A$.

For any $x \in\{1,2,\cdots,r\}$ and $y \in\{1,2,\cdots,l\}$,
let $w_{xy}=a_x\cdot b_y\in\mathbb{N}$. By Theorem \ref{c_olift}, there exists a homomorphism $\eta_{xy}$ from $\widetilde{I}_{w_{xy}}$ to $M_r(A)$ inducing the following commutative diagram $\lambda_{xy}\in C(\widetilde{I}_{w_{xy}},A)$:

$$
\xymatrixcolsep{2pc}
\xymatrix{
{\,\,0\,\,} \ar[r]^-{}
& {\,\,\mathrm{K}_0(\widetilde{I}_{w_{xy}})\,\,} \ar[d]_-{} \ar[r]^-{}
& {\,\,\mathbb{Z}\oplus\mathbb{Z}\,\,} \ar[d]_-{\lambda_0 ^{xy}} \ar[r]^-{(w_{xy},-w_{xy})}
& {\,\,\mathbb{Z}\,\,} \ar[d]_-{1} \ar[r]^-{}
& {\,\,\mathrm{K}_1(\widetilde{I}_{w_{xy}})\,\,} \ar[d]_-{} \ar[r]^-{}
& {\,\,0\,\,}\\
{\,\,0\,\,} \ar[r]^-{}
& {\,\,\mathrm{K}_0(A)\,\,} \ar[r]_-{\pi_*}
& {\,\,\mathbb{Z}^p \,\,} \ar[r]_-{\alpha-\beta}
& {\,\,\mathbb{Z} \,\,} \ar[r]_-{\iota_*}
& {\,\,\mathrm{K}_1(A)\,\,} \ar[r]^-{}
& {\,\,0\,\,}}
$$
where
\begin{equation}\lambda_0^{xy}=
\left(
\begin{array}{ccc}
0&\vdots\\[2 mm]
b_y&0\\[2 mm]
0&a_x\\[2 mm]
\vdots&0
\end{array}
\right)
\end{equation}
(the $p\times2$ matrix with all entry 0 except for $(x,1)^{\rm th}$ entry $b_y$ and $(r+y,2)^{\rm th}$ entry $a_x$).
Here, we use $\mathrm{KK}(\lambda)$ to denote the $\mathrm{KK}$-element induced by a commutative diagram $\lambda$, (which is the natural map listed in Theorem \ref{CM}).

Note that $$\mathrm{KK}(\lambda_{xy})\in \mathrm{KK}^+(\widetilde{I}_{w_{xy}},A)\subset \underline{\mathrm{K}}^+(A).$$
Set
$$F=\{\,\mathrm{KK}(\lambda_{xy}):\, x=1,\cdots,r, \,y=1,\cdots,l\}\cup
\{\,\mathrm{KK}(\xi_i)\in \mathrm{K}^+_*(A):\, r+l+1\leq i\leq p\},
$$
where $\mathrm{KK}_*^+(A)\triangleq\mathrm{KK}^+(C(S^1),A)$ (Definition \ref{liftable})
and $\xi_i\in C(C(S^1),A)$ is the following diagram
$$
\xymatrixcolsep{2pc}
\xymatrix{
{\,\,0\,\,} \ar[r]^-{}
& {\,\,\mathrm{K}_0(C(S^1))\,\,} \ar[d]_-{} \ar[r]^-{}
& {\,\,\mathbb{Z}\,\,} \ar[d]_-{e_i^{\rm T}} \ar[r]^-{1-1=0}
& {\,\,\mathbb{Z}\,\,} \ar[d]_-{1} \ar[r]^-{}
& {\,\,\mathrm{K}_1(C(S^1))\,\,} \ar[d]_-{} \ar[r]^-{}
& {\,\,0\,\,}\\
{\,\,0\,\,} \ar[r]^-{}
& {\,\,\mathbb{Z}^p\,\,} \ar[r]_-{\pi_*}
& {\,\,\mathrm{K}_0(F_1) \,\,} \ar[r]_-{\alpha-\beta}
& {\,\,\mathbb{Z} \,\,} \ar[r]_-{\iota_*}
& {\,\,\mathrm{K}_1(A)\,\,} \ar[r]^-{}
& {\,\,0\,\,},}
$$
$e_i=(0,0,\cdots,0,1,0,\cdots,0)$, 1 is at the $i^{\rm th}$ position.

With the completely same calculation of the proof of Theorem \ref{ONEORDER}, there is a diagram
$\lambda\in C^+(A,B)$
$$
\xymatrixcolsep{2pc}
\xymatrix{
{\,\,0\,\,} \ar[r]^-{}
& {\,\,\mathrm{K}_0(A)\,\,} \ar[d]_-{\lambda_{0*}} \ar[r]^-{\pi_*}
& {\,\,\mathrm{K}_0(F_1)\,\,} \ar[d]_-{\lambda_{0}} \ar[r]^-{\alpha-\beta}
& {\,\,\mathbb{Z}\,\,} \ar[d]_-{\lambda_{1}} \ar[r]^-{\iota_*}
& {\,\,\mathrm{K}_1(A)\,\,} \ar[d]_-{\lambda_{1*}} \ar[r]^-{}
& {\,\,0\,\,}\\
{\,\,0\,\,} \ar[r]^-{}
& {\,\,\mathrm{K}_0(B)\,\,} \ar[r]_-{\pi_*'}
& {\,\,\mathrm{K}_0(F_1') \,\,} \ar[r]_-{\alpha'-\beta'}
& {\,\,\mathbb{Z} \,\,} \ar[r]_-{\iota_*'}
& {\,\,\mathrm{K}_1(B)\,\,} \ar[r]^-{}
& {\,\,0\,\,}}
$$
such that $\mathrm{KK}(\lambda)=\gamma$.

Note that for any $A\in\mathcal{C},\,\mathrm{K}_0(A)$ is torsion free.
It follows that for any ${\mathrm C}^*$-algebras $A,\,B\in \mathcal{D}$, the requirement
$$
\gamma(\underline{\mathrm{K}}^+(A))\subset \underline{\mathrm{K}}^+(B)
$$
is equivalent to
$$
\gamma(\mathrm{K}_0 ^+(A;\mathbb{Z}\oplus\mathbb{Z}_p))\subset \mathrm{K}_0 ^+(B;\mathbb{Z}\oplus\mathbb{Z}_p),\quad p\geq2,\quad
\rm{and}
\quad
\gamma(\mathrm{K}_*^+(A))\subset \mathrm{K}_*^+(B),
$$
where $$(\mathrm{K}_0 (\bullet;\mathbb{Z}\oplus\mathbb{Z}_p),\mathrm{K}_0 ^+(\bullet;\mathbb{Z}\oplus\mathbb{Z}_p))\cong (\mathrm{KK}(\widetilde{I}_p,\bullet),\mathrm{KK}^+(\widetilde{I}_p,\bullet)).$$

By Theorem \ref{c_olift}, we have $\gamma(\mathrm{K}_0 ^+(A;\mathbb{Z}\oplus\mathbb{Z}_p))\subset \mathrm{K}_0 ^+(B;\mathbb{Z}\oplus\mathbb{Z}_p)$. (As for any $\zeta\in C^+(\widetilde{I}_p,A)$, we have $\zeta\times \lambda\in C^+(\widetilde{I}_p,B)$.)

For the case of $0\neq\zeta\in C^+(C(S^1),A)$
$$
\xymatrixcolsep{2pc}
\xymatrix{
{\,\,0\,\,} \ar[r]^-{}
& {\,\,\mathrm{K}_0(C(S^1))\,\,} \ar[d]_-{} \ar[r]^-{}
& {\,\,\mathbb{Z}\,\,} \ar[d]_-{\zeta_0} \ar[r]^-{1-1=0}
& {\,\,\mathbb{Z}\,\,} \ar[d]_-{\zeta_1} \ar[r]^-{}
& {\,\,\mathrm{K}_1(C(S^1))\,\,} \ar[d]_-{} \ar[r]^-{}
& {\,\,0\,\,}\\
{\,\,0\,\,} \ar[r]^-{}
& {\,\,\mathrm{K}_0(A)\,\,} \ar[r]_-{\pi_*}
& {\,\,\mathrm{K}_0(F_1) \,\,} \ar[r]_-{\alpha-\beta}
& {\,\,\mathbb{Z} \,\,} \ar[r]_-{\iota_*}
& {\,\,\mathrm{K}_1(A)\,\,} \ar[r]^-{}
& {\,\,0\,\,},}
$$
we only need to consider $\zeta_1=1$.

(1)\,If $\lambda_0\zeta_0\neq(0,0,\cdots,0)^{\rm T}$, by Theorem \ref{CRS1}, we have $\zeta\times\lambda$ lifted.

(2)\,If $\lambda_0\zeta_0=(0,0,\cdots,0)^{\rm T}$ and $\lambda_1=0$, we have $\zeta\times\lambda=0$.

(3)\,If $\lambda_0\zeta_0=(0,0,\cdots,0)^{\rm T}$ and $\lambda_1\neq0$,
note that  none of the first $r+l$ rows of
$\lambda_0$ is $(0,0,\cdots,0)^{\rm T}$ ($\lambda$ is a commutative diagram), i.e., we have $$\zeta_0=\sum_{i=r+l+1}^{p}\rho_i e_i^{\rm T}=(0,\cdots,0,\rho_{r+l+1},\cdots,\rho_p)^{\rm T}, \quad\rho_i\geq0.$$ Then there exists $\rho_{i_0}>0$, for some $i_0$ ($0\neq \zeta$).
As $\xi_{i_0}\times\lambda$ can be lifted by assumption and $(\zeta-\xi_{i_0})\times\lambda$ can be lifted as a homomorphism with finite dimensional image (by Theorem \ref{SUFF yuan} or one can also easily construct the homomorphism as $\zeta_1-1=0$).
Then $\zeta\times\lambda$ can be lifted.

In summary, $\gamma$ preserves Dadarlat-Loring order.
\end{proof}
\begin{remark}
We should give an explanation that in the last part of the proof above, we divide the case of $\zeta\in C^+(C(S^1),A)$ into
3 cases instead of just using Theorem \ref{CRS1}.
It is because sometimes the condition $\lambda$ is positive may not imply $\zeta\times\lambda$ is positive.
And in particular, $\mathrm{KK}(\zeta\times\lambda)$ can be lifted while $\zeta\times\lambda$ may not. We list an example here.

Let $F_1=\mathbb{C}\oplus\mathbb{C}\oplus\mathbb{C}\oplus\mathbb{C}\oplus\mathbb{C}$, $F_2=M_5(\mathbb{C})$,
$$\varphi_0(a\oplus b\oplus c\oplus d\oplus e)={\rm{diag}}\{a,a,b,b,e\},
\quad
\varphi_1(a\oplus b\oplus c\oplus d\oplus e)={\rm{diag}}\{c,c,d,d,e\},
$$
$F_1'=\mathbb{C}\oplus\mathbb{C}\oplus\mathbb{C}\oplus\mathbb{C}$, $F_2'=M_4(\mathbb{C})$,
$$\varphi_0'(a\oplus b\oplus c\oplus d)={\rm{diag}}\{a,a,b,b\},
\quad
\varphi_1'(a\oplus b\oplus c\oplus d)={\rm{diag}}\{c,c,d,d\}.
$$
Set
$A=A(F_1,F_2,\varphi_0,\varphi_1)\in \mathcal{D}$,
$B=B(F_1',F_2',\varphi_0',\varphi_1')\in \mathcal{D}$. Then we have
$$\alpha=(2,2,0,0,1)
\quad
{\rm and}
\quad
\beta=(0,0,2,2,1),$$
$$\alpha'=(2,2,0,0)
\quad
{\rm and}
\quad
\beta'=(0,0,2,2).$$

\noindent Consider the following commutative diagram $\delta$ in $C(A,B)$:
$$
\xymatrixcolsep{2pc}
\xymatrix{
{\,\,0\,\,} \ar[r]^-{}
& {\,\,\mathrm{K}_0(A)\,\,} \ar[d]_-{} \ar[r]^-{\pi_*}
& {\,\,\mathrm{K}_0(F_1)\,\,} \ar[d]_-{\delta_{0}} \ar[r]^-{(2,2,-2,-2,0)}
& {\,\,\mathrm{K}_1(SF_2)\,\,} \ar[d]_-{1} \ar[r]^-{\iota_*}
& {\,\,\mathrm{K}_1(A)\,\,} \ar[d]_-{} \ar[r]^-{}
& {\,\,0\,\,}\\
{\,\,0\,\,} \ar[r]^-{}
& {\,\,\mathrm{K}_0(B)\,\,} \ar[r]_-{\pi_*'}
& {\,\,\mathrm{K}_0(F_1') \,\,} \ar[r]_-{(2,2,-2,-2)}
& {\,\,\mathrm{K}_1(SF_2') \,\,} \ar[r]_-{\iota_*'}
& {\,\,\mathrm{K}_1(B)\,\,} \ar[r]^-{}
& {\,\,0\,\,},}
$$
where
$$
\delta_0=
\left(
  \begin{array}{ccccccc}
    1&  &  & & 0      \\
     & 1 &  & & 0    \\
     &  & 1 & & 0   \\
     &  &  & 1 & 0    \\

  \end{array}
  \right).
$$

Consider $\zeta\in C^+(S^1,A)$
$$
\xymatrixcolsep{2pc}
\xymatrix{
{\,\,0\,\,} \ar[r]^-{}
& {\,\,\mathrm{K}_0(C(S^1))\,\,} \ar[d]_-{} \ar[r]^-{}
& {\,\,\mathbb{Z}\,\,} \ar[d]_-{\zeta_0} \ar[r]^-{1-1=0}
& {\,\,\mathbb{Z}\,\,} \ar[d]_-{1} \ar[r]^-{}
& {\,\,\mathrm{K}_1(C(S^1))\,\,} \ar[d]_-{} \ar[r]^-{}
& {\,\,0\,\,}\\
{\,\,0\,\,} \ar[r]^-{}
& {\,\,\mathrm{K}_0(A)\,\,} \ar[r]_-{\pi_*}
& {\,\,\mathrm{K}_0(F_1) \,\,} \ar[r]_-{(2,2,-2,-2,0)}
& {\,\,\mathbb{Z} \,\,} \ar[r]_-{\iota_*}
& {\,\,\mathrm{K}_1(A)\,\,} \ar[r]^-{}
& {\,\,0\,\,},}
$$
where $\zeta_0=(0,0,0,0,1)^{\rm T}$.
Note that $\zeta\times\delta\in C(S^1,B)$
$$
\xymatrixcolsep{2pc}
\xymatrix{
{\,\,0\,\,} \ar[r]^-{}
& {\,\,\mathrm{K}_0(C(S^1))\,\,} \ar[d]_-{} \ar[r]^-{}
& {\,\,\mathbb{Z}\,\,} \ar[d]_-{0} \ar[r]^-{1-1=0}
& {\,\,\mathbb{Z}\,\,} \ar[d]_-{1} \ar[r]^-{}
& {\,\,\mathrm{K}_1(C(S^1))\,\,} \ar[d]_-{} \ar[r]^-{}
& {\,\,0\,\,}\\
{\,\,0\,\,} \ar[r]^-{}
& {\,\,\mathrm{K}_0(B)\,\,} \ar[r]_-{\pi_*}
& {\,\,\mathrm{K}_0(F_1) \,\,} \ar[r]_-{(2,2,-2,-2)}
& {\,\,\mathbb{Z} \,\,} \ar[r]_-{\iota_*}
& {\,\,\mathrm{K}_1(B)\,\,} \ar[r]^-{}
& {\,\,0\,\,}}
$$
is not positive and it is easily seen that $\mathrm{KK}(\zeta\times\delta)$ can not be lifted. In particular, we do not have $\zeta\times\delta\notin M(C(S^1),B)$, either. (It is because 1 is not in the image of $(2,2,-2,2)$.)

We mention that this Example is different from that given in Remark \ref{counter example}. The difference comes from that 1 is in the image of $(1,1,-1,1)$ but not in the image of $(2,2,-2,2)$.
And in this case, we have
$${\bf Hom}^+_\Lambda(\underline{\mathrm{K}}(A),\underline{\mathrm{K}}(B))\subsetneqq
(C(A,B)/M(A,B))^+.
$$
\end{remark}
\begin{remark}\label{homomorphismunital}\rm
Let $A,\,B\in \mathcal{C}$. $\phi$ is a homomorphism from $A$
to $M_r(B)$ constructed as that in Lemma \ref{D^0lift lemma}, inducing the commutative diagram $\lambda\in C^+(A,B)$:
$$
\xymatrixcolsep{2pc}
\xymatrix{
{\,\,0\,\,} \ar[r]^-{}
& {\,\,\mathrm{K}_0(A)\,\,} \ar[d]_-{\lambda_{0*}} \ar[r]^-{\pi_*}
& {\,\,\mathrm{K}_0(F_1)\,\,} \ar[d]_-{\lambda_{0}} \ar[r]^-{\alpha-\beta}
& {\,\,\mathrm{K}_1(SF_2)\,\,} \ar[d]_-{\lambda_{1}} \ar[r]^-{\iota_*}
& {\,\,\mathrm{K}_1(A)\,\,} \ar[d]_-{\lambda_{1*}} \ar[r]^-{}
& {\,\,0\,\,}\\
{\,\,0\,\,} \ar[r]^-{}
& {\,\,\mathrm{K}_0(B)\,\,} \ar[r]_-{\pi_*'}
& {\,\,\mathrm{K}_0(F_1') \,\,} \ar[r]_-{\alpha'-\beta'}
& {\,\,\mathrm{K}_1(SF_2') \,\,} \ar[r]_-{\iota_*'}
& {\,\,\mathrm{K}_1(B)\,\,} \ar[r]^-{}
& {\,\,0\,\,}.}
$$
If
$$
\mathrm{KK}(\phi)([1_A])\leq[1_B],
$$
then
$$
\lambda_0(k_1,k_2,\cdots,k_p)\leq(k_1',k_2',\cdots,k_{p'}').
$$
which means that we can use Lemma 3.5 in \cite{AE:2017} more carefully to get a new
a homomorphism $\psi:\,A\to B$, inducing $\lambda$. In particular, if
$$
\lambda_0(k_1,k_2,\cdots,k_p)=(k_1',k_2',\cdots,k_{p'}'),
$$
the homomorphism can be also
constructed unital.
\end{remark}

We recall the following definition:
\begin{definition}[\cite{EfK:1986}]
Let $A=\underrightarrow{\lim}(A_n,\phi_{n,m})$ and $B=\underrightarrow{\lim}(B_n,\psi_{n,m})$. These two inductive limit systems
are said to be shape equivalent if there are sequences $\{k_i\},\,\{l_i\}$, and
$\xi_1:\,A_{k_i}\to B_{l_i}$ and maps $\eta_i:\,B_{l_i}\to A_{k_{i+1}}$ such that
$
\eta_i\circ\xi_i:\,A_{k_i}\to A_{k_i+1}
$
is homotopic to
$
\phi_{k_i,k_{i+1}}:\,A_{k_i}\to A_{k_i+1},
$
and
$
\xi_{i+1}\circ\eta_i:\,B_{l_i}\to B_{l_{i+1}}
$
is homotopic to
$
\psi_{l_i,l_{i+1}}:\,B_{l_i}\to B_{l_{i+1}}.
$
\end{definition}
From the example in Remark 3.14 in \cite{AE:2017} (also Example \ref{stable hom}), we know that two homomorphisms with the same $\mathrm{KK}$-class may not be homotopic, but after adding a same homomorphism, they can be homotopic to each other. So we list the following definitions for stably homotopic and  $\mathrm{KK}$-shape equivalent.
\begin{definition}
Let $A,\,B$ be $C^*$-algebras and $\phi,\,\psi$ be two homomorphisms from $A$ to $B$.
We say $\phi$ and $\psi$ are stably homotopic, if there exists a homomorphism $\eta:\,A\to M_k(B)$, for some integer $k$
such that
$$\phi\oplus\eta\sim_h\psi\oplus\eta,$$
i.e., $\phi\oplus\eta$ and $\psi\oplus\eta$ are homotopic as homomorphisms from $A$ to $M_{k+1}(B)$.
\end{definition}
\begin{definition}[\cite{DG:1997}]
Let $A=\underrightarrow{\lim}(A_n,\phi_{n,m})$ and $B=\underrightarrow{\lim}(B_n,\psi_{n,m})$. These two inductive limit systems
are said to be $\mathrm{KK}$-shape equivalent if there are sequences of homomorphisms $\{k_i\},\,\{l_i\}$, and homomorphisms
$\xi_1:\,A_{k_i}\to B_{l_i}$ and homomorphisms $\eta_i:\,B_{l_i}\to A_{k_{i+1}}$ such that
$$
\mathrm{KK}(\eta_i\circ\xi_i)=\mathrm{KK}(\phi_{k_i,k_{i+1}})\in\mathrm{KK}(A_{k_i},A_{k_i+1}),
$$
and
$$
\mathrm{KK}(\xi_{i+1}\circ\eta_i)=\mathrm{KK}(\psi_{l_i,l_{i+1}})\in\mathrm{KK}(B_{l_i}, B_{l_{i+1}}).
$$
\end{definition}

\begin{notion}\label{decomposion we need}\rm
Let $A=A(F_1,\,M_n(\mathbb{C}),\,\varphi_0,\,\varphi_1),\,B=B(F_1',\,M_{n'}(\mathbb{C}),\,\varphi_0',\,\varphi_1')$ be in $\mathcal{D}$, $\psi$ be a homomorphism from $A$ to $B$ with finite dimensional image, whose spectrum points are contained in $(0,1)\subset Sp(A).$
For convenience, we need to give a description for $\psi$. As $\psi$ has finite dimensional image,
we have a finite dimensional algebra $M_{u_1}(\mathbb{C})\oplus M_{u_2}(\mathbb{C})\oplus\cdots\oplus M_{u_\bullet}(\mathbb{C})$
isomorphic to the image of $\psi$. Then the isomorphism induces an embedding homomorphism $\iota:\,M_{u_1}(\mathbb{C})\oplus M_{u_2}(\mathbb{C})\oplus\cdots\oplus M_{u_\bullet}(\mathbb{C})\to B$ such that
the following diagram
$$
\xymatrixcolsep{3pc}
\xymatrix{
{\,\,A\,\,} \ar[r]^-{\psi} \ar[rd]_-{\widetilde{\psi}}
& {\,\,B\,\,} \\
{\,\,\,\,}
& {\,\,M_{u_1}(\mathbb{C})\oplus M_{u_2}(\mathbb{C})\oplus\cdots\oplus M_{u_\bullet}(\mathbb{C})\,\,} \ar[u]^-{\iota}}
$$
is commutative.
Here we can choose $\iota$ carefully (by conjugate a unitary in $M_{u_1}(\mathbb{C})\oplus M_{u_2}(\mathbb{C})\oplus\cdots\oplus M_{u_\bullet}(\mathbb{C})$) such that
$$
\pi_i\circ\widetilde{\psi}(f,a)=f(y_i),
$$
where
$\pi_i$ is the $i^{th}$ projective map from $M_{u_1}(\mathbb{C})\oplus M_{u_2}(\mathbb{C})\oplus\cdots\oplus M_{u_\bullet}(\mathbb{C})$ to $M_{u_i}(\mathbb{C})$, and $y_i\in(0,1)$.
Here we regard each $M_{u_i}(\mathbb{C})$ as a sub-algebra of $M_{u_1}(\mathbb{C})\oplus M_{u_2}(\mathbb{C})\oplus\cdots\oplus M_{u_\bullet}(\mathbb{C})$.
In fact, $u_i=n$ for all $i=1,2,\cdots,\bullet$.

We concern the homomorphism $\iota\circ\pi_i\circ\widetilde{\psi}$,
we push the spectral point $y_i$ to $0$ and get a new homomorphism denoted by $\psi_i$. (Similarly, we can also push it to 1.)
Then we have $\mathrm{KK}(\iota\circ\pi_i\circ\widetilde{\psi})=\mathrm{KK}(\psi_i)$.
Note that $\psi_i$ exactly induces the following diagram $\lambda_i\in C^+(A,B)$
$$
\xymatrixcolsep{2pc}
\xymatrix{
{\,\,0\,\,} \ar[r]^-{}
& {\,\,\mathrm{K}_0(A)\,\,} \ar[d]_-{} \ar[r]^-{\pi_*}
& {\,\,\mathrm{K}_0(F_1)\,\,} \ar[d]_-{\lambda_{0}^i} \ar[r]^-{\alpha-\beta}
& {\,\,\mathbb{Z}\,\,} \ar[d]_-{0} \ar[r]^-{\iota_*}
& {\,\,\mathrm{K}_1(A)\,\,} \ar[d]_-{0} \ar[r]^-{}
& {\,\,0\,\,}\\
{\,\,0\,\,} \ar[r]^-{}
& {\,\,\mathrm{K}_0(B)\,\,} \ar[r]_-{\pi_*'}
& {\,\,\mathrm{K}_0(F_1') \,\,} \ar[r]_-{\alpha'-\beta'}
& {\,\,\mathbb{Z} \,\,} \ar[r]_-{\iota_*'}
& {\,\,\mathrm{K}_1(B)\,\,} \ar[r]^-{}
& {\,\,0\,\,}}
$$
where $\lambda_0^i=(g_1^i,g_2^i,\cdots,g_{p'}^i)^{\rm T}\cdot\alpha$, $$(g_1^i,g_2^i,\cdots,g_{p'}^i)=\frac{[\iota(1_{u_i})]}{n}\in\mathrm{K}_0(B)\subset\mathrm{K}_0(F_1').$$

Then $\mathrm{KK}(\psi)=\sum_{i=1}^\bullet\mathrm{KK}(\psi_i)=\sum_{i=1}^\bullet\mathrm{KK}(\lambda_i)$
and
$$[\psi(1_A)]=\sum_{i=1}^\bullet[\iota(1_{u_i})]=n\sum_{i=1}^\bullet (g_1^i,g_2^i,\cdots,g_{p'}^i).$$
\end{notion}

We have shown in Remark \ref{counter example} (also Example 5.7 in \cite{AE:2017}) that a $\mathrm{KK}$-element preserving Dadarlat-Loring order may not be liftable, which means we do not have a existence result directly for the building blocks. We list our main result for the existence theorem here.
\begin{thrm}\label{composed existence}
Let $A,\,B$ and $C$ be minimal blocks in $\mathcal{D}$,
$\psi$ be a homomorphism from $A$ to $B$.
$\gamma$ is a $\mathrm{KK}$-element in $\mathrm{KK}(B,C)$ preserving Dadarlat-Loring order.
$$
A\xrightarrow{\psi}B\xrightarrow{\gamma} C
$$
If $\psi=\psi_{F_1}\oplus\psi_{(0,1)}\oplus\psi_r$, where $\psi_{F_1}$ and $\psi_{(0,1)}$ are homomorphisms with finite dimensional images, whose spectrum are contained in $Sp(F_1)$ and $(0,1)\subset Sp(A)$, respectively, such that
$$
[\psi_{(0,1)}(1_A)]\geq [\psi_r(1_A)].
$$
Then $\mathrm{KK}(\psi)\times \gamma$ can be lifted as a homomorphism.
\end{thrm}

\begin{proof}
If one of $A,\,B$ and $C$ is finite dimensional, it is easy to lift $\mathrm{KK}(\psi)\times \gamma$ to a homomorphism with finite finite dimensional image.
Also note that $\mathrm{KK}(\psi_{F_1})\times \gamma $ is always liftable,
as $\gamma$ preserves Dadarlat-Loring order.
Assume that $A,\,B$ and $C$ are not finite dimensional and just for convenience, we also assume $\psi_{F_1}=0$.
We will still use the notation in \ref{notion of A B}, and in addition, we assume
$$C=C(F_1'',M_{n''}(\mathbb{C}),\varphi_0 '',\varphi_1 '')$$
with $$
F_1 ''= \bigoplus_{i''=1}^{p''} M_{{k''}_{i''}}(\mathbb{C}).$$
From \ref{bianjietongtai}, there is a homomorphism $\psi_r'$ ($\psi_r'\sim_h\psi_r$) inducing $\eta\in C^+(A,B)$
$$
\xymatrixcolsep{2pc}
\xymatrix{
{\,\,0\,\,} \ar[r]^-{}
& {\,\,\mathrm{K}_0(A)\,\,} \ar[d]_-{} \ar[r]^-{\pi_*}
& {\,\,\mathrm{K}_0(F_1)\,\,} \ar[d]_-{\eta_{0}} \ar[r]^-{\alpha-\beta}
& {\,\,\mathbb{Z}\,\,} \ar[d]_-{\eta_{1}} \ar[r]^-{\iota_*}
& {\,\,\mathrm{K}_1(A)\,\,} \ar[d]_-{} \ar[r]^-{}
& {\,\,0\,\,}\\
{\,\,0\,\,} \ar[r]^-{}
& {\,\,\mathrm{K}_0(B)\,\,} \ar[r]_-{\pi_*'}
& {\,\,\mathrm{K}_0(F_1') \,\,} \ar[r]_-{\alpha'-\beta'}
& {\,\,\mathbb{Z} \,\,} \ar[r]_-{\iota_*'}
& {\,\,\mathrm{K}_1(B)\,\,} \ar[r]^-{}
& {\,\,0\,\,}}
$$
and from Theorem \ref{ONEORDER}, there is a diagram $\lambda\in C^+(B,C)$
$$
\xymatrixcolsep{2pc}
\xymatrix{
{\,\,0\,\,} \ar[r]^-{}
& {\,\,\mathrm{K}_0(B)\,\,} \ar[d]_-{} \ar[r]^-{\pi_*'}
& {\,\,\mathrm{K}_0(F_1')\,\,} \ar[d]_-{\lambda_{0}} \ar[r]^-{\alpha'-\beta'}
& {\,\,\mathbb{Z}\,\,} \ar[d]_-{\lambda_{1}} \ar[r]^-{\iota_*'}
& {\,\,\mathrm{K}_1(B)\,\,} \ar[d]_-{} \ar[r]^-{}
& {\,\,0\,\,}\\
{\,\,0\,\,} \ar[r]^-{}
& {\,\,\mathrm{K}_0(C)\,\,} \ar[r]_-{\pi_*''}
& {\,\,\mathrm{K}_0(F_1'') \,\,} \ar[r]_-{\alpha''-\beta''}
& {\,\,\mathbb{Z} \,\,} \ar[r]_-{\iota_*''}
& {\,\,\mathrm{K}_1(C)\,\,} \ar[r]^-{}
& {\,\,0\,\,},}
$$
such that $\mathrm{KK}(\lambda)=\gamma$.

Case.\,1.\, If $\lambda_1\eta_1>0$, we push all the spectrum points of $\psi_{(0,1)}$ to 0 and get
a new homomorphism $\psi_{(0,1)}'$ ($\psi_{(0,1)}'\sim_h\psi_{(0,1)}$), inducing a diagram $\zeta\in C^+(A,B)$.
$$
\xymatrixcolsep{2pc}
\xymatrix{
{\,\,0\,\,} \ar[r]^-{}
& {\,\,\mathrm{K}_0(A)\,\,} \ar[d]_-{} \ar[r]^-{\pi_*}
& {\,\,\mathrm{K}_0(F_1)\,\,} \ar[d]_-{\zeta_{0}} \ar[r]^-{\alpha-\beta}
& {\,\,\mathbb{Z}\,\,} \ar[d]_-{0} \ar[r]^-{\iota_*}
& {\,\,\mathrm{K}_1(A)\,\,} \ar[d]_-{} \ar[r]^-{}
& {\,\,0\,\,}\\
{\,\,0\,\,} \ar[r]^-{}
& {\,\,\mathrm{K}_0(B)\,\,} \ar[r]_-{\pi_*'}
& {\,\,\mathrm{K}_0(F_1') \,\,} \ar[r]_-{\alpha'-\beta'}
& {\,\,\mathbb{Z} \,\,} \ar[r]_-{\iota_*'}
& {\,\,\mathrm{K}_1(B)\,\,} \ar[r]^-{}
& {\,\,0\,\,}}
$$
Recall that by \ref{A^0}, $\alpha'-\beta'=(a_1',a_2,'\cdots,a_{r'}',-b_1',-b_2,'\cdots,-b_{l'}',0,0,\cdots,0).$
Note that $\lambda_1\neq0$, $(\alpha'-\beta')\lambda_0=\lambda_1(\alpha''-\beta'')$.
Then, the ${i'}^{\rm th}$ row of $\lambda_0$ is not $(0,0,\cdots,0)^{\rm T}$ for all $1\leq i'\leq r'+l'$.
Assume that $\lambda_0$ has the form of $(\Lambda,0_{p''\times t})$,
where $\Lambda$ is a matrix without $(0,0,\cdots,0)^{\rm T}$ rows,
$0\leq t\leq p'-r'-l'$.
That is, all the $(0,0,\cdots,0)^{\rm T}$ of $\lambda_0$ are in the last $t$ rows. Then denote
$$
E=\{e\in \mathrm{K}_0^+(B)\mid e=\sum_{i'=p'-t+1}^{p'}m_{i'}e_{i'},\, m_i\in \mathbb{N}\cup\{0\}\},
$$
where $e_{i'}=(0,0,\cdots,0,1,0,\cdots,0)$, 1 is in the ${i'}^{\rm th}$ entry.

If $[\psi_{(0,1)}(1_A)]\in E$, from the assumption, we have $[\psi_{r}(1_A)]\in E$.
Then $[\psi(1_A)]\in E$. As
$$(\gamma_*[(\psi(1_A)])^{\rm T}=\lambda_0 [\psi(1_A)]^{\rm T}=(0,0,\cdots,0)^{\rm T},$$ and
$\mathrm{KK}(\psi)\times \gamma$ is a $\mathrm{KK}$-element in $\mathrm{KK}(A,C)$ preserving Dadarlat-Loring order,
by Lemma \ref{0 existence theorem}, we have $\mathrm{KK}(\psi)\times \gamma=0$.

If $[\psi_{(0,1)}(1_A)]=[\psi'_{(0,1)}(1_A)]\notin E$,  from \ref{decomposion we need}, there is a $g_0\in \mathrm{K}_0^+(B)$ and
$g_0\notin E$ such that
$$
\zeta_0\geq g_0^{\rm T}\alpha\geq 0_{p'\times p}.
$$
And note that $g_0\notin E$ implies
$$
(\gamma_*g_0)^{\rm T}=\lambda_0g_0^{\rm T}\neq(0,0,\cdots,0)^{\rm T}.
$$
Then
$$
\alpha''\lambda_0\zeta_0\geq \alpha''\lambda_0g_0^{\rm T}\alpha.
$$
Since $(\lambda_0g_0^{\rm T})^{\rm T}$ is a non-zero element in $\mathrm{K}_0^+(C)$, we have
$$\alpha''\lambda_0g_0^{\rm T}=\beta''\lambda_0g_0^{\rm T}\geq1,$$
which implies
$$
\alpha''\lambda_0\zeta_0\geq \alpha''\lambda_0g_0^{\rm T}\alpha\geq\alpha.
$$
On one hand, we have considered the $\mathrm{KK}$-element $\mathrm{KK}(\psi_{(0,1)})\times \gamma\in \mathrm{KK}(A,C)$ ,
and on the other hand we should consider $\mathrm{KK}(\psi_r)\times \gamma$. Let us concern the diagram $\eta\times\lambda\in C^+(A,C)$:
$$
\xymatrixcolsep{2pc}
\xymatrix{
{\,\,0\,\,} \ar[r]^-{}
& {\,\,\mathrm{K}_0(A)\,\,} \ar[d]_-{} \ar[r]^-{\pi_*}
& {\,\,\mathrm{K}_0(F_1)\,\,} \ar[d]_-{\lambda_0\eta_{0}} \ar[r]^-{\alpha-\beta}
& {\,\,\mathbb{Z}\,\,} \ar[d]_-{\lambda_1\eta_1} \ar[r]^-{\iota_*}
& {\,\,\mathrm{K}_1(A)\,\,} \ar[d]_-{} \ar[r]^-{}
& {\,\,0\,\,}\\
{\,\,0\,\,} \ar[r]^-{}
& {\,\,\mathrm{K}_0(C)\,\,} \ar[r]_-{\pi_*''}
& {\,\,\mathrm{K}_0(F_1'') \,\,} \ar[r]_-{\alpha''-\beta''}
& {\,\,\mathbb{Z} \,\,} \ar[r]_-{\iota_*''}
& {\,\,\mathrm{K}_1(C)\,\,} \ar[r]^-{}
& {\,\,0\,\,}.}
$$
As
$$
(\alpha''-\beta'')\lambda_0\eta_0=\lambda_1\eta_1(\alpha-\beta),
$$
we have
$$
\alpha''\lambda_0\eta_0-\lambda_1\eta_1(\alpha-\beta)=\beta''\lambda_0\eta_0\geq0.
$$
Note that
$$\alpha-\beta=(a_1,a_2,\cdots ,a_r,-b_1,-b_2,\cdots ,-b_l,0,0,\cdots ,0) \quad{\rm (see\, \,\ref{A^0})}.$$
This implies
$$
\alpha''\lambda_0\eta_0-j(\alpha-\beta)\geq0,\quad j=0,1,2,\cdots,\lambda_1\eta_1.
$$
Now the commutative diagram $(\zeta+\eta)\times\lambda\in C^+(A,C)$
$$
\xymatrixcolsep{2pc}
\xymatrix{
{\,\,0\,\,} \ar[r]^-{}
& {\,\,\mathrm{K}_0(A)\,\,} \ar[d]_-{} \ar[r]^-{\pi_*}
& {\,\,\mathrm{K}_0(F_1)\,\,} \ar[d]_-{\lambda_0(\zeta_0+\eta_{0})} \ar[r]^-{\alpha-\beta}
& {\,\,\mathbb{Z}\,\,} \ar[d]_-{\lambda_1\eta_1} \ar[r]^-{\iota_*}
& {\,\,\mathrm{K}_1(A)\,\,} \ar[d]_-{} \ar[r]^-{}
& {\,\,0\,\,}\\
{\,\,0\,\,} \ar[r]^-{}
& {\,\,\mathrm{K}_0(C)\,\,} \ar[r]_-{\pi_*''}
& {\,\,\mathrm{K}_0(F_1'') \,\,} \ar[r]_-{\alpha''-\beta''}
& {\,\,\mathbb{Z} \,\,} \ar[r]_-{\iota_*''}
& {\,\,\mathrm{K}_1(C)\,\,} \ar[r]^-{}
& {\,\,0\,\,}}
$$
inducing the $\mathrm{KK}$-element $\mathrm{KK}(\psi)\times \gamma$,
satisfies the condition (1) in Lemma \ref{D^0lift lemma}, then can be lifted as a homomorphism.

Case.\,2.\,If $\lambda_1\eta_1<0$, we push all the spectrum points of $\psi_{(0,1)}$ to 1 (not 0) and get
a new homomorphism. The rest calculation is as same as we have done in Case.\,1 and we will use the condition (3) in Lemma \ref{D^0lift lemma} at last.

Case.\,3.\,If $\lambda_1\eta_1=0$, we push all the spectrum points of $\psi_{(0,1)}$ to 0 or 1 and get a commutative diagram $\zeta$. It is very easy to lift the commutative diagram $(\zeta+\eta)\times\lambda\in C^+(A,C)$ as a homomorphism with finite dimensional image (by Lemma \ref{SUFF yuan}).

In summary, the $\mathrm{KK}$-element $\mathrm{KK}(\psi)\times \gamma$ can be lifted.
\end{proof}
Then we have the following result about shape theory:
\begin{thrm}\label{exist0}
Let $A=\underrightarrow{\lim}(A_n,\phi_{n,m})$, $B=\underrightarrow{\lim}(B_n,\psi_{n,m})$ be two $A\mathcal{D}$ algebras with real rank zero. If we also have
$$(\underline{\mathrm{K}}(A),\underline{\mathrm{K}}^+(A),\Sigma(A))
\cong
(\underline{\mathrm{K}}(B),\underline{\mathrm{K}}^+(B),\Sigma(B)),$$
then $A$ and $B$ are weakly shape equivalent.
\end{thrm}
\begin{proof}
Denote $\rho:\,\underline{\mathrm{K}}(A)\to\underline{\mathrm{K}}(B)$ the above graded isomorphism. Let $\varrho=\rho^{-1}$, $\phi_n:\,A_n\to A$ and $\psi_n:\,B_n\to B$ be the obvious maps.

We construct a commutative diagram
$$
\xymatrixcolsep{3pc}
\xymatrix{
{\,\,\underline{\mathrm{K}}(A_{r_1})\,\,} \ar[d]_-{\rho_1} \ar[r]^-{\phi_{r_1,r_2*}}
& {\,\,\underline{\mathrm{K}}(A_{r_2})\,\,} \ar[d]_-{\rho_2} \ar[r]^-{}
& {\,\,\cdots\,\,} \ar[r]^-{}
& {\,\,\underline{\mathrm{K}}(A)\,\,}\ar[d]_-{\rho}
 \\
{\,\,\underline{\mathrm{K}}(B_{s_1})\,\,} \ar[r]_-{\psi_{s_1,s_2*}} \ar[ru]^-{\varrho_1}
& {\,\,\underline{\mathrm{K}}(B_{s_2}) \,\,} \ar[r]_-{} \ar[ru]^-{\varrho_2}
& {\,\,\cdots \,\,} \ar[r]_-{}
& {\,\,\underline{\mathrm{K}}(B)\,\,}}
$$
where $\rho_n,\,\varrho_n$ are liftable to $*$-homomorphisms $\xi_n:\,A_{r_n}\to B_{s_n}$
and $\psi_n:\,B_{s_n}\to A_{r_{n+1}}$.
The construction is done inductively. We may assume that $A_{r_1}=B_{s_1}={0}$ hence take $\rho=\varrho=0$.
Assume now that $\rho_i$ and $\varrho_i$ have been constructed for all $i\leq n-1$.
For $C^*$-algebra $A_{r_n}$, by Corollary \ref{deccor}, there is a integer $m$,
such that the homomorphism $\phi_{r_n,m}$: $A_{r_n}\to A_{m}$, satisfies the conditions (1) and (2) in the corollary.
let $F\subset \underline{\mathrm{K}}^+(A_{m})$ be provided by Theorem \ref{d_0 F lift}.
Since , by Proposition \ref{dg finite}, the $\Lambda$-module $\underline{\mathrm{K}}(A_m)$ is finitely generated,
there is $k\geq s_{n-1}$ and there is a $\xi\in{\bf Hom}_\Lambda(\underline{\mathrm{K}}(A_m),\underline{\mathrm{K}}(B_k))$ such that
$$
\psi_{k*}\xi=\rho\phi_{m*},\quad
\xi\varrho_{n-1}=\psi_{s_{n-1},k *},\quad
\xi(F)\subset\underline{\mathrm{K}}^+(B_k)
$$
and
$$
\xi[1_{A_{m}}]\leq[1_{B_k}].
$$
Then by Theorem \ref{composed existence} and Remark \ref{homomorphismunital},
$\mathrm{KK}(\phi_{r_n,m})\times\xi$ can be lifted as a $*$-homomorphism from $A_{r_n}$ to $B_{s_n}$.
We conclude the construction by setting $k=s_n$ and $\rho_n=\mathrm{KK}(\phi_{r_n,m})\times\xi$.
It is clear that $\rho_n\varrho_{n-1}=\psi_{s_n,s_{n+1}*}$, $\rho\phi_{r_n *}=\psi_{s_n *}\rho_n$.
Let $\xi_n:\,A_{r_n}\to B_{s_n}$ be a $*$-homomorphism implementing $\rho_n$.
The construction of $\varrho_n$ is similar.
Recall that by Theorem \ref{UCT} we identify ${\bf Hom}_\Lambda(\underline{\mathrm{K}}(A_i),\underline{\mathrm{K}}(B_j))$
with $\mathrm{KK}(A_i,B_j)$.
\end{proof}

 \section{Uniqueness results}

 In this section, we prove two kinds of uniqueness theorem under different conditions.
 One condition is that the $K_1$ of the basic block is $\mathbb{Z}$, and the other is that the basic block has torsion $K_1$ (this case contains $K_1=0$). The main results  are Theorem \ref{freeuqe} and Theorem \ref{toruqe}.
\begin{lem} \label{unicase}
  Let $A\in\mathcal{D}$ be minimal with $K_1(A)=\mathbb{Z}$, $F\subset A$ be a finite set and $\varepsilon>0$. Let $\phi,\psi:A\rightarrow B$ be homomorphisms with finite dimensional ranges and $[\phi(e)]\geq [\psi(e)]$ holds for all projection $e\in A$, then there exist a unitary $u\in B$ and a homomorphism $\tau$ which can be factored through $F_1$ such that
  $$
  \|u\phi(f)u^*-\psi(f)\oplus\tau(f)\|<\varepsilon+2\,\omega(F),\quad\forall\, f\in F.
  $$
  \end{lem}
\begin{proof}
  There exist unitaries $v_1,v_2\in B$ such that $\phi,\psi$ are of the following form
  $$
  \phi(f,a)=v_1^*{\rm diag}\big( a(\theta_1)^{\sim t_1},\cdots,a(\theta_p)^{\sim t_p},f(x_1),f(x_2),\cdots,f(x_{\bullet})\big)v_1,
  $$
  $$
  \psi(f,a)=v_1^*{\rm diag}\big( a(\theta_1)^{\sim s_1},\cdots,a(\theta_p)^{\sim s_p},f(y_1),f(y_2),\cdots,f(y_{\bullet\bullet})\big)v_2,
  $$
  where $x_1,\cdots, x_{\bullet},y_1,\cdots,y_{\bullet\bullet}\in (0,1)\subset Sp(A)$.

  Now we change all the points $x_1,\cdots,x_{\bullet},y_1,\cdots,y_{\bullet\bullet}$ to $0$, then we obtain
  two homomorphisms $\phi',\psi'$ with finite dimensional ranges and factor through $F_1$, obviously, $\phi'$ is homotopy
  to $\phi$ and $\psi'$ is homotopy to $\psi$, by the definition of $\omega(F)$, we have
  $$
  \|\phi(f)-w_1^*\phi'(f)w_1\|\leq\omega(F),\quad \|\psi(f)-w_2^*\psi'(f)w_2\|\leq\omega(F),\,\,\forall\, f\in F.
  $$

  Since $\phi'$ and $\psi'$ factor through $F_1$, there exist homomorphisms $\phi'',\psi'':F_1\rightarrow B$ such that $\phi'=\phi''\circ\pi_e$ and $\psi'=\psi''\circ\pi_e$.

Recall that one has the six-term exact sequence
$$
0\to \mathrm{K}_0(A)\xrightarrow{\pi_*} \mathrm{K}_0(F_1)\xrightarrow{\alpha-\beta} \mathrm{K}_0(F_2)\xrightarrow{\iota_*} \mathrm{K}_1(A)\to 0,
$$
Since $K_0(F_2)=\mathbb{Z}$ and $K_1(A)=\mathbb{Z}$, then $Im(\alpha-\beta)=0$ and $K_0(A)=K_0(F_1)$. Recall that $[\phi(e)]\geq [\psi(e)]$ holds for all projection $e\in A$, then $[\phi''(e)]\geq [\psi''(e)]$ holds for any projection $e\in F_1$, by Lemma \ref{simcase}, there exist a unitary $u\in B$ and a homomorphism $\tau:A\rightarrow B$ which can be factored through $F_1$ such that
$$
\|u\phi(f)u^*-\psi(f)\oplus\tau(f)\|<\varepsilon+2\,\omega(F)$$
holds for all $f\in F$.
\end{proof}
We should point out that if two homomorphisms between Elliott-Thomsen algebras determining
the same $\mathrm{KK}$-class, sometimes they are not homotopic to each other, but after adding another homomorphism, they are homotopic to each other. We present an example here.
\begin{example}\label{stable hom}
Let $F_1=\mathbb{C}\oplus\mathbb{C}$, $F_2=M_2(\mathbb{C})$,
$$\varphi_0(a\oplus b)=\bigg(
\begin{array}{cc}
a& \\[1.5 mm]
 &a
\end{array}\bigg),
\quad
\varphi_1(a\oplus b)=\bigg(
\begin{array}{cc}
b& \\[1.5 mm]
 &a
\end{array}\bigg),
$$
$B=\mathbb{C}$,
and $A=A(F_1,F_2,\varphi_0,\varphi_1)$,
define two homomorphisms $\delta_1$, $\delta_2:\,A\to B$:
$$
\delta_1 (f,a\oplus b)=a,\quad\forall\, (f,a\oplus b)\in A,
$$
and
$$
\delta_2 (f,a\oplus b)=b,\quad\forall\,(f,a\oplus b)\in A.
$$
Then  $\delta_1$, $\delta_2$ induce the two diagrams
$$
\xymatrixcolsep{3pc}
\xymatrix{
{\,\,\mathbb{Z}\oplus\mathbb{Z}\,\,} \ar[d]_-{(1,0)} \ar[r]^-{(1,-1)}
& {\,\,\mathbb{Z}\,\,} \ar[d]^-{0} \\
{\,\,\mathbb{Z} \,\,} \ar[r]_-{0}
& {\,\,0 \,\,}}
$$
and
$$
\xymatrixcolsep{3pc}
\xymatrix{
{\,\,\mathbb{Z}\oplus\mathbb{Z}\,\,} \ar[d]_-{(0,1)} \ar[r]^-{(1,-1)}
& {\,\,\mathbb{Z}\,\,} \ar[d]^-{0} \\
{\,\,\mathbb{Z} \,\,} \ar[r]_-{0}
& {\,\,0 \,\,}.}
$$
At the same time, we have
$$
\delta_1\nsim_h \delta_2
\quad
{\rm{but}}
\quad
\delta_1\oplus\delta_1 \sim_h \delta_2\oplus\delta_1,
$$
and the homotopy path is just
$$
F_s(f,a\oplus b)=f(s),\quad \forall (f,a\oplus b)\in A,~~s\in[0,1].
$$
Denoting $S\delta_1,\,S\delta_2$ by the homomorphisms from suspension algebra $SA$ to suspension algebra $SB$ induced by $\delta_1,\, \delta_2$,
we can also get $S\delta_1\sim_h S\delta_2$.
\end{example}
As shown in the example above, we list the following result.
\begin{lem}[Theorem 1.6 in \cite{AELZ:2019}] \label{kkhom}
  Let $A,B\in\mathcal{D}$ be minimal, $\phi,\psi : A\rightarrow  B $ be two unital homomorphisms with $KK(\phi)=KK(\psi)$, then there exist a positive number $m$ and a homomorphism $\eta:A\rightarrow M_m(B)$ with finite dimensional range such that $\phi\oplus\eta\sim_h\psi\oplus\eta$.
\end{lem}
The following lemmas is a special case of \cite[Theorem 4.8]{AELZ:2019}, for the reader's convenience, we give a short proof here.
\begin{lem}\label{pro h}
  Let $A\in\mathcal{D}$ be minimal with $K_1(A)=\mathbb{Z}$, then for any finite $F \subset A$, $\varepsilon > 0$, there exist $r\in \mathbb{N}$, a homomorphism $\tau :A \rightarrow M_{r-1}(A)$ and a homomorphism $\mu :A \rightarrow M_r(A)$ with finite dimensional image such that
  $$
  \|f\oplus\tau(f)-\mu(f)\|<\varepsilon,\,\,\,\forall f\in F
  $$
\end{lem}
\begin{proof}
  Let $A=A(F_1,M_l(\mathbb{C}),\varphi_0,\varphi_1)$. Since $K_1(A)=\mathbb{Z}$, then $\alpha=\beta$. By \cite[Proposition 3.14]{GLN:2015}, we can assume that $\varphi_0=\varphi_1$. Hence, $f(0)=\varphi_0(a)=\varphi_1(a)=f(1)$ for all $(f,a)\in A$. Let $\sigma: A\rightarrow A$ be defined by $\sigma(f)(t)=f(1-t)$. Then we have
  $$
  id\oplus \sigma\sim_h f(0)\oplus f(1).
  $$

  Let $D$ be a real rank zero $C^*$-algeba constructed as an inductive limit $D=\lim (D_i,\nu_i,j)$, where $D_i=M_{(l+2)^i}(A)$, $i\geq 0$ and $\nu_{i,i+1}(f)=f\oplus \sigma(f)\oplus \mu_i(f)$ for a suitable homomorphism $\mu_i$ with finite dimension range. Since $K_1(D)=0$, it follows from Theorem 3.1 in \cite{ELP2:1999} that the inclusion of $D_0=A$ into $D$ can be approximated arbitrarily well by homomorphisms with finite dimensional range. Then for any $\varepsilon>0$ and finite set $F\subset A$, there exist $i$ and a homomorphism $\mu: A\rightarrow D_i$ with finite dimensional range such that $\|\nu_{0,i}(f)-\mu(f)\|<\varepsilon$ for all $f\in F$. Obviously, $\nu_{0,i}$ is of the form $id\oplus \tau$ for some homomorphism $\tau$. This completes the proof.
\end{proof}
 Then we have the following result from \cite[Lemma 1.4]{D1:1995}(see also \cite[Lemma 1.6.5]{G2:2002}) as a corollary of Lemma \ref{pro h}.
\begin{lem}\label{stabhom}
  Let $A\in \mathcal{D}$ be minimal with $K_1(A)=\mathbb{Z}$. For a finite set $F\subset A$, $\varepsilon>0$ and a positive integer $N$. There are a finite set $G\subset A$, $\delta>0$, and a positive integer $k$ such that the following is true.

  For any block $B\in \mathcal{D}$ and $\phi_0,\phi_1,\cdots,\phi_N$ is a sequence of maps from $A$ to $B$ such that $\phi_j$ is $\delta-$ multiplicative on $G$ for $j=0,1,\cdots,N$, then there exist a homomorphism $\rho:A\rightarrow M_k(B)$ with finite dimensional range and a unitary $u\in M_{k+1}(B)$ such that
  $$
  \|u\big(\phi_0(f)\oplus\rho(f)\big)u^*-\phi_N(f)\oplus\rho(f)\|<\varepsilon+\omega,
  \quad \forall \,f\in F,
  $$
  where
  $$
  \omega=\max_{f\in F}\max_{0\leq j\leq N-1}\|\phi_j(f)-\phi_{j+1}(f)\|.
  $$
\end{lem}
\begin{lem}\label{lasthom}
  Let $A,B\in \mathcal{D}$ be minimal with $K_1(A)=\mathbb{Z}$, $\varepsilon>0$ and $F\subset A$ be a finite subset with $\omega(F)<\varepsilon$. Suppose that $\phi,\psi:A\rightarrow B$ are homomorphisms with the property $KK(\phi)=KK(\psi)$. There exists a positive integer $L$ such that the following is true.

  If $C\in \mathcal{D}$ is minimal, $q\in C$ is a projection, $\lambda:B\rightarrow qCq$ is a homomorphism, $\nu:B\rightarrow (1-q)C(1-q)$ is a homomorphism with finite dimensional range and  with $[\nu(e)]\geq L\cdot[\lambda(1)]$ for any nonzero projection $e\in B$, then there is a unitary $u\in B$ such that
  $$
  \|u\big(\lambda\phi(f)\oplus\nu\phi(f)\big) u^*-\lambda\psi(f)\oplus\nu\psi(f)\|<8\varepsilon,\quad \forall\, f\in F.
  $$
\end{lem}
\begin{proof}

  Since $KK(\phi)=KK(\psi)$, by Lemma \ref{kkhom}, there exist an integer $m$  and a homomorphism $\eta:A\rightarrow M_m(B)$ with finite dimensional image such that $\phi\oplus\eta\sim_h\psi\oplus\eta$. There is a continuous path of homomorphisms $\phi_t$, $(0\leq t\leq 1)$, such that $\phi_0=\phi\oplus\eta$ and $\phi_1=\psi\oplus\eta$. Choose $0=t_0<t_1<\cdots<t_n=1$ such that
  $$
  \|\phi_{t_{j+1}}(f)-\phi_{t_j}(f)\|<\varepsilon,\quad\forall\, f\in F,\,\,\, \forall j\in\{1,2,\cdots,n-1\}.
  $$

  Apply Lemma \ref{stabhom} for homomorphisms $\phi_{t_0},\phi_{t_1},\cdots,\phi_{t_n}$, there exist an integer $k$ and a homomorphism $\rho: A\rightarrow M_{k(m+1)}(B)$ with finite dimensional range such that
  $$
  \|v\big(\phi_{t_0}(f)\oplus \rho(f)\big) v^*-
  \phi_{t_n}(f)\oplus \rho(f)\|<2\varepsilon,\quad \forall\, f\in F.
  $$
  Then we have
  $$
  \|v\big(\phi(f)\oplus\eta(f)\oplus \rho(f)\big)v^*-
  \psi(f)\oplus\eta(f)\oplus \rho(f)\|<2\varepsilon,\quad \forall\, f\in F.
  $$

  Note that for any nonzero projection $e\in A$,
  $$
  [\rho(e)]+[\eta(e)]\leq (km+k+m)\cdot[\phi(1)].
  $$
  Let $L=km+k+m$, this $L$ is as desired.

  Since
  $$
  [\nu\circ \phi(e)]\geq L\cdot[\lambda\circ\phi(1)]\geq [\lambda\circ \eta(e)]+[\lambda\circ\rho(e)].
  $$
  By Lemma \ref{unicase}, there exist homomorphisms $\kappa_1,\kappa_2:A\rightarrow (1-q)C(1-q)$ which can be factored through $F_1$ with $KK(\kappa_1)=KK(\kappa_2)$ and partial isometries $v_1,v_2\in (1-q)C(1-q)$ and  such that
  $$
  \|v_1\big(\nu\phi(f)\big) v_1^*-\lambda\eta(f)\oplus\lambda\rho(f)\oplus\kappa_1(f)\|<3\varepsilon
  $$
  and
  $$
  \|v_2\big(\nu\psi(f)\big) v_2^*-\lambda\eta(f)\oplus\lambda\rho(f)\oplus\kappa_2(f)\|<3\varepsilon
  $$
  hold for all $f\in F$. By Lemma \ref{simcase0}, we have
  $$
  w\kappa_1w^*=\kappa_2
  $$
  for some partial isometry $w\in (1-q)C(1-q)$.

  Since
  $$
  \|\lambda(v)\big(\lambda\phi(f)\oplus\lambda\eta(f)\oplus
  \lambda\rho(f)\big)\lambda(v^*)-\lambda\psi(f)\oplus\lambda\eta(f)
  \oplus\lambda\rho(f)\|<2\varepsilon
  $$
  holds for all $f\in F$.

  Then we have
  $$
  \|u\big(\lambda\phi(f)\oplus\nu\phi(f)\big) u^*-\lambda\psi(f)\oplus\nu\psi(f)\|<8\varepsilon,\quad \forall\, f\in F.
  $$
  for some unitary $u\in C$.
\end{proof}
  Combine Corollory \ref{deccor} and Lemma \ref{lasthom}, we obtain the following theorem.
\begin{thrm} \label{freeuqe}
  Let $A=\underrightarrow{lim}(A_n,\phi_{n,m})$ be a real rank zero inductive limit of Elliott-Thomsen algebras in $\mathcal{D}$. Suppose that $B\in \mathcal{D}$ is a a minimal block with $K_1(B)=\mathbb{Z}$, $\varepsilon>0$, $F\subset B$ be a finite set with $\omega(F)<\varepsilon$ and $\phi,\psi: B\rightarrow A_n$ be homomorphisms with $KK(\phi)=KK(\psi)$, then there exist an integer $m\geq n$ and a unitary $u\in A_m$ such that
  $$
  \|u\big(\phi_{n,m}\phi(f)\big)u^*-\phi_{n,m}\psi(f)\|<18\varepsilon, \,\,\forall\,f\in F.
  $$
\end{thrm}
 The following lemma is the basic homotopy lemma in \cite[Lemma 6.1]{GLN:2015}.
\begin{lem} \label{homolem}
  Let $A$ be a unital separable $\mathrm{C}^*$-algebra and let $\phi:A\rightarrow M_k$ (for some integer $k\geq 1$) be a unital linear map.
  Suppose that $u\in M_k$ is a unitary such that
  $$
  \phi(f)u=u\phi(f),\quad \forall\,\, f\in A.
  $$
  Then there exists a continuous path of unitaries $\{u_t:t\in [0,1]\}\in M_k$ such that
  $$
  u_0=u,\,\,u_1=1,\,\,u_t\phi(f)=\phi(f)u_t,\quad\forall\,\,f\in A.
  $$
\end{lem}
 We also need the following lemma (see  corollary 8.2.2  and proposition 2.2.9 in \cite{ELP1:1998}).
 \begin{lem} \label{crelem}
   Consider a two-dimensional $NCCW$ complex $A_2$. If every infinitesimal of $K_0(A_2)$ is torsion, then $A_2$ is weakly
   semiprojective with respect to finite-dimensional $\mathrm{C}^*$-algebras.
 \end{lem}

\begin{lem} \label{homocor}
  Let $A\in \mathcal{D}$ be minimal with torsion $K_1$, $F\subset A$ be a finite subset, $\varepsilon>0$, there exist a finite subset $G\subset A$ and $\delta>0$ satisfying the following:
  For a homomorphism $\phi:A\rightarrow M_l(\mathbb{C})$ and a unitary $u\in M_l(\mathbb{C})$ such that
  $$
  \|\phi(g)u-u\phi(g)\|<\delta,\quad\forall g\in G,
  $$
  there exist a unital homomorphism $\psi:A\rightarrow M_l(\mathbb{C})$ and a unitary $v\in M_l(\mathbb{C})$ such that
  $$
  \|v-u\|<\varepsilon,\quad\psi(f)v=v\psi(f),\quad\forall f\in A,
  $$
  $$
  \|\phi(f)-\psi(f)\|<\varepsilon,\quad\forall f\in F.
  $$
  Moreover, there exists a unitary path $u_t$ with $u_0=I$ and $u_1=u$ such that
  $$
  \|\phi(f)-u_t^*\phi(f)u_t\|<4\varepsilon,\quad\forall f\in F.
  $$
\end{lem}
\begin{proof}
  We assume that the lemma is false. There exist a finite set $F\subset A$, $\varepsilon >0$, an increasing sequence of finite subsets $G_n\subset A$ such that $G_n\subset G_{n+1}$ and such that $\bigcup_{n=1}^{+\infty}G_n$ is dense in $A$, a sequence of integers ${k_n}$, a sequence of decreasing   positive numbers $\{\delta_n\}$ with $\sum_{n=1}^{+\infty}\delta_n < +\infty$, a sequence of unitaries $u_n\in M_{k_n}$
  and a sequence of unital homomorphisms $\phi_n:A\rightarrow M_{k_n}(\mathbb{C})$  such that
  $$
  \|\phi_n(g)u_n-u_n\phi_n(g)\|<\delta_n,\quad\forall g\in G_n.
  $$
  and such that
  $$
  \inf_{v,\phi}{\rm max}\{\|u_n-v\|,\,\|\phi_n(f)-\psi(f)\|\,|\,f\in F\}\geq \varepsilon,
  $$
  where $v,\,\phi$ run over all $v\in M_{n}(\mathbb{C})$ and all $\psi:A\rightarrow M_{n}(\mathbb{C})$ with $\psi v=v\psi$.

  Define
  $\widetilde{\phi}_n:A\otimes C(S^1)\rightarrow M_{k_n}(\mathbb{C})$  as follows:
  $$
  \widetilde{\phi}_n(f\otimes a)=\phi_n(f)\cdot a(u_n)
  $$
  where $f\in A$, $a\in\{1_{C(S^1)},z\}$, $z$ is the standard unitary generator of $C(S^1)$.

  Denote
  $$
  \phi=\pi\circ\{\widetilde{\phi}_n\}_{n=1}^{+\infty}:A\otimes C(S^1)\rightarrow \prod_{n=1}^{+\infty}M_{k_n}(\mathbb{C})\xrightarrow{\pi}
  \prod_{n=1}^{+\infty}M_{k_n}(\mathbb{C})/\bigoplus_{n=1}^{+\infty}M_{k_n}(\mathbb{C})
  $$
  In general, $\widetilde{\phi}_n$ is not a homomorphism, but $\phi$ is a homomorphism. Apply Lemma \ref{crelem} for a finite set
  $\mathcal{F}=\{f\otimes a\,\,|\,f\in F, \,a\in \{1_{C(S^1)},z\}\}\subset A\otimes C(S^1)$ and $\varepsilon>0$, there
  exist $M>0$  and a homomorphism $\overline{\phi}:A\otimes C(S^1)\rightarrow \prod_{n=M}^{+\infty}M_{k_n}(\mathbb{C})$ such that
  $$
  \|\pi\circ \overline{\phi}(f)-\phi(f)\|<\varepsilon,\,\,\,\forall\,\,f\in \mathcal{F}.
  $$
  where $\pi:\prod_{n=M}^{+\infty}M_{k_n}(\mathbb{C})\rightarrow
  \prod_{n=1}^{+\infty}M_{k_n}(\mathbb{C})/\bigoplus_{n=1}^{+\infty}M_{k_n}(\mathbb{C})$.

  There exists a sequence of homomorphisms $\overline{\phi}_n:A\otimes C(S^1)\rightarrow M_{k_n}(\mathbb{C})$,
   for $n$ large enough, define $\psi: A\rightarrow M_{k_n}(\mathbb{C})$ by $\psi(f)=\overline{\phi}_n(f\otimes 1)$ and  $v=\overline{\phi}_n(1\otimes z)$,
  then we have
  $$
  \psi(f)v=v\psi(f),\,\,\, \forall\,f\in A,
  $$
  $$
  \|v-u_n\|<\varepsilon,\quad\|\phi_n(f)-\psi(f)\|<\varepsilon,\,\,\forall f\in F.
  $$
  This fact contradicts with the assumption.
  Then for certain homomorphism $\phi$, there exist a unital homomorphism $\psi:A\rightarrow M_l(\mathbb{C})$ and a unitary $v\in M_l(\mathbb{C})$ such that
  $$
  \|v-u\|<\varepsilon,\quad\psi(f)v=v\psi(f),\,\,\forall f\in A,
  $$
  $$
  \|\phi(f)-\psi(f)\|<\varepsilon,\,\,\forall f\in F.
  $$
  Apply Lemma \ref{homolem}, there exists a unitary path $u_t,\,t\in [0,\frac{1}{2}]$ connect $u_0=I$ and $u_{\frac{1}{2}}=v$, such that
  $$
  \psi(f)u_t=u_t\psi(f),\,\,\forall f\in A,\,\,t\in[0,\dfrac{1}{2}].
  $$
  We  can also connect $v$ to $u$ by a unitary path $u_t,\,t\in [\frac{1}{2},1]$ with  $u_{\frac{1}{2}}=v$, $u_1=u$ and
  $\|u_t-u\|< \varepsilon$, for all $t\in [\frac{1}{2},1]$.

  Then we have a unitary path $u_t$ with $u_0=I$ and $u_1=u$ such that
  $$
  \|\phi(f)-u_t^*\phi(f)u_t\|<4\varepsilon,\,\,\forall f\in F.
  $$
\end{proof}
\begin{lem}\label{torlast}
   Let $A\in \mathcal{D}$ be minimal with torsion $K_1$, $F\subset A$ be a finite subset and $\varepsilon>0$, there exist a finite subset $G\subset A$ and $\varepsilon>\delta>0$ satisfying the following: If there exist homomorphisms $\phi_t,\psi_t:A\rightarrow M_l(\mathbb{C}),\,\,t\in[0,1]$ and unitaries $u,v\in M_l(\mathbb{C})$ such that
  $$
  \|\phi_t(g)-\phi_0(g)\|<\delta,\,\,\|\psi_t(g)-\psi_0(g)\|<\delta,
  $$
  $$
  \|\phi_0(g)-v^*\psi_0(g)v\|<\delta,\,\,\|\phi_1(g)-u^*\psi_1(g)u\|<\delta,
  $$
  hold for all $g\in G$, $t\in [0,1]$,
  then there exists a unitary path $u_t$ with $u_0=v$ and $u_1=u$ such that
  $$
  \|\phi_t(f)-u_t^*\psi_t(f)u_t\|<8\varepsilon
  $$
  for all $f\in F$, $t\in [0,1]$.
\end{lem}
\begin{proof}
  Let $G\subset A$ be a finite subset and $4\delta>0$ (in place of $\delta$) as required in Lemma \ref{homocor}, then we have
  $$
  \|\phi_0(g)-v^*u \phi_0(g)u^*v\|<4\delta,\,\,\forall \,\, g\in G.
  $$
  There exists a unitary path $v_t$ with $v_0=I$ and $v_1=u^*v$ such that
  $$
  \|\phi_0(f)-v_t^*\phi_0(f)v_t\|<4\varepsilon,\,\,\forall \,\, f\in F.
  $$

  Since
  $$
  \|\phi_t(f)-\phi_0(f)\|<\varepsilon\quad{\rm and }\quad \|\phi_0(f)-v^*\psi_t(f)v\|<2\varepsilon.
  $$
  Denote $u_t$ by $u_t=vv_t^*$, then $u_0=v$ and $u_1=u$.

  Now we have
  $$
  \|\phi_t(f)-u_t^*\psi_t(f)u_t\|<8\varepsilon,\,\,\forall \,\, f\in F.
  $$
\end{proof}
\begin{thrm}\label{toruqe}
  Let $A\in\mathcal{D}$ be minimal with torsion $K_1$, $F\subset A$ be a finite subset and $\varepsilon>0$. Choose a finite subset $G\subset A$, $\delta>0$ as in Lemma \ref{torlast}.

   Suppose that $B\in\mathcal{D}$ is minimal, $\tau: A\rightarrow B$ is a homomorphism with $\omega(\tau(G))<\delta$. Let $C\in \mathcal{D}$ be minimal and $\phi,\psi:B\rightarrow C$ be homomorphisms with the property that $KK(\phi)=KK(\psi)$, then there exist a unitary $w\in C$ such that
  $$
  \|w\big(\phi\tau(f)\big)w^*-\psi\tau(f)\|<8\varepsilon, \,\,\forall\,f\in F.
  $$
\end{thrm}
\begin{proof}
   Since $\tau(G)\subset B$ is a finite set, there exists an integer $m>0$, such that for any $x,y\in [0,1]$ with $d(x,y)\leq \frac{1}{m}$, then
   $$
   \|\phi_x(g)-\phi_y(g)\|<\delta
   \quad
   {\rm and}
   \quad
   \|\psi_x(g)-\psi_y(g)\|<\delta
   $$
   hold for all $g\in \tau(G)$.

  Divide $[0,1]\subset Sp(C)$ into $m$ subintervals with equal length $\frac{1}{m}$, set $t_0=0$, $t_1=\frac{1}{m}$, $\cdots$, $t_m=\frac{m}{m}=1$.
  Consider each interval $[t_k,t_{k+1}]$, we have
  $$
  \|\phi_t(g)-\phi_{t_k}(g)\|<\delta
  \quad
  {\rm and}
  \quad
  \|\psi_t(g)-\psi_{t_k}(g)\|<\delta
  $$
  hold for any $g\in \tau(G)$, $t\in [t_k,t_{k+1}]$.

  Let $C=C(F_1,F_2,\varphi_0,\varphi_1)$, where $F_1=\bigoplus_{j=1}^s M_{r_j}(\mathbb{C})$, $F_2=M_l(\mathbb{C})$.
  Since $\omega(\tau(G))< \delta$ and $KK(\phi)=KK(\psi)$, then for each $j\in \{1,2,\cdots,s\}$ and $k\in \{1,\cdots,m-1\}$,
  we have
  $$
  KK(\pi_e^j\circ\phi)=KK(\pi_e^j\circ\psi)\quad{\rm and}\quad
  KK(\phi_{t_k})=KK(\psi_{t_k}).
  $$
  Then there exist a sequence of unitaries $v_j\in M_{r_j}(\mathbb{C})$ and  a sequence of unitaries $u_k\in M_{l}(\mathbb{C})$
  such that
  $$
  \|\pi_e^j\circ\phi(g)-v_j^*\big(\pi_e^j\circ\psi(g)\big) v_j\|<\delta.
  $$
  and
  $$
  \|\phi_{t_k}(g)-u_k^*\psi_{t_k}(g)u_k\|<\delta
  $$
  hold for all $g\in \tau(G)$.

  Set
  $$
  u_0=\varphi_0(v_1,v_2,\cdots,v_s),\,\,u_m=\varphi_1(v_1,v_2,\cdots,v_s).
  $$
  Now we have
  $$
  \|\phi_{t_k}(g)-u_k^*\psi_{t_k}(g)u_k\|<\delta
  $$
  and
  $$
  \|\phi_t(g)-\phi_{t_k}(g)\|<\delta,
  \quad
  \|\psi_t(g)-\psi_{t_k}(g)\|<\delta,\quad \forall\,t\in [t_k,t_{k+1}]
  $$
  hold for any $g\in \tau(G)$ and $k\in \{0,1,\cdots,m\}$.

   By Lemma \ref{torlast}, there exists a unitary path $w|_{[t_k,t_{k+1}]}$ with $w_{t_k}=u_k$ and $w_{t_{k+1}}=u_{k+1}$.
  Then we obtain a unitary $w\in C$ such that $w_{t_k}=u_k$ for $k=0,1,\cdots,m$, and
  $$\|w\big(\phi\tau(f)\big) w^*-\psi\tau(f)\|<8\varepsilon, \quad\forall\,f\in F.$$
\end{proof}
\section{Classification}
Combine Theorem \ref{freeuqe} and Theorem \ref{toruqe}, then we have
\begin{thrm}\label{thmuqe}
  Let $A=\underrightarrow{lim}(A_n,\phi_{n,m})$ be a real rank zero inductive limit of Elliott-Thomsen algebras in $\mathcal{D}$.
  Let $\varepsilon>0$, $F\subset A_n$ be a finite set with $\omega(F)<\varepsilon$,
  there exists an integer $m\geq n$ such that the following is true.

   Suppose $r>m$ is an integer and $\phi,\psi:A_m\rightarrow A_r$ be homomorphisms with $KK(\phi)=KK(\psi)$, then there exist an integer $s\geq r$ and a unitary $u\in A_s$ such that
  $$
  \|u\big(\phi_{r,s}\circ\phi\circ\phi_{n,m}(f)\big)\, u^*-\phi_{r,s}\circ\psi\circ\phi_{n,m}(f)\|<18\varepsilon, \,\,\forall\,f\in F.
  $$
\end{thrm}
\begin{proof}
  Let $A_n=\bigoplus_{i=1}^{l_n} A_{[n,i]}$, where $l_n,[n,i]\in \mathbb{N}$, all $A_{[n,i]}$ are minimal blocks.

  Define index sets
  $$
  J_n=\{1,2,\cdots,l_n\},\quad J_n'=\{i\,|\,K_1(A_{[n,i]})=\mathbb{Z}\},\quad J_n''=J_n\backslash J_n'.
  $$
  Define $A_n',A_n''$ as follows:
  $$
  A_n'=\bigoplus_{i\in J_n'} A_{[n,i]},\quad  A_n''=\bigoplus_{i\in J_n''} A_{[n,i]}.
  $$
  Then we have $A_n\cong A_n'\oplus A_n''$.

  Let $F'\subset A_n'$ and $F''\subset A_n''$ denote the components of $F$. For $F''\subset A_n''$ and $\varepsilon>0$, we can find a finite subset $G\subset A_n''$ and $\delta>0$ as in Lemma \ref{torlast}. By Lemma \ref{wv2}, there exist an integer $m>n$ such that $\omega(\phi_{n,m}(G))<\delta$. This $m$ is as desired.

  For any homomorphisms $\phi,\psi:A_m\rightarrow A_r$ with $KK(\phi)=KK(\psi)$, by Lemma \ref{toruqe}, there exist a unitary $u_1\in A_r$ such that
  $$
  \|u_1\big(\phi\circ\phi_{n,m}|_{A_n''}(f)\big) u_1^*-\psi\circ\phi_{n,m}|_{A_n''}(f)\|<8\varepsilon, \,\,\forall\,f\in F.
  $$
  Since $KK(\phi\circ\phi_{n,m}|_{A_n'})=KK(\psi\circ\phi_{n,m}|_{A_n'})$, by Lemma \ref{freeuqe}, there exist an integer $s\geq r$ and a unitay $u_2\in A_s$ such that
  $$
  \|u_2\big(\phi_{r,s}\circ\phi\circ\phi_{n,m}|_{A_n'}(f)\big) u_2^*-\phi_{r,s}\circ\psi\circ\phi_{n,m}|_{A_n'}(f)\|<18\varepsilon, \,\,\forall\,f\in F.
  $$

  Then we have
  $$
  \|u\big(\phi_{r,s}\circ\phi\circ\phi_{n,m}(f)\big) u^*-\phi_{r,s}\circ\psi\circ\phi_{n,m}(f)\|<18\varepsilon, \,\,\forall\,f\in F.
  $$
  for some unitary $u\in A_s$.

\end{proof}
  Using Theorem \ref{wv2} and Theorem \ref{thmuqe}, by the intertwining argument of \cite[2.3,\,2.4]{Ell2:1993}(also see \cite[Theorem 1.10.16]{Lin:2001}), we obtain the following theorem.
\begin{thrm} \label{uqethm0}
  Let $A=\underrightarrow{\lim}(A_n,\phi_{n,m})$ and $B=\underrightarrow{\lim}(B_n,\psi_{n,m})$ be real rank zero inductive limits of Elliott-Thomsen algebras in $\mathcal{D}$, if $(A_n,\phi_{n,m})$ and $(B_n,\phi_{n,m})$ are KK-shape equivalent, then $A\cong B$.
\end{thrm}
 Combine Theorem \ref{uqethm0} with Theorem \ref{exist0}, we have our classification result.
\begin{thrm}
  Let $A=\underrightarrow{\lim}(A_n,\phi_{n,m})$ and $B=\underrightarrow{\lim}(B_n,\psi_{n,m})$ are real rank zero inductive limits of Elliott-Thomsen algebras in $\mathcal{D}$, then $A\cong B$ if and only if
 $$(\underline{\mathrm{K}}(A),\underline{\mathrm{K}}^+(A),\Sigma(A))
\cong
(\underline{\mathrm{K}}(B),\underline{\mathrm{K}}^+(B),\Sigma(B)).$$
\end{thrm}
\section*{Acknowledgement}
The research of Qingnan An and Zhichao Liu were supported by the University of Toronto and NNSF of China (No.:11531003), both the the authors thank the  Fields Institute for their hospitality; the research of Yuanhang Zhang was partly supported by the Natural Science Foundation for Young Scientists of Jilin Province (No.:20190103028JH) and NNSF of China (No.:11601104, 11671167, 11201171).

\end{document}